\documentclass[11pt,leqno]{article}
\usepackage{amsmath, amscd, amsthm, amssymb, graphics, xypic, mathrsfs, setspace, fancyhdr, times, bm, pdfsync, enumitem}
\usepackage[usenames, dvipsnames, svgnames, table]{xcolor}
\usepackage[letterpaper,top=1.05in, bottom=1.05in, left=1.05in, right=1.05in]{geometry}
\usepackage[colorlinks=true,pagebackref=true]{hyperref}
\hypersetup{backref}

\newcommand{\paul}[1]{\textcolor{blue}{#1}} 

\newcommand{\colim}{\operatorname{colim}}

\newcommand{\Spec}{\operatorname{Spec}}

\newcommand{\iso}{\cong}

\newcommand{\isomto}{{\stackrel{\sim}{\;\longrightarrow\;}}}

\newcommand{\sma}{{\scriptstyle{\wedge}\,}}

\newcommand{\Hom}{\operatorname{Hom}}

\renewcommand{\hom}{\operatorname{Hom}}
\newcommand{\Map}{\operatorname{Map}}

\newcommand{\real}{{\mathbb R}}
\newcommand{\PP}{\mathbb{P}}

\newcommand{\cplx}{{\mathbb C}}

\newcommand{\Z}{{\mathbb Z}}
\newcommand{\N}{{\mathbb N}}
\newcommand{\A}{{\mathbb A}}

\newcommand{\aone}{{\mathbb A}^1}
\newcommand{\pone}{{\mathbb P}^1}

\renewcommand{\1}{{\rm 1\hspace*{-0.4ex}%
\rule{0.1ex}{1.52ex}\hspace*{0.2ex}}}

\newcommand{\ga}{{{\mathbb G}_{a}}}
\newcommand{\gm}[1]{{{\mathbb G}_{m}^{#1}}}
\newcommand{\Gm}{{\gm{}}}

\newcommand{\ho}[2][]{\mathscr{H}_{#1}({#2})}

\newcommand{\bpi}{\bm{\pi}}

\newcommand{\Nis}{{\operatorname{Nis}}}

\newcommand{\Sm}{\mathrm{Sm}}

\newcommand{\Spc}{\mathrm{Spc}}

\newcommand{\Top}{\mathrm{Top}}
\newcommand{\K}{{{\mathbf K}}}
\newcommand{\KMW}{\K^{MW}}

\newcommand{\F}{{\mathcal F}}

\newcommand{\holim}{\operatornamewithlimits{holim}}

\newcommand{\ZZ}{\Z}

\newcommand{\Addresses}{{
\bigskip
\footnotesize

A.~Asok, Department of Mathematics, University of Southern California, 3620 S.~Vermont Ave., Los Angeles, CA 90089-2532, United States; E-mail address: asok@usc.edu
\medskip

P.~A.~{\O}stv{\ae}r, Department of Mathematics, University of Oslo, P.~O.~Box 1053 Blindern, 0316 Oslo, Norway; E-mail address: paularne@math.uio.no
}}

\renewcommand{\AA}{\mathbb{A}}
\newcommand{\CC}{\mathbb{C}}
\newcommand{\GG}{\mathbb{G}}
\newcommand{\Sch}{\mathrm{Sch}}

\newcommand{\pr}{\mathrm{pr}}

\newcommand{\ML}{\operatorname{ML}}

\newcommand{\Set}{\operatorname{Set}}

\newcommand{\sSet}{\mathrm{sSet}}

\newcommand{\Spt}{\mathrm{Spt}}

\newcommand{\mcal}[1]{\mathcal{#1}}

\renewcommand{\setminus}{\smallsetminus}

\newcommand{\MZ}{{\mathbf{MZ}}}
\newcommand{\HZ}{{\mathbf{HZ}}}
\newcommand{\sphere}{{\mathbf{1}}}
\newcommand{\KGL}{{\mathbf{KGL}}}
\newcommand{\sHom}{{\mathcal{H}om}}

\newcounter{intro}
\theoremstyle{plain}
\newtheorem{thm}{Theorem}[subsection]

\newtheorem{lem}[thm]{Lemma}
\newtheorem{cor}[thm]{Corollary}
\newtheorem{prop}[thm]{Proposition}
\newtheorem*{claim*}{Claim} 

\newtheorem{question}[thm]{Question}
\newtheorem{problem}[thm]{Problem}

\newtheorem{conj}[thm]{Conjecture}
\newtheorem*{thm*}{Theorem}
\newtheorem*{problem*}{Problem}

\newtheorem{thmintro}{Theorem}

\newtheorem{questionintro}[thmintro]{Question}

\newtheorem{conjintro}[thmintro]{Conjecture}

\theoremstyle{definition}
\newtheorem{defn}[thm]{Definition}

\newtheorem{notation}[thm]{Notation}

\theoremstyle{remark}
\newtheorem{rem}[thm]{Remark}
\newtheorem{remintro}[thmintro]{Remark}

\newtheorem{ex}[thm]{Example}
\newtheorem{exintro}[thmintro]{Example}
\newtheorem{entry}[thm]{}

\numberwithin{equation}{subsection}

\begin{document}
\pagestyle{fancy}
\renewcommand{\sectionmark}[1]{\markright{\thesection\ #1}}
\fancyhead{}
\fancyhead[LO,R]{\bfseries\footnotesize\thepage}
\fancyhead[LE]{\bfseries\footnotesize\rightmark}
\fancyhead[RO]{\bfseries\footnotesize\rightmark}
\chead[]{}
\cfoot[]{}
\setlength{\headheight}{1cm}

\author{
Aravind Asok \and
Paul Arne {\O}stv{\ae}r
 }

\title{{\bf ${\mathbb A}^1$-homotopy theory and contractible varieties: a survey}}
\date{}
\maketitle

\begin{abstract}
We survey some topics in $\aone$-homotopy theory.  Our main goal is to highlight the interplay between $\aone$-homotopy theory and affine algebraic geometry, focusing on the varieties that are ``contractible" from various standpoints.
\end{abstract}

\tableofcontents

\newpage


\section{Introduction: topological and algebro-geometric motivations}
In this note, we want to survey the theory of varieties that are ``weakly contractible" from the standpoint of the Morel--Voevodsky $\aone$-homotopy theory of algebraic varieties.
The current version of this document is based on lectures given by the first author at the Fields Institute workshop entitled ``Group actions, generalized cohomology theories and affine algebraic geometry"
at the University of Ottawa in 2011,
and the Nelder Fellow Lecture Series by the second author in connection with the research school ``Homotopy Theory and Arithmetic Geometry: Motivic and Diophantine Aspects"
at Imperial College London in 2018.
\vspace{0.05in}

As a source of inspiration, we review some aspects of the theory of open contractible manifolds, which has source in one of the first false proofs of the Poincar\'e conjecture, and served as a testing ground for many ideas of classical geometric topology in the 1950s and 60s.  We also review some aspects of the theory of complex algebraic varieties whose associated complex manifolds are contractible.  Along the way we formulate some problems that we feel are interesting.

\subsection{Open contractible manifolds}
Suppose $M^n$ is an $n$-dimensional manifold, which we take to mean either {\em topological}, {\em piecewise-linear (PL)} or {\em smooth} (we will specify).  Say an $n$-manifold $M^n$ is open if it is non-compact and without boundary.  As usual, let $I$ be the unit interval $[0,1]$.  Recall that $M^n$ is called {\em contractible} if it is homotopy equivalent to a point.  Euclidean space $\real^n$ is the primordial example of an open contractible manifold: radially contract every point to the origin.
\vspace{0.1in}

The history of open contractible manifolds is closely intertwined with the history of geometric topology.  To explain this connection, recall that the classical {\em Poincar\'e conjecture} asks whether every closed manifold $M^n$ homotopy equivalent to $S^n$ is actually homeomorphic to $S^n$.  The history of the Poincar\'e conjecture is closely connected with the history of open contractible manifolds.
\vspace{0.1in}

In 1934, J.H.C. Whitehead gave what he thought was a proof of the Poincar\'e conjecture in dimension $3$ \cite{zbMATH03015497}.  In brief, his argument went as follows: start with a homotopy equivalence $f: M^3 \longrightarrow S^3$.  Removing a point produces an open contractible manifold $M^3 \setminus pt$ and continuous map $f: M^3 \setminus pt \to \real^3$.  In essence, Whitehead argued that all open contractible $3$-manifolds are homeomorphic to $\real^3$.  If this was true, then such a homeomorphism would extend by continuity across infinity inducing a homemorphism $M^3 \rightarrow S^3$.  Unfortunately, this proof collapsed soon thereafter as Whitehead constructed an open contractible manifold $\mathcal{W}$ (the so-called Whitehead manifold) that is not homeomorphic to the Euclidean space ${\mathbb R}^3$ \cite{zbMATH03019987}.  Whitehead's example raised the question of characterizing Euclidean space among all open contractible manifolds, perhaps in some homotopic way.
\vspace{0.1in}

Whitehead's construction is delightfully geometric and is given as an open subset of a Euclidean space with closed complement the ``Whitehead continuuum", itself built out of intricately linked tori.  Since the construction of the Whitehead manifold is relatively simple, we give it here.  Consider the Whitehead link in $S^3$; thicken it to obtain two linked solid tori; label the interior solid torus $\hat{T}_1$ and the exterior solid torus $T_0$.  The complement of each $T_i$ in $S^3$ is unknotted.  Since $\hat{T}_1$ is unknotted, the complement $T_1$ of the interior of $\hat{T}_1$ in $S^3$ is another unknotted solid torus that contains $T_0$.  Choose a homeomorphism $h$ of $S^3$ that maps $T_0$ onto $T_1$.  Proceeding inductively, may therefore construct solid tori $T_0 \subset T_1 \subset \cdots$ in $S^3$ by setting $T_{j+1} = h(T_{j})$.  The union $\mathcal{W} := \cup_{i} T_i$ is the required open contractible manifold.  To see that $\mathcal{W}$ is simply connected, one proceeds as follows.  Observe that every closed loop in $M^3$ is contained in some $T^j$ by construction.  Now, observe that every closed loop in $T_0$ may be shrunk to a point (possibly crossing through itself) in $T^1$, and thus every closed loop in $T_j$ may be shrunk to a point in $T_{j+1}$.
\vspace{0.1in}

Whitehead's original proof that $\mathcal{W}$ is not homeomorphic to ${\mathbb R}^3$ is essentially geometric.  However, it was later observed that for any sufficiently large compact subset $K \subset \mathcal{W}$ ($K$ contains $T_0$ suffices), the complement ${\mathcal W} \setminus K$ is not simply connected.  More generally, suppose $M$ is an open contractible manifold.  We may consider the collection of compact sets $K \subset M$ ordered with respect to inclusion.  In good situations, the inverse system $\pi_1(M \setminus K)$ stabilizes and defines an invariant of the homeomorphism type of $M$ called the fundamental group at infinity and denoted $\pi_1^{\infty}(M)$.  Observe that ${\mathbb R}^3$ is simply-connected at infinity, since any compact subset is contained in a sufficiently large ball, whose complement in ${\mathbb R}^3$ is simply connected.
\vspace{0.1in}

While Whitehead's construction lay dormant for some time, open contractible manifolds were studied with great intensity in the late 1950s and 1960s.  Although $\mathcal{W}$ is not homeomorphic to $\real^3$ Glimm and Shapiro pointed out that the cartesian product $\mathcal{W}\times\real^1$ is homeomorphic to $\real^4$ \cite{MR0125573}, i.e., $\mathcal{W}$ does not have the {\em cancellation property} with respect to Cartesian products with Euclidean spaces.  Moreover,  Glimm showed that the self-product $\mathcal{W}\times\mathcal{W}$ is homeomorphic to $\real^6$ \cite{MR0125573}.  These kinds ``non-cancellation" phenomenon become another theme explored in the study of open contractible manifolds.
\vspace{0.1in}

By modifying the construction of the Whitehead continuum, D.R. McMillan constructed infinitely many pairwise non-homeomorphic open contractible $3$-manifolds \cite{McMillan}; each of these $3$-manifolds could be embedded in ${\mathbb R}^3$.  While McMillan's original proof that his manifolds were non-homeomorphic was essentially geometric, it was observed that his examples actually have distinct fundamental groups at infinity.  This naturally raises the question of which groups may appear as a fundamental groups at infinity of open contractible $3$-manifolds.
\vspace{0.1in}

Simultaneously, McMillan showed his examples also had the property that taking a Cartesian product with the real line yielded a manifold homeomorphic to ${\mathbb R}^4$.  In fact, McMillan showed that the product of {\em any} open contractible $3$-manifold $M$ and the real line was homeomorphic to ${\mathbb R}^4$ assuming the $3$-dimensional Poincar\'e conjecture.  Later joint work of J. Kister and McMillan \cite{KisterMcMillan} established existence of open contractible $3$-manifolds that could not be embedded in ${\mathbb R}^3$.  Work of Zeeman and McMillan then showed that the product of any open contractible $3$-manifold and ${\mathbb R}^2$ was homeomorphic to ${\mathbb R}^5$.  From one point of view, the zoo of open contractible manifolds presented some kind of ``lower bound" on the complexity of the Poincar\'e conjecture.
\vspace{0.1in}

The existence of open contractible manifolds of higher dimensions was also studied in conjunction with the higher dimensional Poincar\'e conjecture.
Constructions of open contractible $4$-manifolds were initially given by Mazur and Poenaru \cite{Mazur,Poenaru}, but as with many other constructions in geometric topology, these isolated examples in dimension $4$ did not at first fit into a general picture.  Dimension $\geq 5$ proved itself to be more tractable. Extending the results of McMillan, for any integer $n \geq 5$, Curtis and Kwun \cite{CurtisKwun} constructed uncountably many pairwise non-homeomorphic open contractible $n$-manifolds.
\vspace{0.1in}

The method of Curtis and Kwun in some sense mirrors McMillan's construction, but requires significantly more group theoretic input.  The basic idea is to construct open contractible manifolds with a prescribed (finitely presented) fundamental group at infinity.  If $P$ is a finitely presented group, then one may build many cell complexes $K(P)$ with fundamental group $P$.  If one embeds $K(P)$ in $S^n$ for $n \geq 5$, then the boundary of a regular neighborhood $N$ of $K(P)$ has fundamental group $P$ (this fails for $n = 4$, since the fundamental group of the boundary depends on the embedding).  One says that $K(P)$ is a homologically trivial presentation if $H_i(K(P),\Z)$ is trivial for $i = 1,2$.  In particular, if $P$ admits a homologically trivial presentation then the Hurewicz theorem implies that $K(P)$ is a perfect group.  If $N$ is a regular neighborhood of an embedding $K(P) \hookrightarrow S^n$, $n \geq 5$, then one sees that $S^n \setminus \partial N$ is contractible if $K(P)$ is a homologically trivial presentation, but has fundamental group at infinity $P$.  The output is that one has countably many compact contractible manifolds that act as ``building blocks" for an uncountable collection by taking suitable connected sums. The method of Curtis and Kwun for producing ``building blocks" evidently does not extend to $n = 4$.  Later, Glaser extended Mazur's construction to produce countably many ``building blocks" from which uncountably many pairwise non-homeomorphic open contractible $4$-manifolds could be built by taking suitable connected sums as above \cite{Glaser}.  We summarize these results in the following statement.

\begin{thmintro}
\label{thmintro:uncountableexistence}
In every dimension $\geq 3$, there exists uncountably many pairwise non homeomorphic open contractible $n$-manifolds.
\end{thmintro}

Returning to non-cancellation phenomena, Stallings \cite{Stallings} generalized the results of Zeeman and McMillan by observing:
\begin{enumerate}[noitemsep,topsep=1pt]
\item for any integer $n \geq 5$, ${\mathbb R}^n$ is the unique open contractible $PL$-manifold that is simply-connected at infinity; and
\item if $M^n$ is an open contractible $m$-manifold, then $M^n \times {\mathbb R}$ is always simply-connected at infinity.
\end{enumerate}
These observations tamed the zoo of open contractible manifolds in some sense.  The homotopical characterization of Euclidean space above is sometimes called the open Poincar\'e conjecture.  Stallings' result was later generalized by Siebenmann \cite{Siebenmann} to yield an essentially homotopical characterization of the Euclidean space amongst all open contractible topological $n$-manifolds.  In conjunction with more recent work on the low-dimensional Poincar\'e conjecture, i.e., the celebrated work of Perelman \cite{MR3308754} and Freedman \cite{MR1622913}, Siebenmann's theorem can be extended to the following statement.
\vspace{0.1in}

\begin{thmintro}
\label{thmintro:characterization}
Suppose $n \geq 0$ is an integer.
\begin{enumerate}[noitemsep,topsep=1pt]
\item For $n \leq 2$, Euclidean space ${\mathbb R}^n$ is the unique open contractible $n$-manifold.
\item For $n \geq 3$, Euclidean space ${\mathbb R}^n$ is the unique open contractible $n$-manifold that is simply connected at infinity.
\item If $n \geq 3$, and $M$ is an open contractible $n$-manifold, then $M \times \real$ is homeomorphic to ${\mathbb R}^{n+1}$.
\end{enumerate}
\end{thmintro}

The modern study of open manifolds formalizes the notion of ``behavior at $\infty$" of an open manifold as follows.  If $W$ is a topological space, we write $\dot{W}$ for a $1$-point compactification of $W$.  Write $[0,\infty]$ for the extended real line (i.e., with a limit point at infinity).  The end space $e(W)$ consists of the space of maps of pairs $\omega: ([0,\infty],\{\infty\}) \to (\dot{W},\{ \infty\})$ such that $\omega^{-1}(\infty)=\infty$ \cite[\S1, Definition 1.2]{HughesRanicki}.  The end space may be viewed as a homotopy theoretic model for the behaviour at $\infty$ of $W$.  If $W$ is already compact, then $\dot{W} = W_+$, so the end space is empty.  The end space may not be connected: $e(\real)$ is homotopy equivalent to $S^{0}$ corresponding to the two infinite simple edge paths starting at $0\in\real$.  More generally, $e(\real^{n})\simeq S^{n-1}$.  In that case, we will say that $W$ is simply connected at infinity if $e(W)$ is connected and simply connected.

\subsection{Contractible algebraic varieties}
Cancellation questions like those mentioned above were also studied in algebraic geometry, though in several different contexts.  Perhaps the first such question was posed by Zariski \cite{Segre}: if $k$ is a field, $L$ and $L'$ are finitely generated extensions of $k$ such that the purely transcendental extensions $L(x)$ and $L'(x)$ are isomorphic as extensions of $k$, are $L$ and $L'$ isomorphic?  Phrased in the language of algebraic geometry: if $X$ and $X'$ are irreducible algebraic varieties over a field $k$ such that $X \times {\pone}$ is $k$-birationally equivalent to $X' \times {\pone}$, is $X$ $k$-birational to $X'$?  At some point, in the late 1960s/early 1970s, a closely related ``biregular" form of Zariski's original cancellation question was posed: if $A$ and $B$ are finitely generated $k$-algebras, and $A[x] \cong B[x]$ as $k$-algebras, is $A \cong B$?  Even at the beginning, special attention was paid to the case where $A$ is a polynomial ring.  M.P. Murthy \cite{MR0286801} also asked: if $A$ is a Krull dimension $2$ extension of a field $k$ and $A[x] \cong k[x_1,x_2,x_3]$, then must $A$ be a polynomial ring in $2$-variables?  In the language of algebraic geometry this leads to the following question.

\begin{questionintro}[Biregular cancellation question]
\label{question:biregularcancellation}
If $X$ is a smooth affine variety over a field such that $X \times {\mathbb A}^n \cong {\mathbb A}^{N}$, then is $X$ isomorphic to affine space?
\end{questionintro}

\begin{remintro}
The more general cancellation question will be of interest to us as well, and we refer the reader to Sections~\ref{ss:cancellationproblemsaoneweakequivalences} and \ref{ss:Danielewski} for more discussion and history of these kinds of questions.
The above question is also sometimes called the ``Zariski cancellation problem", but it appears never to have been explicitly posed or considered by Zariski.
\end{remintro}

The situation surrounding Question~\ref{question:biregularcancellation} is quite different depending on the characteristic.  Over fields of positive characteristic, recent work of N. Gupta, building on some old ideas of Asanuma \cite{Asanuma} exploiting interesting pathologies existing only in positive characteristic produced counterexamples to the biregular cancellation question \cite{Gupta3dimcancellation}.  Nevertheless, at the time of writing, the question is still open over fields having characteristic $0$.
\vspace{0.1in}

From the outset, ideas of topology played a role in approaches to the biregular cancellation problem over fields having characteristic $0$.  If $k = \cplx$, then we write $X^{an}$ for the set $X(\cplx)$ equipped with its usual structure of a complex manifold.  More generally, if $k$ is a field for which we can find an embedding $\iota: k \hookrightarrow \cplx$, then we may use $\iota$ to define an associated complex manifold $X^{an}_{\iota}$.  If $k$ is a field that admits a complex embedding, then we will say that $X$ is {\em topologically contractible} if $X^{an}_{\iota}$ is a contractible space for every complex embedding $\iota$ of $k$.  The variety ${\mathbb A}^n_k$ is the primordial example of a topologically contractible variety.  Thus, any complex variety that is stably isomorphic to ${\mathbb A}^n$ is also automatically topologically contractible.  Bearing this in mind, we would like to place the discussion of the biregular cancellation question in context using the theory of open contractible manifolds as a template.
\vspace{0.1in}

Ramanujam writes in \cite{MR0286801} that it was initially hoped that a $2$-dimensional non-singular, affine, rational, topologically contractible complex variety would be isomorphic to ${\mathbb A}^2$.  However, he constructs a counter-example to this hope in his landmark paper--a surface now called the Ramanujam surface in his honor; we briefly recall the construction of the Ramanujam surface.  Ramanujam also gave a topological characterization of the affine plane among complex algebraic varieties: ${\mathbb C}^2$ is the unique non-singular contractible complex surface that is simply connected at infinity.

\begin{exintro}
\label{exintro:ramanujamsurface}
In ${\mathbb P}^{2}$ take a cubic $C$ with a cusp at $q$.  Let $Q$ be an irreducible conic meeting $C$ with multiplicity $5$ at a point $p$ and transversally at a point $r$.  The existence of $Q$ can be deduced from the group law on the non-singular part of $C$, or directly by avoiding that $p$ is a flex-point of $C$.  Let $X$ be the blow-up of ${\mathbb P}^{2}$ at $r$ and let $\tilde{C}$, $\tilde{Q}$ denote the strict transforms of $C$, $Q$, respectively.  The Ramanujam surface is defined by setting
\[
\mathcal{R} = X \setminus \tilde{C} \cup \tilde{Q}.
\]
\end{exintro}

Smooth complex varieties $X$ that are not isomorphic to ${\mathbb A}^{n}$ but for which $X^{an}$ is diffeomorphic to ${\mathbb R}^{2n}$ are known as exotic affine spaces.  Such varieties are automatically topologically contractible;  for a survey we refer the reader to \cite{ZaidenbergSurvey}.  After the work of Ramanujam, many examples of topologically contractible smooth varieties were invented; for example, see the work of tom Dieck and Petrie \cite{MR1064448} for examples given by explicit equations.  In fact, there is a veritable zoo of such examples: there exist arbitrary dimensional moduli of topologically contractible smooth complex affine surfaces!
\vspace{0.1in}

The biregular cancellation problem was later solved in the affirmative for surfaces by Miyanishi--Sugie \cite{MiyanishiSugie} and Fujita \cite{FujitaCancellation} in characteristic $0$ and Russell in positive characteristic: the affine plane (over an algebraically closed field) is the unique surface that is stably isomorphic to an affine space (see Section~\ref{ss:cancellation} for precise references).
\vspace{0.1in}

Ramanujam's topological characterization of the affine plane also fails to hold in higher dimensions.  Indeed, Dimca and Ramanujam both proved that if $X$ is a topologically contractible smooth complex affine variety of dimension $d \geq 3$, then $X$ is automatically diffeomorphic to ${\mathbb R}^{2d}$.  Putting these facts together, one sees that there are many exotic affine varieties in dimensions $\geq 3$: simply take the products of topologically contractible surfaces and affine spaces!  This observation, together with the fact that ``topological" invariants do not always have natural counterparts in positive characteristic, suggest that one might hope for a purely algebraic theory of contractible varieties that is more refined than the topological notion discussed here.  Nevertheless, there are a great number of questions about the geometry of topologically contractible varieties that one may pose here, and we highlight several questions that have guided the development of the theory of contractible algebraic varieties.

\begin{questionintro}[Generalized Serre Question]
\label{questionintro:generalizedserrequestion}
If $X$ is a topologically contractible smooth complex affine variety, then is every algebraic vector bundle on $X$ trivial?
\end{questionintro}

\begin{remintro}
This question is called the generalized Serre question in the literature \cite[\S 8]{ZaidenbergSurvey} and we have followed this terminology here. To the best of our knowledge, Serre only ever considered this question for affine space over a field \cite[p. 243]{SerreFAC}, where it spectacularly resolved (independently) by Quillen and Suslin \cite{QuillenProjective,SuslinSerreProblem}; we refer the reader to \cite{LamSerreProblem} for a nice survey of these results.  Of course, it is based on the fact that complex topological vector bundles on contractible varieties are trivial.  We will see later that the affineness assumption in the question is essential; see Example~\ref{ex:nontrivialvectorbundles} for more details.
\end{remintro}

\begin{questionintro}[Generalized van de Ven question]
\label{question:vandeVenquestion}
If $X$ is a topologically contractible smooth complex variety, then is $X$ rational?
\end{questionintro}

\begin{remintro}
This question is sometimes referred to as the van de Ven Question in the literature \cite[Remark 2.2]{ZaidenbergSurvey} with reference given to \cite{vandeVen}.  We call it the generalized van de Ven question because van de Ven was explicitly concerned with surfaces: he writes \cite[\S3 p.197-198]{vandeVen} ``All the known examples of algebraic non-singular compactifications of homology $2$-cells are rational surfaces.  It would follow that there are no others, at least if we restrict ourselves to simply connected homology $2$-cells, if the following statement would be true: $({\mathfrak C})$ Every simply connected non-singular algebraic surface with geometric genus $0$ is rational."  Unfortunately, the conjecture $({\mathfrak C})$ was later disproved: for example the Barlow surfaces \cite{Barlow} are simply connected surfaces of general type with geometric genus $0$.  Nevertheless, Gurjar and Shastri showed \cite{GSI,GSII} that non-singular contractible surfaces are always rational by a careful case-by-case analysis.
\end{remintro}

A theorem of Fujita shows that all topologically contractible smooth complex surfaces are necessarily affine.  A remarkable example of Winkelmann \cite{MR1032948} shows that Fujita's statement does not hold in higher dimensions.  Winkelmann gave the first example of a quasi-affine but not affine smooth contractible complex variety.

\begin{exintro}
\label{exintro:winkelmann}
Winkelmann defines a scheme-theoretically free action of the additive group $\mathbb{G}_a$ on ${\mathbb A}^5$ as follows.  Consider the standard $2$-dimensional representation $V$ of $\mathbb{G}_a$ given by the embedding $\mathbb{G}_a \hookrightarrow SL_2$ as upper triangular matrices.  The $6$-dimensional affine space attached to $V \oplus V \oplus V$ thus carries a linear representation of $\mathbb{G}_a$.  If we pick coordinates $x_0,x_1$, $y_0,y_1$ and $z_0,z_1$ on this affine space, then $x_0$, $y_0$ and $z_0$ are degree $1$ invariant functions, while $x_0y_1 - y_0x_1$, $x_0z_1-z_0x_1$ and $y_0z_1 - z_0y_1$ are degree $2$ invariants.  The hypersurface $x_0 = 1 - y_0z_1 - z_0y_1$ is thus a $\mathbb{G}_a$-invariant smooth hypersurface in ${\mathbb A}^6$ isomorphic to ${\mathbb A}^5$.  The induced action is scheme-theoretically free, and invariant computation just mentioned identifies the quotient as an open subscheme of a smooth affine quadric of dimension $4$ with complement a codimension $2$ affine space.
\end{exintro}

Winkelmann's example highlights another important issue in algebraic geometry that is not visible in the topological story.  Indeed, it follows from the theorems we stated before that every open contractible manifold may be realized as a quotient of ${\mathbb R}^n$ by a free action of the additive group ${\mathbb R}$.   However, every principal ${\mathbb R}$-bundle on a topological space having the homotopy type of a CW-complex is trivial because ${\mathbb R}$ is contractible.  In contrast, in algebraic geometry, a natural analog of free $\real$-actions on ${\mathbb R}^n$ is given by (scheme-theoretically) free actions of the additive group scheme $\mathbb{G}_a$ on ${\mathbb A}^n$.  One must first ask then: given a scheme-theoretically free action of $\mathbb{G}_a$, when does a quotient even exist as a scheme (a quotient always exists as an algebraic space)?  If $X$ is a scheme, then the set of isomorphism classes of $\mathbb{G}_a$-torsors on $X$ is parameterized by the sheaf cohomology group $H^1(X,\mathbb{G}_a) = H^1(X,\mathscr{O}_X)$.  On an affine scheme, this cohomology group necessarily vanishes, but it need not vanish on a quasi-affine scheme that is not affine.
\vspace{0.1in}

Example \ref{exintro:winkelmann} gives the primordial example of a smooth variety that is ``algebraically" homotopy equivalent to affine space, without being isomorphic to affine space.  While there are various ``naive" notions of algebraic homotopy equivalence (e.g., one may think of homotopies parameterized by the affine line), a robust homotopy theory for algebraic varieties in which homotopies are parameterized by the affine line requires more formal machinery.  In the sequel, we begin by trying to highlight the key points of one solution--the construction of the Morel--Voevodsky $\aone$-homotopy \cite{MV}--from the standpoint of the tools required by an end user interested in algebro-geometric questions like those posed above.  Section~\ref{section:aqitmht} contains motivation and the basic ideas of the construction, without going into any of the formal categorical preliminaries.  In Section~\ref{s:concreteaoneweakequivalences}, we try to understand concretely and geometrically how to study isomorphisms in the $\aone$-homotopy category, a.k.a.~$\aone$-weak equivalence.  In the context of the discussion above, we introduce the key notion of an $\aone$-contractible space,  see Definition~\ref{defn:aonecontractible}.  In what remains, we review the theory of $\aone$-contractible smooth varieties, guided by the discussion above, especially the theory of open-contractible manifolds.  Along the way, we discuss an algebro-geometric version of Theorem~\ref{thmintro:uncountableexistence}.

\begin{thmintro}[See Theorems~\ref{thm:familesdimgreater6} and \ref{thm:3dimlfamilies}]
Assume $k$ is a field.  For every pair of integers $d \geq 3$ and $n \geq 0$, there exists a connected $n$-dimensional scheme $S$ and a smooth morphism $\pi: X \to S$ of relative dimension $d$ whose fibers are $\aone$-contractible.  Moreover, the fibers over distinct $k$-points of $\pi$ are pairwise non-isomorphic.
\end{thmintro}

Much of the rest of the discussion can be viewed as providing background for the following questions, which suggest a best-possible approximation to an algebro-geometric variant of Theorem~\ref{thmintro:characterization} (see Definition~\ref{defn:cancellationtypes} for the terminology).

\begin{questionintro}
\label{questionintro:lowdimensions}
If $k$ is a perfect field, is ${\mathbb A}^n$ the only $\aone$-contractible smooth $k$-scheme of dimension $\leq 2$?
\end{questionintro}

\begin{questionintro}
\label{questionintro:highdimensions}
Suppose $X$ is an $\aone$-contractible smooth scheme.
\begin{enumerate}[noitemsep,topsep=1pt]
\item Does there exist an affine space ${\mathbb A}^N$ and a Nisnevich locally trivial smooth morphism ${\mathbb A}^N \to X$?
\item If $X$ is moreover affine, then is $X$ a retract of ${\mathbb A}^N$?
\end{enumerate}
\end{questionintro}

The first question above is discussed in detail in Section~\ref{ss:aonecontractibilityvstopologicalcontractibility}.  The second is motivated by the discussion of Section~\ref{ss:quasiaffineaonecontractibles} in conjunction with the discussion of Sections~\ref{ss:cancellationproblemsaoneweakequivalences}, \ref{ss:KRthreefolds}, \ref{ss:krthreefoldI} and \ref{ss:KRfiberbundles}.  Finally, we suggest that there is a rather subtle relationship between topologically contractible varieties and $\aone$-contractible varieties involving ``suspension" in the $\aone$-homotopy category.

\begin{conjintro}[See Conjecture~\ref{conj:stablecontractibility}]
\label{conjintro:stablecontractibility}
If $X$ is a topologically contractible smooth complex variety with a chosen base-point $x \in X(\cplx)$, then there exists an integer $n \geq 0$ such that $\Sigma^n_{\pone}(X,x)$ is $\aone$-contractible; in fact, $n = 2$ should suffice.
\end{conjintro}

As with any survey, this one reflects the knowledge and biases of the authors.  The literature in affine algebraic geometry on cancellation and related questions is vast, and we can only apologize to those authors whose work we have (inadvertently) failed to appropriately credit.

\section{A user's guide to $\aone$-homotopy theory}
\label{section:aqitmht}
We began by stating a slogan, loosely paraphrasing the first section of \cite{MV}:
\begin{quote}
there should be a homotopy theory for algebraic varieties over a base where the affine line plays the role assigned to the unit interval in topology.
\end{quote}

The path to motivate the construction of the $\aone$-homotopy category that we follow is loosely based on the work of Dugger \cite{DuggerUniversal}.  The original constructions of the $\aone$-homotopy category rely on \cite{JardinePresheaves} and are to be found in \cite{MV}.  General overviews of $\aone$-homotopy theory may be found in \cite{VICM}, and, especially for ``unstable" results in \cite{MICM}.  Morel's paper \cite{MIntro} provides an introductory text, but recent advances like \cite{AHW} and \cite{MR3431674} allow one to get to the heart of the matter much more quickly.  We encourage the reader to consult \cite{MR3727503}, \cite{IO}, \cite{MR3534540}, and \cite{WW} for recent surveys.

\subsection{Brief topological motivation}
The jumping off point for the discussion was the classical Brown-representability theorem in unstable homotopy theory. Recall that the ordinary homotopy category, denoted here ${\mathscr H}$, has as objects ``sufficiently nice" topological spaces $\Top$ (including, for example, all CW complexes), and morphisms given by homotopy classes of continuous maps between spaces.  Unfortunately, the category of CW complexes itself is not ``categorically good" enough to build the homotopy category.  Indeed, there are various constructions one wants to perform in topology (quotients, loop spaces, mapping spaces, suspensions) that do not stay in the category of CW complexes (they only stay in the category up to homotopy).  We will refer to objects of the category $\Top$ (which we have not precisely specified) as {\em spaces}.  As usual, for two spaces $X$ and $Y$, we write $[X,Y]$ for the set of morphisms between $X$ and $Y$ in the homotopy category.
\vspace{0.1in}

Suppose $\mathbf{C}$ is a category of ``algebraic structures".  In practice, one may take $\mathbf{C}$ to be the category of sets or an abelian category like abelian groups, or chain complexes, but we will need some flexibility in the choice.  A $\mathbf{C}$-valued invariant is then a contravariant functor $\F: \Top \to \mathbf{C}$.  We will furthermore consider $\mathbf{C}$-valued invariants on $\Top$ that satisfy the following properties:
\begin{itemize}[noitemsep,topsep=1pt]
\item[i)] (Homotopy invariance axiom) If $X$ is a topological space, and $I = [0,1]$ is the unit interval, then the map $\F(X) \to \F(X \times I)$ is a bijection.
\item[ii)] (Mayer-Vietoris axiom) If $X$ is a CW complex covered by subcomplexes $U$ and $V$ with intersection $U \cap V$, then we have a diagram of the form
    \[
    \F(X) \to \F(U) \times \F(V) \to \F(U \cap V),
    \]
    and given $u \in \F(U)$ and $v \in \F(V)$, such that the images of $u$ and $v$ under the right hand map coincide, then there is an element of $\F(X)$ whose image under the first map is the pair $(u,v)$.
\item[iii)] (Wedge axiom) The functor $\F$ takes sums to products.
\end{itemize}

Given a $\mathbf{C}$-valued invariant $\F$ on $\Top$ as above, the first condition implies that $\F$ factors through a functor $\mathscr{H} \to \mathbf{C}$, i.e., $\F$ is a $\mathbf{C}$-valued homotopy invariant.  There is a very natural class of {\em representable} $\Set$-valued homotopy invariants, given by $[-,Y]$ for some topological space $Y$.  Given a $\mathbf{C}$-valued homotopy invariant $\F$, the classical Brown representability theorem says that if it satisfies the second and third conditions above, then this functor is a representable homotopy invariant, i.e., there is a CW complex $Y$ and an isomorphism of functors $\mathscr{F} \cong [-,Y]$.
\vspace{0.1in}

\begin{rem}
In general, ``cohomology theories" satisfy more properties than just those mentioned: e.g., one has a Mayer--Vietoris long exact sequence.  For example, consider the functor sending a space to its usual integral singular cochain complex $S^*(X)$.  If $X$ is a topological space, then $S^*(X) \to S^*(X \times I)$ is not an isomorphism of chain complexes, but a chain-homotopy equivalence.  Likewise, if $X = U \cup V$, then there is a sequence
\[
S^*(X) \longrightarrow S^*(U) \oplus S^*(V) \longrightarrow S^*(U \cap V),
\]
but this sequence fails to be an exact sequence of chain complexes.  To get a Mayer--Vietoris sequence one must replace $S^*(X)$ by a suitable subcomplex with the same cohomology.  These observations necessitate making various constructions that are ``homotopy invariant" from the very start.
\end{rem}

\begin{rem}
If a ``cohomology theory" is represented on CW complexes by a space $Z$, and the cohomology theory is geometrically defined (e.g., topological K-theory with $Z$ the infinite grassmannian), then the ``representable" cohomology theory extended to all topological spaces need not coincide with the geometric definition for spaces that are not CW complexes.
\end{rem}

\subsection{Homotopy functors in algebraic geometry}
We would like to guess what properties the ``homotopy category" will have based on the known invariance properties of cohomology theories in algebraic geometry.  To this end, we must have some actual ``cohomology theories" at hand.  The examples we want to use are the theory of Chow groups \cite{Fulton} and (higher) algebraic K-theory \cite{Quillen}, though the reader not familiar with general definitions could focus on the Picard group, which is related to both.  We begin by formulating a notion of homotopy invariance in algebraic geometry.  As above, suppose $\mathbf{C}$ is a category of algebraic structures.

\begin{defn}
A $\mathbf{C}$-valued contravariant functor $\F$ on $\Sm_k$ is {\em $\aone$-invariant} if the morphism $\F(U) \to \F(U \times \aone)$ is a bijection.
\end{defn}

\begin{ex}
The functor $Pic(X)$ is not $\aone$-invariant on schemes with singularities that are sufficiently complicated \cite{Traverso}.  More generally, Chow groups are $\aone$-invariant \cite[Theorem 3.3]{Fulton} for regular schemes.  Likewise, Grothendieck established an $\aone$-invariance property for algebraic $K_0$ on regular schemes, and Quillen established a homotopy invariance property for algebraic K-theory of regular schemes (this is one of the fundamental properties of higher algebraic K-theory proven in \cite{Quillen}).
\end{ex}

In the examples above, $\aone$-invariance, failed to hold on all schemes.  As a consequence, we restrict our attention to the category $\Sm_k$ of schemes that are separated, smooth and have finite type over $k$ (we use smooth schemes rather than regular schemes since smoothness is more functorially well-behaved than regularity; if we assume we work with varieties over a perfect field, then there is no need to distinguish between the two notions).  Our restriction to smooth schemes will be analogous to the restriction to (finite) CW complexes performed above.
\vspace{0.1in}

We would like to impose some Mayer--Vietoris-like condition on $\mathbf{C}$-valued invariants.  The most obvious choice would involve the Zariski topology on schemes.  In practice, a number of classical algebro-geometric cohomology theories have a Mayer--Vietoris property for a Grothendieck topology that is finer than the Zariski topology.  Indeed, Chow groups and algebraic K-theory have ``localization" exact sequences; we review such sequences here as a motivation for the finer topology.
\vspace{0.1in}

Let us recall the localization sequence for Chow groups: if $X$ is a smooth variety, and $U \subset X$ is an open subvariety with closed complement $Z$ (say equi-dimensional of codimension $d$), there is an exact sequence of the form
\[
CH^{*-d}(Z) \longrightarrow CH^*(X) \longrightarrow CH^*(U) \longrightarrow 0;
\]
to extend this sequence further to the left, one needs to introduce Bloch higher Chow groups \cite{BlochHigherChow,BlochMoving}, but we avoid discussing this here.  We leave the reader the exercise of showing that from this localization sequence, one may formally deduce that Chow groups and algebraic K-theory have a suitable Mayer--Vietoris property for Zariski open covers by two sets.
\vspace{0.1in}

One often considers the \'etale topology in algebraic geometry, and one might ask whether there is an appropriate Mayer-Vietoris sequence for \'etale covers.  In this direction, consider the following situation.  Suppose given an open immersion $j: U \hookrightarrow X$ and an \'etale morphism $\varphi: V \to X$ such that the pair $(j,\varphi)$ are jointly surjective and such that the induced map $\varphi^{-1}(X \setminus U) \to X \setminus U$ is an isomorphism, diagrammatically this is a picture of the form:
\[
\xymatrix{
U \times_X V \ar[r]^{j'}\ar[d]^{\varphi} & V \ar[d]^{\varphi} \\
U \ar[r]^{j} & X
}
\]
We will refer to such diagrams as {\em Nisnevich distinguished squares}.  One can show that $X$ is the {\em colimit} in the category of smooth schemes of the diagram $U \longleftarrow U \times_X V \longrightarrow V$.
\vspace{0.1in}

By a straight-forward diagram chase, one may show that the sequence
\[
CH^*(X) \to CH^*(U) \oplus CH^*(V) \to CH^*(U \times_X V)
\]
is exact: given an element $(u,v)$ in $CH^*(U) \oplus CH^*(V)$, if the restriction of $(u,v)$ to $CH^*(U \times_X V)$ is zero, then there is an element $x$ in $CH^*(X)$ whose restriction to $CH^*(U) \oplus CH^*(V)$ is $(u,v)$.  Therefore, any theory that satisfies localization will have the Mayer--Vietoris property with respect to a Grothendieck topology on schemes that is finer than the Zariski topology.

\begin{ex}
Suppose $k$ is a field of characteristic unequal to $2$.  Consider the diagram where $X = \aone$, $U = \aone \setminus \{ 1 \}$, $V = \aone \setminus \{0,-1 \}$.  Let $j$ be the usual open immersion of $\aone \setminus \{1 \}$ into $\aone$, and let $\varphi$ be the \'etale map given by the composite $\aone \setminus \{ 0,-1 \} \hookrightarrow \gm{} \to \gm{} \hookrightarrow \aone$, where the map $\gm{} \to \gm{}$ is $z \mapsto z^2$.  It is easily checked that this diagram provides a square as above.
\end{ex}

To define a homotopy category for smooth schemes, informed by the above observations, we will attempt to build a category through which any $\aone$-homotopy invariant on smooth schemes satisfying Mayer--Vietoris in the Nisnevich sense prescribed above factors.  We do this in a few stages.  Just as the category of CW-complexes was not ``categorically good", we will first enlarge the category of schemes to a suitable category of spaces.  Then, we will force Mayer--Vietoris for Nisnevich covers and then impose $\aone$-homotopy invariance.
\vspace{0.1in}

\begin{rem}
As is hopefully evident, we have made a number of choices here: a category of schemes with which to begin, and a topology for which we would like to impose Mayer--Vietoris; we have motivated a particular choice here, but other choices are often warranted.  For example, we might want to make constructions involving non-smooth schemes and be able to compare this situation with the one we alluded to above.  For this reason, we will try to leave some flexibility in the constructions.
\end{rem}

\subsection{The unstable $\aone$-homotopy category: construction}
\subsubsection*{Spaces}
For concreteness, fix a base Noetherian commutative unital ring $k$ of finite Krull dimension.  In practice, $k$ will be a field, but for comparing constructions in different characteristics, it will often be useful for $k$ to be a discrete valuation ring or the integers $\Z$.  Write $\Sm_k$ for the category of schemes that are separated, smooth and have finite type over $k$.  Write $\Sch_k$ for the category of Noetherian $k$-schemes of finite Krull dimension.  Write $\sSet$ for the category of simplicial sets.  There is a functor $\Set \to \sSet$ sending a set $S$ to the corresponding constant simplicial set (i.e., all face and degeneracy maps are the identity), and we use this functor to identify $\Set$ as a full subcategory of $\sSet$.

\begin{defn}
Write $\Spc_k$ for the category of simplicial presheaves on $\Sm_k$, and $\Spc_k'$ for the category of simplicial presheaves on $\Sch_k$; objects of these categories will be called {\em motivic spaces} or $k$-spaces, depending on whether we want to explicitly specify $k$.
\end{defn}

Sending a (smooth) scheme to its corresponding representable presheaf (of constant simplicial sets) defines a functor $\Sch_k \to \Spc_k'$ ($\Sm_k \to \Spc_k$) that is fully-faithful by the Yoneda lemma.  We use these functors without mention to identity (smooth) schemes as spaces.  Likewise, there is a functor $\sSet \to \Spc_k$ sendin a simplicial set to the corresponding constant simplicial presheaf.
\vspace{0.1in}

There are many constructions we may perform in the category of spaces: these categories are both complete and cocomplete, i.e., have limits and colimits indexed by all small categories.  The category $\Spc_k$ (resp. $\Spc_k'$) has a final object typically denote $\ast$ and an inital object $\emptyset$.  Using the fact that $\Spc_k$ has all small limits and colimits, the following definitions make sense.  A pointed space is a pair $(\mathscr{X},x)$ where $\mathscr{X} \in \Spc_k$ and $x: \ast \to \mathscr{X}$ is a morphism of spaces.  We write $\Spc_{k,\bullet}$ for the category of {\em pointed} $k$-spaces.  It is important to emphasize that these constructions are being made in the category of {\em spaces} and NOT in the category of schemes.  Moreover, moving outside of the category of schemes has a number of tangible benefits.

\begin{itemize}[noitemsep,topsep=1pt]
\item If $(\mathscr{X},x)$ and $({\mathscr Y},y)$ are pointed $k$-spaces, then the wedge sum $\mathscr{X} \vee \mathscr{Y}$ is the pushout (colimit) of the diagram $\mathscr{X} \stackrel{x}{\longleftarrow} \ast \stackrel{y}{\longrightarrow} \mathscr{Y}$.
\item If $(\mathscr{X},x)$ and $({\mathscr Y},y)$ are pointed $k$-spaces, then the smash product $\mathscr{X} \wedge \mathscr{Y}$ is the quotient of $\mathscr{X} \times \mathscr{Y}/{\mathscr X} \vee \mathscr{Y}$.
\item We write $S^i$ for the constant presheaf with value $\partial \Delta^{i+1}$, where $\Delta^{i+1}$ is the usual $i+1$-simplex.
\end{itemize}

\subsubsection*{Nisnevich and cdh distinguished squares}
We now make precise various ideas given above motivated by Mayer--Vietoris sequences.

\begin{defn}
\label{defn:Nisnevichsquare}
A Nisnevich distinguished square is a pull-back diagram of schemes
\[
\xymatrix{
W \ar[r]^{j'}\ar[d]^{\varphi} & V \ar[d]^{\varphi} \\
U \ar[r]^{j}& X
}
\]
such that $j: U \hookrightarrow X$ is an open immersion, $\varphi: V \to X$ is an \'etale morphism, and the induced map $\varphi^{-1}(X \setminus U) \to X \setminus U$ is an isomorphism of schemes given the reduced induced scheme structure.
\end{defn}

\begin{rem}
Every Zariski open cover of a scheme by two open sets gives rise to a Nisnevich distinguished square.  Nisnevich distinguished squares generate a Grothendieck topology on $\Sm_k$ or $\Sch_k$ called the Nisnevich topology.  This topology is conveniently described in terms of covering sieves: it is the coarsest topology such that the empty sieve covers $\emptyset$, and for every distinguished square as above, the sieve on $X$ generated by $U \to X$ and $V \to X$ is a covering sieve.  This definition of the Nisnevich topology is equivalent to other standard ones in the literature.  In fact, the Nisnevich topology may be generated by much simpler squares where all corners of the square are affine and the reduced closed complement of $U$ in $X$ is given by a principal ideal.  We refer the interested reader to \cite[\S 2]{AHW} for more details.
\end{rem}

For the most part, we will only use Nisnevich distinguished squares and the Nisnevich topology, but the following definition will also be useful.

\begin{defn}
\label{defn:cdhsquare}
An abstract blow-up square is a pull-back diagram of schemes of the form:
\[
\xymatrix{
E \ar[r]\ar[d] & Y \ar[d]^{\pi} \\
Z \ar[r]^i& X
}
\]
where $i$ is a closed immersion, $\pi$ is a proper map, and the induced map $\pi^{-1}(X \setminus i(Z)) \to X \setminus i(Z)$ is an isomorphism.
\end{defn}

\begin{rem}
As above, the proper cdh topology on $\Sch_k$ is the coarsest topology such that the empty sieve covers $\emptyset$, and for which the sieve on $X$ generated by $Z \to X$ and $Y \to X$ is a covering sieve.
The cdh topology is the smallest Grothendieck topology whose covering morphisms include those of the proper cdh topology and those of the Nisnevich topology.
\end{rem}

\subsubsection*{Localization}
Given the category $\Spc_k$ (resp. $\Spc'_{k}$), we now want to formally invert a set of morphisms $\mathscr{S}$ to build a homotopy category.  Each covering sieve of a Grothendieck topology gives rise to a monomorphism of presheaves with target a representable presheaf; to impose ``Mayer--Vietoris" with respect to a given topology, one first formally inverts the set of all of these monomorphisms.  Concretely, if $u: U \to X$ is a Nisnevich covering, then one may build a simplicial presheaf $\breve{C}(U)$ by taking iterated fiber products of $U$ with itself over $X$.  The morphism $u$ then yields an augmentation
\[
\breve{C}(U) \longrightarrow X
\]
that we would like to force to be a weak equivalence.  The Nisnevich (resp. cdh) local homotopy category is, in essence, the universal category where the above maps have been inverted (see \cite[Lemma 3.1.3]{AHW} for a precise statement).  There are numerous constructions of this category now and the ``universal" point-of-view espoused here was studied in great detail by D. Dugger; see \cite{DuggerUniversal} for more details.

One standard way to invert the relevant set of morphisms is to equip $\Spc_k$ (resp. $\Spc_k'$) with the structure of model category (see, e.g., \cite{Hovey}); this involves specifying classes of cofibrations and fibrations along with the classes of morphisms that are to be inverted, i.e., the morphisms we want to become weak equivalences.  In any case, we will write $\mathrm{H}_{\Nis}(k)$ (resp. $\mathrm{H}_{cdh}(k)$) for the corresponding local homotopy category; we refer the reader to \cite{Jardine} for a textbook treatment of local homotopy theory.  By construction, there is a functor $\Spc_k \to \mathrm{H}_{\Nis}(k)$ (resp. $\Spc_k' \to \mathrm{H}_{cdh}(k)$ that is suitably (i.e., homotopically) initial.

One may build $\aone$-homotopy categories in a similar universal fashion.  After formally inverting Nisnevich (resp. cdh) local weak equivalences, we further invert the projection from the affine line:
\[
\mathscr{X} \times \aone \longrightarrow \mathscr{X}.
\]
We write $\mathrm{H}_{\aone}(k)$ for the corresponding category obtained by localizing $\mathrm{H}_{\Nis}(k)$ and $\mathrm{H}_{\aone}^{cdh}(k)$ for the category obtained by localizing $\mathrm{H}_{cdh}(k)$.



\begin{notation}
An isomorphism in $\mathrm{H}_{\aone}(k)$ or $\mathrm{H}_{\aone}^{cdh}(k)$ will be called an $\aone$-weak equivalence.  To emphasize the analogy with topology, we write $[\mathscr{X},{\mathscr Y}]_{\aone}$ for the set of morphisms in either category; we will read this as the set of $\aone$-homotopy classes of maps between $\mathscr{X}$ and $\mathscr{Y}$.
\end{notation}

\begin{rem}
\label{rem:naivehomotopy}
Given $\mathscr{X}$ and $\mathscr{Y}$ in $\Spc_k$ and a morphism $H: \mathscr{X} \times \aone \to \mathscr{Y}$, we may think of $H$ as an algebraic homotopy between $H(-,0): \mathscr{X} \to \mathscr{Y}$ and $H(-,1): \mathscr{X} \to \mathscr{Y}$; we will say that $H(-,0)$ and $H(-,1)$ are naively homotopic.  More generally, we will say that two morphisms $\mathscr{X} \to \mathscr{Y}$ are {\em naively homotopic} if they are equivalent for the equivalence relation generated by naive homotopies.  Note that naively homotopic maps determine the same $\aone$-homotopy class, i.e., there is an evident function
\[
\hom_{\Spc_k}(\mathscr{X},\mathscr{Y})/\{\text{naive }\aone\text{-homotopies}\} \longrightarrow [\mathscr{X},\mathscr{Y}]_{\aone},
\]
but this morphism is rarely a bijection.
\end{rem}

\subsection{The unstable $\aone$-homotopy category: basic properties}
\subsubsection*{Motivic spheres}
In addition to the simplicial circle $S^1$ described above, we introduce the Tate circle $\gm{}$, viewed as a space pointed by the identity section.  We then define bigraded motivic spheres by the formula
\[
S^{i,j} = S^i \wedge \gm{\sma j}.
\]
From these definitions, and basic homotopy-invariance statements, one may write down some standard algebro-geometric models of motivic spheres.  Homotopy categories typically do not have robust notions of limit or colimit, and one therefore considers a more flexible ``up to homotopy" notion.  The results below are straightforward homotopy colimit computations, and we mention the required facts without going into details.

\begin{ex}
The space $\aone/\{0,1\}$ is a model for $S^1$.  This is an algebro-geometric analog of the fact that the circle $S^1$ can be realized as the homotopy pushout of the two-point space $S^0$ along the maps projecting onto each point.
\end{ex}

One of the basic consequences of our construction of the $\aone$-homotopy category is that any Zariski Mayer--Vietoris square is a pushout square.

\begin{ex}
There is an $\aone$-weak equivalence $S^{1,1} \cong \pone$.  To see this, use the standard Zariski open cover of $\pone$ by two copies of $\aone$.
\end{ex}

\begin{ex}
There is an $\aone$-weak equivalence ${\mathbb A}^n \setminus 0 \cong S^{n-1,n}$.  Use induction and the open cover of ${\mathbb A}^n \setminus 0$ by the open sets $({\mathbb A}^{n-1}\setminus 0) \times \aone$ and $\gm{} \times {\mathbb A}^{n-1}$ with intersection $\gm{} \times ({\mathbb A}^{n-1} \setminus 0)$.
\end{ex}

\begin{ex}
\label{ex:2nnsphere}
There is an $\aone$-weak equivalence ${\mathbb P}^n/{\mathbb P}^{n-1} \cong S^{n,n}$.  Use the open cover of ${\mathbb P}^n$ by ${\mathbb A}^n$ and ${\mathbb P}^n \setminus 0$ with intersection ${\mathbb A}^n \setminus 0$.
\end{ex}

\begin{rem}
With reference to Definition \ref{defn:homotopysheaves} we remark that $S^{i,j}$ is $\aone$-$(i-1)$-connected in the sense that the $\aone$-homotopy sheaf $\bpi_{n}^{\aone}(S^{i,j})$
vanishes for $n\leq i-1$,
see \cite[\S3]{MICM}.
This corresponds precisely to the connectivity of the real points of $S^{i,j}$.
\end{rem}

\subsubsection*{Representability statements}
Just like in topology, there is a Brown representability theorem characterizing homotopy functors in algebraic geometry.  In addition to homotopy invariance,
one wants functors that turn Nisnevich distinguished squares into ``homotopy" fiber products; for more details, see the works of Jardine \cite{JardineRepresentability} and Naumann-Spitzweck \cite{MR2811714}.  One additional subtlety that arises is that many natural functors of algebro-geometric origin that arise fail to be $\aone$-invariant on the category of all smooth schemes.  Thus, one would also like to investigate representability questions for such functors.  We summarize some theorems that will be useful in the sequel.
\vspace{0.1in}

In topology, complex line bundles are represented by homotopy classes of maps to infinite complex projective space, at least on spaces having the homotopy type of a CW complex.  There is an algebro-geometric analog of this result.  For any base commutative unital ring $k$, one may define the space ${\mathbb P}^{\infty}_k$ as the colimit of ${\mathbb P}^n_k$ along the standard closed immersions ${\mathbb P}^n \hookrightarrow {\mathbb P}^{n+1}$.  The space ${\mathbb P}^{\infty}$ can be given a multiplication in the $\aone$-homotopy category.  Indeed, the Segre embeddings ${\mathbb P}^n \times {\mathbb P}^m \longrightarrow {\mathbb P}^{(n+1)(m+1)-1}$ may be used to define a map ${\mathbb P}^{\infty} \times {\mathbb P}^{\infty} \to {\mathbb P}^{\infty}$.  This multiplication map may be shown to be associative up to $\aone$-homotopy and equips $[-,{\mathbb P}^{\infty}]_{\aone}$ with a group structure.  In the algebro-geometric setting, one has a representability theorem for algebraic line bundles.

\begin{prop}[{\cite[\S 4 Proposition 3.8]{MV}}]
\label{prop:picrepresentable}
If $k$ is a regular ring, then for any smooth $k$-scheme $X$, there is an isomorphism of the form
\[
Pic(X) \isomto [X,{\mathbb P}^{\infty}]_{\aone};
\]
this isomorphism is functorial in $X$.
\end{prop}

In topology, complex vector bundles of a given rank are represented by homotopy classes of maps to a suitable infinite Grassmannian, at least on spaces having the homotopy type of a CW complex.  Unfortunately, the functor assigning to a smooth scheme $X$ over a base $k$ the set $\mathscr{V}_r(X)$ of isomorphism classes of rank $r$ vector bundles on $X$ fails to be $\aone$-invariant, as the following example shows.  Thus, this functor cannot be representable on the category of {\em all} smooth schemes.

\begin{ex}
\label{ex:failureofaoneinvariancevb}
Consider ${\mathbb P}^1$.  A classical result of Dedekind--Weber often attributed to Grothendieck asserts that all vector bundles on ${\mathbb P}^1$ are direct sums of line bundles (see \cite{HazewinkelMartin} for an elementary proof).  Consider $X = {\mathbb P}^1 \times \aone$ with coordinates $t$ and $x$.  The matrix
\[
\begin{pmatrix}
t^{-1} & x \\
0 & t
\end{pmatrix}
\]
is the clutching function for a rank $2$ vector bundle on ${\mathbb P}^1 \times \aone$.  The fiber over $x = 0$ of this bundle is isomorphic to $\mathscr{O}(-1) \oplus \mathscr{O}(1)$, while the fiber over $x = 1$ (or any other non-zero value) is isomorphic to $\mathscr{O} \oplus \mathscr{O}$.  This rank $2$ vector bundle is therefore evidently not pulled back from a vector bundle on ${\mathbb P}^1$.
\end{ex}

A classical result of Lindel \cite{Lindel} shows that the functor $\mathscr{V}_r(X)$ is $\aone$-invariant upon restriction to the category of smooth affine $k$-schemes, if $k$ is a field.  This result was extended by Popescu to the case where $k$ is a Dedekind domain with perfect residue fields.  Then, even though $\mathscr{V}_r(X)$ fails to be representable on all smooth schemes, one can hope it is representable on smooth affine schemes.  On smooth affine schemes, given a vector bundle, one may always choose generating sections, i.e., every vector bundle is determined by a morphism to a Grassmannian.  We may define $Gr_n$ as a space to be the colimit of $Gr_{n,N}$ over standard inclusions $Gr_{n,N} \to Gr_{n,N+1}$.

\begin{thm}[Morel, Schlichting, Asok--Hoyois--Wendt]
\label{thm:vbrepresentable}
If $k$ is smooth over a Dedekind ring with perfect residue fields, then for any smooth affine $k$-scheme $X$, there is a pointed bijection
\[
[X,Gr_r]_{\aone} \cong \mathscr{V}_r(X);
\]
this bijection is functorial in $X$.
\end{thm}

\begin{proof}[Comments on the proof.]
This result was stated originally by F. Morel \cite[Theorem 8.1]{MField} in the case where $k$ is perfect and assuming $r \neq 2$.  Morel's argument was greatly simplified by M. Schlichting \cite[Theorem 6.22]{Schlichting}.  The result appears in the form above in \cite[Theorem 1]{AHW}.
\end{proof}

\begin{rem}
While vector bundles of a given rank are only representable on smooth affine schemes, upon passing to stable isomorphism classes, i.e., algebraic K-theory, one may obtain a representability statement on all smooth schemes.  Representability of algebraic K-theory was first established in \cite[Theorem 3.13]{MV}, though we refer the reader to \cite[Remark 2 p. 1162]{SchlichtingTripathi} for some corrections to the original argument.
\end{rem}

\subsubsection*{Representability of Chow groups}
Chow cohomology groups are also representable on smooth schemes, but the representing object is a bit more subtle.  To set the stage recall that ordinary singular cohomology with coefficients in an abelian group $A$ is representable on finite CW-complexes by Eilenberg--Mac Lane spaces $K(A,n)$.  A classical result of Dold and Thom gives a concrete geometric model for the Eilenberg--Mac Lane space $K(\Z,n)$.  Indeed for a pointed topological space $T$, we may define the symmetric product $Sym^n(T)$ as the quotient of the $n$-fold product by the action of the symmetric group on $n$ letters permuting the factors, i.e.,
\[
Sym^n(T) = T^n/\Sigma_n.
\]
Using the base-point, there are natural inclusions $Sym^n T \to Sym^{n+1} T$, and one defines the infinite symmetric product $Sym(T)$ as the colimit of the finite symmetric powers with respect to these inclusions.  The space $Sym(T)$ may be thought of as the free commutative monoid on $T$.  By ``group completing," we may define a space $Sym(T)^+$ that is a topological abelian group.  For connected spaces, the group completion process does not alter the homotopy type, and the classical Dold--Thom theorem shows that for every $n \geq 0$, there are weak equivalences of the form:
\[
K(\Z,n) \cong Sym(S^n),
\]
i.e., Eilenberg--Mac Lane spaces may be realized as infinite symmetric products of spheres.

The procedure sketched above yields a reasonable representing model for Chow groups as well.  For a smooth scheme $X$, define a presheaf $\Z_{tr}(X)$ on $\Sm_k$ by assigning to $U \in \Sm_k$ the free abelian group on irreducible closed subschemes of $U \times X$ that are finite and surjective over a component of $U$.  This construction is covariantly functorial in $X$ as well, and for a closed subscheme $Z \subset X$, we define $\Z_{tr}(X/Z) = \Z_{tr}(X)/Z_{tr}(Z)$, where the latter quotient is the quotient as presheaves of abelian groups.  In particular, we saw earlier that ${\mathbb P}^n/{\mathbb P}^{n-1}$ is a model of the motivic sphere $S^{n,n}$, and we set
\[
K(\Z(n),2n) := \Z_{tr}({\mathbb P}^n/{\mathbb P}^{n-1}).
\]
The spaces $K(\Z(n),2n)$ are usually called {\em motivic Eilenberg--Mac Lane spaces}.  With that definition in mind, we can formulate the appropriate representability theorem.

\begin{thm}[{\cite{Deligne}}]
\label{thm:chowrepresentable}
Assume $k$ is a perfect field.  For every $n \geq 0$, and every smooth $k$-scheme $X$, there are canonical bijections
\[
[X_+,K(\Z(n),2n)]_{\aone} \longrightarrow CH^n(X).
\]
\end{thm}

\begin{rem}
We refer the reader to \cite[\S 2]{VRed} for a convenient summary of properties of motivic cohomology phrased in terms of motivic Eilenberg--Mac Lane spaces.  The relationship between $K(\Z(n),2n)$ and symmetric products, which may appear a bit obscure above stems from the link between symmetric powers and algebraic cycles on quasi-projective varieties; we refer the interested reader to \cite[\S 6.1]{VICM} and the references therein for more details.  This relationship is perhaps most clearly seen in the case $n = 1$ where $CH^1(X) = Pic(X)$.  One may define $Sym^n$ on the category of (say) quasi-projective varieties over a field.  It is well-known that ${\mathbb P}^n \cong Sym^n(\pone)$ as schemes (essentially, this is the map sending a polynomial in $1$ variable to its roots).  Thus, $Sym({\mathbb P}^1) \cong {\mathbb P}^{\infty}$.  Using this identification, one may build a map ${\mathbb P}^{\infty} \to K(\Z(1),2)$; this map is an $\aone$-weak equivalence.
\end{rem}

\subsubsection*{The purity isomorphism}
\begin{defn}
\label{defn:thomspace}
If $X$ is a smooth scheme and $\pi: E \to X$ is a vector bundle with zero section $i: X \to E$, we define $Th(\pi) = E/E - i(X)$.
\end{defn}

This definition of Thom space has many of the same properties as the corresponding construction in topology; we summarize some here.

\begin{prop}
Suppose $X$ is a smooth scheme.
\begin{enumerate}[noitemsep,topsep=1pt]
\item If $\pi: X \times {\mathbb A}^n \to X$ is a trivial bundle of rank $n$, then $Th(\pi) \cong {\pone}^{\sma n} \wedge X_+$.
\item If $\psi: X' \to X$ is a morphism of schemes, $\pi': E' \to X'$ is a vector bundle on $X'$ and $\pi: E \to X$ is a vector bundle on $X$ fitting into a commutative square of the form
\[
\xymatrix{
E' \ar[r]^{\varphi}\ar[d]^{\pi'} & E \ar[d]^{\pi} \\
X' \ar[r]^{\psi} & X,
}
\]
then there is an induced morphism $Th(\pi') \to Th(\pi)$.
\item If $\pi: E \to X$ and $\pi': E' \to X$ are vector bundles over $X$, then $Th(\pi \oplus \pi') \cong Th(\pi) \wedge Th(\pi')$.
\end{enumerate}
\end{prop}

The importance of the notion of Thom space stems from the purity isomorphism, which we summarize in the next result.  For later use, we will need to understand functoriality of purity isomorphism. In order to precisely formulate the functoriality properties, we need to introduce a bit of terminology. Suppose we have a cartesian square of smooth schemes of the form
\[
\xymatrix{
Z' \ar[r]^{i'}\ar[d] & X' \ar[d]^{f} \\
Z \ar[r]^{i} & X
}
\]
where $i$ is a closed immersion.  In that case, $i'$ is also a closed immersion and we will say the square is {\em transversal} if the induced map of normal bundles $\varphi: \nu_{Z'/X'} \to f^* \nu_{Z/X}$ is an isomorphism.  The following result was established in \cite[\S 3 Theorem 2.23]{MV}; the subsequent functoriality statements appear in \cite[\S 2]{VoevodskyMC}.

\begin{thm}[Homotopy purity]
\label{thm:homotopypurity}
Suppose $k$ is a Noetherian ring of finite Krull dimension.
\begin{enumerate}[noitemsep,topsep=1pt]
\item If $i: Z \to X$ is a closed immersion of smooth $k$-schemes, then there is a purity isomorphism
\[
X/X-i(Z) \cong Th(\nu_{i}).
\]
\item If $i: Z \to X$ is the zero section of a geometric vector bundle $\pi: X \to Z$, then the purity isomorphism is the identity map.
\item Given a transversal diagram of smooth schemes as above, the purity isomorphism is functorial in the sense that there is a commutative square of the form
    \[
    \xymatrix{
    X'/(X' \setminus Z') \ar[r]\ar[d] & Th(\nu_{Z'/X'}) \ar[d] \\
    X/(X \setminus Z) \ar[r] & Th(\nu_{Z/X}),
    }
    \]
    where the horizontal maps are the purity isomorphisms, the left vertical map is the map on quotients induced by commutativity of the square and the right vertical map is the map induced by functoriality of Thom spaces.
\end{enumerate}
\end{thm}

\begin{rem}
As will become clear in the sequel, homotopy purity is an absolutely fundamental tool in $\aone$-homotopy theory, especially from the standpoint of computations.  In the context of stable categories to be introduced in Section~\ref{ss:stablestory} it will lead to Gysin exact sequences.
\end{rem}

\subsubsection*{Comparison of Nisnevich and cdh-local $\aone$-weak equivalences}
There is an obvious inclusion $\Sm_k \to \Sch_k$; this yields a functor $\Spc_k' \to \Spc_k$ by restriction.  One may show that there is an induced (derived) ``pullback" functor $\pi^*: \mathrm{H}_{\aone}(k) \to \mathrm{H}_{\aone}^{cdh}(k)$.  For later use, we will need the following comparison theorem of Voevodsky \cite[Theorem 5.1]{Blander} \cite[Theorem 4.2]{VoevodskyNiscdh}.

\begin{thm}[Voevodsky]
\label{thm:voevodskyresolutionofsings}
Assume $k$ is a field having characteristic $0$.  Suppose $f: (\mathscr{X},x) \to (\mathscr{Y},y)$ is a pointed morphism in $\Spc_k$.  If $\pi^*(f)$ is an isomorphism in $\mathrm{H}_{\aone}^{cdh}(k)$, then $\Sigma f$ is an isomorphism in $\mathrm{H}_{\aone}(k)$.
\end{thm}

\subsection{A snapshot of the stable motivic homotopy category}
\label{ss:stablestory}
One of the basic lessons of classical homotopy theory is that calculations become more accessible after inverting the suspension functor on the homotopy category of pointed spaces. The notion of a topological spectrum makes this process precise \cite{MR0402720}.  Similar constructions turn out to be extremely useful in the setting of motivic homotopy theory following \cite{DRO}, \cite{MR1860878}, \cite{JardineMotivic}, \cite{MIntro}, \cite{VICM}.
\vspace{0.1in}

For the purposes of this survey it is useful to know there exists a closed symmetric monoidal category $\mathscr{SH}(k)$ called the stable motivic homotopy category of the field $k$.
We note that $\mathscr{SH}(k)$ is an additive category,
in fact a triangulated category equipped with the auto-equivalence given by smashing with the simplicial circle.
This category is obtained by formally inverting suspension with the projective line $\mathbb{P}^{1}$ on the category $\Spc_{k,\bullet}$ of pointed motivic spaces.
One formally inverts $\pone$ by considering ``spectra".
A motivic spectrum $E\in\Spt_k$ is comprised of pointed motivic spaces $E_{n}$ for all $n\geq 0$ together with structure maps $\mathbb{P}^{1}\wedge E_{n}\to E_{n+1}$.
For example, every $X\in\Sm_k$ has an associated motivic suspension spectrum $\Sigma^\infty_{{\mathbb P}^{1}} X_{+}$ with terms $(\mathbb{P}^{1})^{\wedge n}\wedge X_{+}$ and identity structure maps.
Reminiscent of the way the natural numbers give rise of the integers we are entitled to motivic spheres $S^{p,q}$ in $\mathscr{SH}(k)$ for all $p,q\in\Z$.
In fact,
$\mathscr{SH}(k)$ is generated by shifted motivic suspension spectra of the form $\Sigma^{p,q}\Sigma^\infty_{{\mathbb P}^{1}} X_{+}$.
With a great deal of tenacity one can make precise the statement that $\mathscr{SH}(k)$ is the associated stable homotopy category of $\Spt_k$.
Moreover, one can define a symmetric monoidal structure on $\mathscr{SH}(k)$ for which the sphere spectrum $\sphere=\Sigma^\infty_{{\mathbb P}^{1}} \Spec(k)_{+}$ is the unit.

There are standard Quillen adjunctions, whose left adjoints preserve weak equivalences
\begin{gather*}
\Sigma_{{\mathbb P}^{1}}: \Spc_{k,\bullet}\rightleftarrows \Spc_{k,\bullet}:\Omega_{{\mathbb P}^{1}}\\
\Sigma^\infty_{{\mathbb P}^{1}} : \Spc_{k,\bullet}\rightleftarrows \Spt_k:
\Omega^\infty_{{\mathbb P}^{1}}.
\end{gather*}
Moreover,
we let $\sHom(E,F)$ denote the internal homomorphism object of $\mathscr{SH}(k)$ characterized by the adjunction isomorphism
\[
\Hom_{\mathscr{SH}(k)}(D,\sHom(E,F))
\simeq
\Hom_{\mathscr{SH}(k)}(D\wedge E,F).
\]

Later on in our discussion of $\aone$-contractibility we will appeal to the following result connecting the stable and unstable worlds of motivic homotopy theory.  We include a proof since it illustrates some basic concepts and techniques.

\begin{lem}
\label{lem:compactness}
Let $X$ be a smooth scheme and $x\in X$ a closed point.  If $\Sigma^\infty_{{\mathbb P}^{1}}(X,x)\simeq\ast$ in $\mathscr{SH}(k)$,
then there exists an integer $n\geq 0$ such that $\Sigma^n_{{\mathbb P}^{1}}(X,x)$ is $\A^1$-contractible.
\end{lem}

\begin{proof}
By \cite[Definition 2.10]{DRO}, an object $F\in\Spc_{k,\bullet}$ is fibrant exactly when for every $X\in\Sm_k$,
(1) $F(X)$ is a Kan complex;
(2) the projection $X\times\A^1\to X$ induces a homotopy equivalence $F(X)\simeq F(X\times\A^1)$;
(3) $F$ maps Nisnevich elementary distinguished squares in $\Sm_k$ to homotopy pullback squares of simplicial sets,
and $F(\emptyset)$ is contractible.
Moreover,
a motivic spectrum $E\in\Spt_k$ is fibrant if and only if it is levelwise fibrant and an $\Omega_{{\mathbb P}^{1}}$-spectrum.
\vspace{0.1in}

Let $(E_n)_{n\geq 0}$ be a levelwise fibrant replacement of $\Sigma^\infty_{{\mathbb P}^{1}}(X,x)$, i.e., $E_n$ is a fibrant replacement of $\Sigma^n_{{\mathbb P}^{1}}(X,x)$ in $\Spc_{k,\bullet}$, and let $E$ be a fibrant replacement of $\Sigma^\infty_{{\mathbb P}^{1}}(X,x)$ in $\Spt_k$.
A key observation is that filtered colimits in $\Spc_{k,\bullet}$ preserve fibrant objects; this follows from the above description of fibrant objects and the facts that filtered colimits of simplicial sets preserve Kan complexes, homotopy equivalences, and homotopy pullback squares (\cite[Corollary 2.16]{DRO}).
Putting these facts together, one deduces that there is a simplicial homotopy equivalence
\[
\Omega^\infty_{{\mathbb P}^{1}} E\simeq \colim_{n\to\infty} \Omega^n_{{\mathbb P}^{1}} E_n.
\]

Let $\tilde X\in\Spc_{k,\bullet}$ be the simplicial presheaf $(X,x)\vee \Delta^1$ pointed at the free endpoint of $\Delta^1$;
this is a cofibrant replacement of $(X,x)$ in $\Spc_\bullet(k)$.
Since $\tilde X\in\Spc_{k,\bullet}$ is $\omega$-compact, the following are homotopy equivalences of Kan complexes,
where $\Map$ denotes the simplicial sets of maps in the above simplicial model categories
\begin{align*}
\Map(\Sigma^\infty_{{\mathbb P}^{1}} \tilde X, E)
& \simeq
\Map(\tilde X, \Omega^\infty_{{\mathbb P}^{1}} E) \\
& \simeq
\Map(\tilde X,\colim_{n\to\infty}\Omega^n_{{\mathbb P}^{1}}E_n) \\
& \simeq
\colim_{n\to\infty}\Map(\tilde X,\Omega^n_{{\mathbb P}^{1}}E_n) \\
& \simeq
\colim_{n\to\infty}\Map(\Sigma^n_{{\mathbb P}^{1}}\tilde X, E_n).
\end{align*}
The hypothesis that $\Sigma^\infty_{{\mathbb P}^{1}}(X,x)$ is weakly contractible means that the weak equivalence $\Sigma^\infty_{{\mathbb P}^{1}}\tilde X\stackrel\sim\to E$ and the zero
map $\Sigma^\infty_{{\mathbb P}^{1}}\tilde X\to *\to E$ are in the same connected component of the Kan complex $\Map(\Sigma^\infty_{{\mathbb P}^{1}}\tilde X,E)$.
Since $\pi_0$ preserves filtered colimits of simplicial sets,
there exists an integer $n\geq 0$ such that the weak equivalence $\Sigma^n_{{\mathbb P}^{1}}\tilde X\stackrel\sim\to E_n$ and the zero map $\Sigma^n_{{\mathbb P}^{1}}\tilde X\to *\to E_n$ belong to the
same connected component of $\Map(\Sigma^n_{{\mathbb P}^{1}}\tilde X, E_n)$.
In other words,
$\Sigma^n_{{\mathbb P}^{1}}\tilde X\simeq\Sigma^n_{{\mathbb P}^{1}}(X,x)$ is $\A^1$-contractible.
\end{proof}

\subsubsection*{Milnor--Witt K-theory}
For later reference we recall the definition of Milnor-Witt $K$-theory $\KMW_{*}(k)$ in \cite{MField}.  It is the quotient of the free associative integrally graded ring on the set of symbols $[k^\times] := \{[u]\mid u\in k^\times\}$ in degree $1$ and $\eta$ in degree $-1$ by the homogeneous ideal imposing
the relations
\begin{itemize}[noitemsep,topsep=1pt]
\item[(1)] $[uv] = [u]+[v]+\eta[u][v]$ ($\eta$-twisted logarithm),
\item[(2)] $[u][v] = 0$ for $u+v=1$ (Steinberg relation),
\item[(3)] $[u]\eta=\eta[u]$ (commutativity), and
\item[(4)] $(2+[-1]\eta)\eta = 0$ (hyperbolic relation).
\end{itemize}
Milnor-Witt $K$-theory is $\varepsilon$-commutative for $\varepsilon = -(1+[-1]\eta)$.
By work of Morel there is an isomorphism with the graded ring of endomorphisms of the sphere
\[
\KMW_{*}(k)
\cong
\bigoplus_{n\in\ZZ}\pi_{n,n}\sphere.
\]

Moreover,
we have $\KMW_{0}(k) \cong GW(k)$, the Grothendieck-Witt ring of stable isomorphism classes of symmetric bilinear forms \cite{MR0506372}.
Inverting $\eta$ in $\KMW_{*}(k)$ yields the ring of Laurent polynomials $W(k)[\eta^{\pm 1}]$ over the Witt ring,
and $\KMW_{*}(k)/\eta \cong\K^{M}_*(k)$, the Milnor $K$-theory ring of $k$  \cite{MR0260844}.

\subsubsection*{Stable representablity of algebraic K-theory}
Algebraic $K$-theory is also representable in the stable $\aone$-homotopy category.  To see this, it suffices to consider the infinite projective space ${\mathbb P}^{\infty}$ and a certain ``Bott element" $\beta$ obtained from the virtual vector bundle $[\mathscr{O}]-[\mathscr{O}(-1)]$ over ${\mathbb P}^{1}$.
The precise context involves the stable motivic homotopy category $\mathscr{SH}(k)$;
we replace ${\mathbb P}^{\infty}$ with its motivic suspension spectrum  $\Sigma^{\infty}_{{\mathbb P}^{1}}{\mathbb P}^{\infty}_{+}$ upon which it makes sense to invert $\beta$.

\begin{thm}[Gepner--Snaith, Spitzweck--{\O}stv{\ae}r]
\label{thm:Krepresentable}
If $k$ is a noetherian ring with finite Krull dimension, then there is a natural isomorphism in $\mathscr{SH}(k)$
\[
\Sigma^{\infty}_{{\mathbb P}^{1}}{\mathbb P}^{\infty}_{+}[\beta^{-1}]
\cong
\KGL.
\]
\end{thm}

\begin{proof}[Comments on the proof.]
Here $\KGL$ is the algebraic $K$-theory spectrum introduced by Voevodsky \cite{VICM}.
Independent proofs of this result are given in \cite{MR2540697} and \cite{MR2496504}.
The conclusion holds more generally over any noetherian base scheme of finite Krull dimension.
\end{proof}

\section{Concrete $\aone$-weak equivalences}
\label{s:concreteaoneweakequivalences}
In this section, we attempt to make the discussion of the previous section more concrete.  In particular, we will discuss fundamental examples of isomorphisms in the unstable $\aone$-homotopy category.  Moreover, we recall some results from affine (and quasi-affine) algebraic geometry in the context of $\aone$-homotopy theory.

\subsection{Constructing $\aone$-weak equivalences of smooth schemes}
By construction of the $\aone$-homotopy category, for any smooth scheme $X$ the projection map $X \times \aone \to X$ is an $\aone$-weak equivalence; in particular, the morphism $\aone \to \Spec k$ is an $\aone$-weak equivalence.  By induction, one concludes that ${\mathbb A}^n \to \Spec k$ is an $\aone$-weak equivalence.  In fact, one may give a completely algebraic construction of this $\aone$-weak equivalence using the ideas of Remark~\ref{rem:naivehomotopy}.  Indeed, there is a morphism $\aone \times {\mathbb A}^n \to {\mathbb A}^n$ sending $(t,x_1,\ldots,x_n) \mapsto (tx_1,\ldots, tx_n)$; this corresponds to the usual radial rescaling map.  As in topology, this construction defines a naive $\aone$-homotopy between the identity map ($t = 1$) and the map factoring through the inclusion of $0$ ($t = 1$). In any case, affine space gives the first example of a space satisfying the hypotheses of the following definition.

\begin{defn}
\label{defn:aonecontractible}
A space $\mathscr{X} \in \Spc_k$ is $\aone$-contractible if the structure morphism $\mathscr{X} \to \Spec k$ is an $\aone$-weak equivalence.
\end{defn}

\begin{ex}
Assume $k$ is a field and let $\alpha_{1}$ and $\alpha_{2}$ be coprime integers. The cuspidal curve $\Gamma_{\alpha_{1},\alpha_{2}}=\{y^{\alpha_{1}}-z^{\alpha_{2}}=0\}$ is $\AA^{1}$-contractible.  More precisely, we identify $\Gamma_{\alpha_1,\alpha_2}$ as a motivic space by restricting its functor of points to $\Sm_k$.  Then, the normalization map $\AA^{1}_{x}\to\Gamma_{\alpha_{1},\alpha_{2}}$ given by $x\mapsto (x^{\alpha_{2}},x^{\alpha_{1}})$ is an $\AA^{1}$-weak equivalence, even an isomorphism of presheaves on $\Sm_{k}$ (see \cite[Example 2.1]{ADContractible}).
\end{ex}

Suppose $Z$ is an $\aone$-contractible smooth scheme.  Since we have forced maps that are ``Nisnevich locally" weak equivalences to be weak equivalences, it follows that any map that is Nisnevich locally isomorphic to the product projection $U \times_{\Spec k} Z \to U$ is automatically a weak equivalence.  More precisely, if $f: X \to Y$ is a morphism of smooth $k$-schemes, and there exists a Nisnevich covering map $u: U \to Y$ and an isomorphism of $U$-schemes $X \times_Y U \cong U \times_{\Spec k} Z$, then we will say that $f$ is Nisnevich locally trivial with $\aone$-contractible fibers.

\begin{lem}
\label{lem:nisloctrivaffinefibers}
If $f: X \to Y$ is any morphism of smooth $k$-schemes that is Nisnevich locally trivial with affine space fibers, then $f$ is an $\aone$-weak equivalence.
\end{lem}

Any morphism $f$ of smooth schemes that is Nisnevich locally trivial with affine space fibers is automatically an $\aone$-weak equivalence.  For example, the projection map in a geometric vector bundle is automatically a vector bundle.  More generally, if $\pi: E \to X$ is a torsor under a vector bundle on $X$, then $\pi$ is Zariski locally trivial (this folows from the vanishing of coherent cohomology on affine schemes), and thus an $\aone$-weak equivalence.  Jouanolou originally observed \cite{Jouanolou} that given a quasi-projective variety $X$, one could find a torsor under a vector bundle over $X$ whose total space was affine; such a scheme will be called an {\em affine vector bundle torsor}.  Thomason generalized this observation to schemes admitting an ample family of line bundles; and in the next result we summarize the consequences for $\aone$-homotopy theory.

\begin{lem}[Jouanolou--Thomason homotopy lemma]
\label{lem:jouanolouthomason}
If $k$ is a regular ring, and $X$ is a separated, finite type, smooth $k$-scheme, then there exists a smooth affine $k$-scheme $\tilde{X}$ and morphism $\pi: \tilde{X} \to X$ that is a torsor under a vector bundle; in particular, $\pi:$ is Zariski locally trivial with affine space fibers and is thus an $\aone$-weak equivalence.
\end{lem}

\begin{proof}
If $k$ is regular, then $X$ is a separated, regular, Noetherian scheme.  In that case, $X$ admits an ample family of line bundles (combine \cite[Expose III  Corollaire 2.2.7.1]{SGA6}.  The result then follows from  \cite[Proposition 4.4]{WeibelHomotopy}, a result attributed to Thomason.
\end{proof}

\begin{rem}
If $X$ is a smooth scheme, then a choice of smooth affine scheme $\tilde{X}$ and an $\aone$-weak equivalence $\pi: \tilde{X} \to X$ will be called a Jouanolou device over $X$.  Unfortunately, the construction of Jouanolou devices is not functorial.
\end{rem}

\begin{ex}
\label{ex:standardJouanolou}
For any base ring $k$, following Jouanolou, there is a very simple example of a Jouanolou device over ${\mathbb P}^n$.  Let $\widetilde{{\mathbb P}^n}$ be the complement of the incidence hyperplane in ${\mathbb P}^n \times {\mathbb P}^n$ (viewing the second ${\mathbb P}^n$ as the dual of the first).  It is easy to see that the composite of the inclusion and the projection onto the first factor defines a morphism $\widetilde{{\mathbb P}^n} \to {\mathbb P}^n$ which is a Jouanolou device; we will refer to this as the standard Jouanolou device over ${\mathbb P}^n$.  For $n = 1$ it is straighforward to check that $\widetilde{{\mathbb P}^1}$ is isomorphic to the closed subscheme of ${\mathbb A}^3_k$ defined by the equation $xy = z(1+z)$.
\end{ex}

\subsection{$\aone$-weak equivalences vs. weak equivalences}
\label{ss:realization}
For this section, we will consider rings $k$ that come equipped with a homomorphism $\iota: k \to \cplx$.  In that case, we may compare $\aone$-weak equivalences and classical weak equivalences via what are often called ``realization functors".  Given a smooth $k$-scheme $X$, we may consider the set $X^{an}_{\iota}$ (via $\iota$) and we view this as a complex manifold with its usual structure.  Morel and Voveodsky \cite[\S 3.3]{MV} show that the assignment $X \mapsto X^{an}_{\iota}$ may then be extended to a functor between homotopy categories
\[
{\mathfrak R}_{\iota}: \ho{k} \longrightarrow \mathcal{H}
\]
which we refer to as a (topological) realization functor; see \cite{DuggerIsaksen} for more discussion of topological realization functors.  In particular, applying ${\mathfrak R}_{\iota}$ to an $\aone$-weak equivalence of smooth schemes yields a weak equivalence of the associated topological spaces of complex points.

\begin{rem}
The choice $\iota$ is important: Serre showed that it is possible to find smooth algebraic varieties over a number field together with two embeddings of $k$ into $\cplx$ such that the resulting complex manifolds are homotopy inequivalent.  In fact, F. Charles provided examples of two smooth algebraic varieties over a number field $k$ together with two embeddings of $k$ into $\cplx$ such that the real cohomology algebras of the resulting complex manifolds are not isomorphic \cite{Charles}.  Said differently, the real homotopy type of a smooth $k$-scheme may depend on the choice of embedding of $k$ into $\cplx$.
\end{rem}

\begin{question}
Assume $\iota_1,\iota_2: k \to \cplx$ are distinct ring homomorphisms.  Is it possible to find a (smooth) $k$-scheme $X$ such that $\mathfrak{R}_{\iota_1}(X)$ is contractible while $\mathfrak{R}_{\iota_2}(X)$ is not?
\end{question}

\begin{rem}
Recall that a connected topological space $X$ is $\Z$-acyclic if $H_{i}(X,\Z) = 0$ for all $i > 0$.  Of course, contractible topological spaces are $\Z$-acyclic.  One can show that the property of being $\Z$-acyclic is independent of the choice of embedding for smooth varieties using \'etale cohomology as follows.  If $X$ is a smooth $k$-scheme, then the integral singular cohomology groups $H_i(X^{an},\Z)$ are finitely generated abelian groups.  By the Artin-Grothendieck comparison theorem, the cohomology of $X(\cplx)$ with $\Z/n$-coefficients is isomorphic to the \'etale cohomology of $X$ with $\Z/n$-coefficients, and \'etale cohomology of $X$ is independent of the choice of embedding of $k$ into $\cplx$.  By appeal to the universal coefficient theorem, one deduces that the vanishing of $H_i(X^{an},\Z)_{tors}$ for one embedding implies the vanishing for any other.  Similarly, the rank of the free part is determined by the Betti numbers, which are also determined by \'etale cohomology and are therefore also independent of the choice of embedding.
\vspace{0.1in}

To our knowledge, all the examples where homotopy types change with the embedding involve a nontrivial fundamental group.  If $X$ is a topologically contractible smooth $k$-scheme, then its \'etale fundamental group is independent of the choice of embedding.  Furthermore, the \'etale fundamental group of $X$ is the profinite completion of the topological fundamental group of $X(\cplx)$.  Thus, if $X$ is topologically contractible, then any of the manifolds $X(\cplx)$ has a fundamental group with trivial profinite completion.  If one could prove that $X(\cplx)$ has trivial fundamental group for any choice of embedding, the above problem would have a positive solution as a consequence of the usual Whitehead theorem.  Let us also note that, working with \'etale homotopy types, one can deduce restrictions on the profinite completions of the other homotopy groups of $X(\cplx)$.
\end{rem}

Using the realization functor mentioned above and the definition of topologically contractible varieties from the introduction, the following result is immediate.

\begin{lem}
\label{lem:aonecontractibleimpliestopologicallycontractible}
If $k$ is a commutative ring, $\iota: k \to \cplx$ is a homomorphism, and $X$ is $\aone$-contractible smooth $k$-scheme, then $\mathfrak{R}_{\iota}(X)$ is topologically contractible.
\end{lem}

\begin{rem}
While complex realization is only available for fields admitting an embedding in $\cplx$, there are other realization functors that may be defined more generally.  For example, one may define an \'etale realization functor on the unstable ${\mathbb A}^1$-homotopy category over any field \cite{Isaksen}. Over fields having characteristic $p$, the $p$-part of the \'etale homotopy type is not $\aone$-invariant in general.  For example, the affine line over a separably closed field having positive characteristic has a large nontrivial \'etale fundamental group.  Thus, \'etale realization of an unstable $\aone$-homotopy type involves completion away from the residue characteristics of whatever base ring we work over.  Correspondingly, the analog of Lemma~\ref{lem:aonecontractibleimpliestopologicallycontractible} says that the \'etale realization of an $\aone$-contractible smooth scheme is only trivial after a suitable completion.  On the other hand, the true ``topological" analog of contractibility, i.e., contractibility in the \'etale sense {\em including} triviality of the $p$-part is extremely restrictive: in fact, there are no nontrivial \'etale contractible varieties \cite{HSS}.
\end{rem}

\subsection{Cancellation questions and $\aone$-weak equivalences}
\label{ss:cancellationproblemsaoneweakequivalences}
We now connect the discussion to the cancellation questions mentioned in the introduction: from the standpoint of $\aone$-homotopy theory, this can be viewed as a source of many interesting $\aone$-weak equivalences.  Before discussing the biregular cancellation problem, we recall some results about the original (birational) Zariski cancellation problem.  In the special case where $X$ is a projective space, this question can be rephrased as follows: if $X$ is a stably $k$-rational variety, then is it $k$-rational?  The work of Beauville--Colliot-Th{\'e}l{\`e}ne--Sansuc--Swinnerton-Dyer from the early 80s answered Zariski's original question in the negative \cite{BCTSSD}, i.e., even over algebraically closed fields, there exist stably rational, non-rational varieties of dimension $\geq 3$ (examples over algebraically closed fields cannot exist in dimension $\leq 2$ by the classification of non-singular surfaces).  The results following this development introduced a heirarchy of birational cancellation questions for discussion of which we refer the interested reader to \cite[\S 1]{CTS}.   Correspondingly, we introduce a heirarchy of ``biregular" cancellation questions mimicking the birational story.

\begin{defn}
\label{defn:cancellationtypes}
Suppose $X$ and $Y$ are irreducible smooth $k$-schemes of the same dimension.  Say that $X$ and $Y$ are
\begin{enumerate}[noitemsep,topsep=1pt]
\item {\em stably isomorphic} if there exist an integer $n \geq 0$ such that $X \times {\mathbb A}^n \cong Y \times {\mathbb A}^n$;
\item {\em common direct factors} if there exists a smooth variety $Z$ such that $X \times Z \cong Y \times Z$; and
\item {\em common retracts} if there exists a smooth variety $Z$ and closed immersions $X \to Z$ and $Y \to Z$ admitting retractions.
\end{enumerate}
\end{defn}

\begin{rem}
If $X$ and $Y$ are stably isomorphic, one may ask for the smallest value of $m$ for which $X \times {\mathbb A}^m \cong Y \times {\mathbb A}^m$; so it makes sense to refine stable isomorphisms and inquire about $m$-stable isomorphisms.
\end{rem}

From the standpoint of $\aone$-homotopy theory, this definition is important because of the following result.

\begin{lem}
Stably isomorphic smooth varieties are $\aone$-weakly equivalent.
\end{lem}

Perhaps the original cancellation question, which was explicitly stated by Coleman and Enochs \cite{ColemanEnochs}, asked whether $1$-stably isomorphic affine varieties are isomorphic.  More generally, Abhyanker--Heinzer--Eakin \cite{AHE} asked whether stably isomorphic affine varieties are isomorphic.  In \cite{AHE}, this question is introduced as follows: a ring $A$ is called invariant if given a ring $B$ and an isomorphism $A[x_1,\ldots,x_n] \cong B[y_1,\ldots,y_n]$ it follows that $A \cong B$.  In fact, Abhyankar--Heinzer--Eakin proved that one-dimensional integral domains over a field are invariant \cite[Theorem 3.3]{AHE}, i.e., cancellation holds for irreducible affine curves.  Then, \cite[Question 7.10]{AHE} asks whether two-dimensional integral domains over a field are invariant, with particular attention drawn to the case of the affine plane.  It is, of course, natural to consider the invariance question in higher dimensions as well.
\vspace{0.1in}

The invariance question becomes more subtle as the dimension of varieties under consideration increases.  In the early 1970s, Hochster gave the first counter-example to a cancellation problem over the real numbers \cite{Hochster}, and similar examples were observed by M.P. Murthy (unpublished). In the mid 1970s, Iitaka and Fujita gave geometric conditions (non-negativity of the so-called logarithmic Kodaira dimension, an invariant taking values among $-\infty$, $0$, $1$, $2$, $\dots$) under which affine varieties that are common direct factors are isomorphic \cite[Theorem 1]{IitakaFujita}.  From the point of $\aone$-homotopy theory, what is more interesting is counter-examples to cancellation questions involving smooth varieties.  In the late 80s, Danielewski gave a rather definitive counter-example to the invariance question for smooth affine surfaces: he wrote down smooth affine surfaces depending on a positive integer $n$, such that any two varieties in the class were stably isomorphic, and showed that the relevant varieties could be non-isomorphic for different values of $n$; we now discuss these varieties in detail.

\subsection{Danielewski surfaces and generalizations}
\label{ss:Danielewski}
\begin{defn}
\label{defn:danielewskisurface}
Fix a polynomial $P(z)$ in one variable and an integer $n \geq 1$.  The Danielewski surface $D_{n,P}$ is the closed subscheme of ${\mathbb A}^3$ defined by the equation $x^ny = P(z)$.
\end{defn}

\begin{ex}
When $n = 1$ and $P(z) = z(1+z)$, the variety $D_{n,P}$ is isomorphic to the standard Jouanolou device over ${\mathbb P}^1$ from Example~\ref{ex:standardJouanolou}.  Assuming $2$ is invertible in our base ring, it is also isomorphic to the standard hyperbolic quadric $xy + z^2 = 1$.
\end{ex}

If $P(z)$ is a separable polynomial, then it is straightforward to see that $D_{n,P}$ is smooth over $k$.  In that case, projection in the $x$-direction determines a morphism $D_{n,P} \to \aone_k$.  Assume for simplicity $k$ is an algebraically closed field, so $P(z)$ factors as a product of distinct linear factors; write $z_1,\ldots,z_d$ for the $d := deg P(z)$ distinct roots of $P(z)$.  In that case, the fibers of the projection morphism are isomorphic to $\aone_k$ over non-zero points of $\aone_k$ while the fiber over $0$ consists of $d$ copies of $\aone_k$ defined by $x = 0, z = z_i$.  The complement of all but $d-1$ of these copies of the affine line is an open subscheme of $D_{n,P}$ that is isomorphic to ${\mathbb A}^2$, and the restriction of the projection morphism is a product projection $\aone_k \times \aone_k \to \aone_k$.  Thus, the projection morphism $D_{n,P} \to \aone_k$ factors through a morphism $D_{n,P} \to \aone_{P}$, where $\aone_P$ is a non-separated version of the affine line with a $d$-fold origin.  Danielewski and Fieseler observed that this product projection makes $D_{n,P} \to \aone_P$ into a torsor under a line bundle (see \cite[Proposition 2.6]{DuboulozDF} for various generalizations of this construction).  We summarize the Danielewski construction in the following result, which follows from the discussion above and the fact that torsors under line bundles over affine schemes may be trivialized (i.e., are isomorphic to line bundles).

\begin{prop}[Danielewski, Fieseler, Dubouloz]
Assume $k$ is algebraically closed and $P$ is a separable polynomial.  If $n$ and $n'$ are distinct positive integers, and $P$ is a separable polynomial over $k$, then set:
\[
D_{n,P} \stackrel{p_1}{\longleftarrow} D_{n,P} \times_{\aone_P} D_{n',P} \stackrel{p_2}{\longrightarrow} D_{n',P}.
\]
The morphisms $p_i$ make the fiber product into the projection map for a geometric line bundle.  In particular, $D_{n,P}$ and $D_{n',P}$ are common retracts of $D_{n,P} \times_{\aone_P} D_{n',P}$ (retraction given by the zero section).  If $P = z(1+z)$, then $D_{n,P}$ and $D_{n',P}$ are furthermore stably isomorphic.
\end{prop}

The above observations, coupled with a homotopy colimit argument allow one to describe the $\aone$-homotopy type of $D_{n,P}$ rather explicitly: it is $\aone$-weakly equivalent to a wedge sum of $d-1$-copies of $\pone_k$.  The proposition even gives very explicit $\aone$-weak equivalences for different values of $n$ and fixed $P$.  The isomorphism class of the varieties $D_{n,z(1+z)}$ may be distinguished over the complex numbers by computing their first homology at infinity: explicitly $D_n$ may be realized as the complement of a divisor in a Hirzebruch surface.

\begin{prop}[Fieseler]
If $k = \cplx$, then for any integer $n \geq 2$, $H_1^{\infty}(D_{n,z(1+z)}) \cong \Z/n\Z$.  In particular, if $n$ and $n'$ are distinct integers, $D_{n,z(1+z)}$ and $D_{n',z(1+z)}$ are not isomorphic.
\end{prop}

Danielewski's original construction has been expanded in many directions.  We refer the reader to \cite{DuboulozDF} and \cite{DuboulozDanielewskiVarieties} for more details.  One even knows that there are pairs of topologically contractible smooth affine varieties giving counter-examples to cancellation \cite{DM-JP}.  Cancellation may fail for open subsets of affine space: this was observed for affine spaces of sufficiently large dimension in \cite{JelonekOpen} and for ${\mathbb A}^3$ in \cite{DuboulozOpen}.  Jelonek even observed \cite[Proposition 3.18]{JelonekOpen} that there exist smooth affine varieties that are $2$-stably isomorphic but not stably isomorphic.  Subsequently, lower-dimensional examples of this phenomena (though which fail to be smooth) were constructed in \cite{AsanumaGupta} and $2$-dimensional smooth counterexamples were constructed in \cite{Dubouloza2cylinders}.  Furthermore, cancellation may fail rather generically: for every smooth affine variety of dimension $d \geq 7$, there exists a smooth affine variety $X'$ birationally equivalent to $X$ such that the variety $X'\times {\mathbb A}^2$ is not invariant.  We refer the reader to \cite{RussellSurvey} for a survey of the state of affairs up to about 2014, though the references above should make it clear that many exciting developments have occurred since that time.

\begin{problem}
Develop tools to distinguish isomorphism classes of smooth schemes having a given unstable $\aone$-homotopy type.
\end{problem}

\begin{rem}
In Section~\ref{ss:endspaces}, we will develop some tools to aid in the study of this problem for smooth schemes that are not proper.
\end{rem}

\subsection{Building quasi-affine $\aone$-contractible varieties}
\label{ss:quasiaffineaonecontractibles}
Winkelmann's example \ref{exintro:winkelmann} is realized as a quotient of affine space by a free action of a unipotent group (we review this below in Example~\ref{ex:winkelmann}.  In this section, we discuss some examples that significantly expand on this idea and therefore show that $\aone$-contractible smooth schemes are abundant in nature.

\subsubsection*{Unipotent quotients}
If the additive group scheme $\mathbb{G}_a$ acts scheme-theoretically freely on a smooth scheme $X$ and a quotient exists as a scheme, then the quotient map is automatically Zariski locally trivial because $H^1(-,\mathbb{G}_a)$ vanishes on affine schemes ($\mathbb{G}_a$ is an example of a linear algebraic group that is {\em special} in the sense of Serre).  In that case, the quotient map is automatically an $\aone$-weak equivalence by appeal to Lemma~\ref{lem:nisloctrivaffinefibers}.

\begin{ex}
\label{ex:winkelmann}
Take $k = \Z$, and suppress it from notation.  Let $Q_4$ be the smooth affine quadric in ${\mathbb A}^5$ defined by the equation $x_1x_3 - x_2x_4 = x_5(1+x_5)$.  Let $E_2$ be the closed subscheme defined by the equation $x_1 = x_3 = x_5+1 = 0$; $E_2$ is isomorphic to ${\mathbb A}^2$.  The complement $X_4 := Q_4 \setminus E_2$ is quasi-affine and not affine; it is Winkelmann's quasi-affine quotient from the introduction (i.e., Example~\ref{exintro:winkelmann}) over $\Spec \Z$.  The variety $X_4$ is an $\aone$-contractible smooth scheme by appeal to Lemma~\ref{lem:nisloctrivaffinefibers}.
\end{ex}

Generalizing these observations, in \cite{ADContractible}, the first author and B. Doran showed that many (pairwise non-isomorphic) strictly quasi-affine $\aone$-contractible varieties could be constructed in this way.  In fact, the following result, which is a first step in the direction of Theorem~\ref{thmintro:uncountableexistence}.

\begin{thm}[{\cite[Theorem 1.3]{ADContractible}}]
\label{thm:familesdimgreater6}
Assume $k$ is a field.  For every integer $m \geq 6$ and every integer $n \geq 0$, there exists a connected $n$-dimensional $k$-scheme $S$ and a smooth morphism $\pi: X \to S$ of relative dimension $m$, whose fibers over $k$-points are $\aone$-contractible, quasi-affine, not affine, and pairwise non-isomorphic.
\end{thm}

By Proposition~\ref{prop:picrepresentable}, it follows that if $X$ is any $\aone$-contractible smooth $k$-scheme, then $Pic(X) = Pic(k)$.  In particular, if $Pic(k)$ is trivial (e.g., if $k$ is a field or $\Z$), then every line bundle on $X$ is trivial.  In that case, every torsor under a line bundle on $X$ is automatically a $\mathbb{G}_a$-torsor.  It is known that total spaces of $\mathbb{G}_a$-torsors over schemes may have non-isomorphic total spaces (e.g., consider the Danielewski varieties above).  The following question extends a question posed initially by Kraft \cite[\S 3 Remark 2]{Kraft} in the case $X = X_4$ (since $\ga$-torsors on affine schemes are always trivial, the question is only interesting for quasi-affine $\aone$-contractible varieties).

\begin{question}
Suppose $X$ is an $\aone$-contractible smooth $k$-scheme.  Is it possible to have two $\mathbb{G}_a$-torsors on $X$ with non-isomorphic total spaces?
\end{question}

Winkelmann's example also has interesting consequences for the shape of the generalized Serre question \ref{questionintro:generalizedserrequestion}.  Indeed, one may use it to see that it {\em is} possible to have nontrivial vector bundles on $\aone$-contractible smooth schemes that are not affine.  This phenomenon is analyzed in great detail \cite{ADBundle} as a measure of the failure of $\aone$-invariance of the functor assigning to a scheme $X$ the set of isomorphism classes of rank $r$ vector bundles on $X$ (cf. Example~\ref{ex:failureofaoneinvariancevb}).  It also shows that the affineness assertion in Question~\ref{questionintro:generalizedserrequestion} is absolutely essential.

\begin{ex}
\label{ex:nontrivialvectorbundles}
The variety $X_4$ above carries a nontrivial rank $2$ vector bundle.  The variety $Q_4$ carries a nontrivial rank $2$ vector bundle.  The easiest way to see this is to realize $Q_4$ as $Sp_4/(Sp_2 \times Sp_2)$, i.e., as the quaternionic projective line $\mathbb{HP}^1$ in the sense of \cite{PaninWalter}.  In that case, the map $Sp_4/Sp_2 \to Q_4$ is a nontrivial $Sp_2$-torsor, and the relevant vector bundle is the associated vector bundle to the standard $2$-dimensional representation of this $Sp_2$-torsor.  In fact, the quotient $Sp_4/Sp_2$ and the variety $Q_4$ both have the $\aone$-homotopy type of a motivic sphere, and the relevant morphism is the motivic Hopf map sometimes called $\nu$.  Since the open immersion $j: X_4 \hookrightarrow Q_4$ has closed complement of codimension $2$, the restriction functor $j^*$ on the category of vector bundles is fully-faithful.  Thus, this rank $2$ bundle restricts to a nontrivial rank $2$ vector bundle on $X_4$.  The total space of this rank $2$ vector bundle is another $\aone$-contractible smooth scheme which is necessary non-isomorphic to affine space as it is itself quasi-affine!
\end{ex}

We have seen above that there are quasi-affine $\aone$-contractible smooth schemes of dimension $d \geq 4$.  On the other hand, one may see using classification that the only $\aone$-contractible smooth scheme of dimension $1$ (say over a perfect field) is $\aone$.  It follows from a general result of Fujita \cite[\S 2 Theorem 1]{FujitaAffine} that any topologically contractible smooth complex surface is necessarily affine.  Thus, the following question remains open.

\begin{question}
If $k$ is a field, does there exist an $\aone$-contractible (resp. topologically contractible) smooth $k$-scheme of dimension $3$ that is quasi-affine but not affine?
\end{question}

\subsubsection*{Other quasi-affine $\aone$-contractible varieties}
In \cite{ADContractible}, it was asked whether every quasi-affine $\aone$-contractible variety could be realized as a quotient of an affine space by a free action of a unipotent group, generalizing the situation in topology suggested by Theorem~\ref{thmintro:characterization}.  The answer to this question was seen to be no in \cite{AsokDoranFasel}.  For any integer $n \geq 0$, write $Q_{2n}$ for the smooth affine $k$-scheme defined by the equation $\sum_i x_i y_i = z(1+z)$ in ${\mathbb A}^{2n+1}$ with coordinates $(x_1,\ldots,x_n,y_1,\ldots,y_n,z)$.  Then, define $E_n$ to be the closed subscheme of $Q_{2n}$ defined by $x_1 = \cdots = x_n = 1+z = 0$.  As in the case $n = 2$, $E_{n}$ is isomorphic to affine space of dimension $n$.  We define a variety
\[
X_{2n} := Q_{2n} \setminus E_n.
\]
For $n = 0$, $E_0 = \Spec k$, and $X_0$ is $\Spec k$ as well.  For $n = 1$, one can check that $X_2$ is isomorphic to ${\mathbb A}^2_k$.  For $n = 2$, $X_4$ is Winkelmann's example studied above.  The following example shows that not every $\aone$-contractible smooth scheme may be realized as a quotient of affine space by a free action of the additive group, in contrast to the situation in topology summarized in Theorem~\ref{thmintro:characterization}.

\begin{thm}[{\cite[Theorems 3.1.1 and Corollary 3.2.2]{AsokDoranFasel}}]
\label{thm:highdimquasiaffinecontractible}
Suppose $n \geq 0$.
\begin{enumerate}[noitemsep,topsep=1pt]
\item The variety $X_{2n}$ is $\aone$-contractible.
\item If $n \geq 3$, then $X_{2n}$ is not a quotient of affine space by a free action of a unipotent group.
\end{enumerate}
\end{thm}

\begin{proof}
We will sketch the proof of the first statement, which follows by induction starting from the fact that $X_2 \cong {\mathbb A}^2$ and is thus $\aone$-contractible.  To set up the induction, we introduce some terminology.  Let $U_n$ be the open subscheme of $Q_{2n}$ defined by $x_n \neq 0$.  Note that $U_n$ is isomorphic to ${\mathbb A}^{2n-1} \times \gm{}$ with coordinates $x_1,\ldots,x_{n-1},y_1,\ldots,y_{n},z$ on the ${\mathbb A}^{2n-1}$-factor and coordinate $x_n$ on the $\gm{}$-factor.  The closed complement $Z_n$ of $U_n$, i.e., the closed subscheme of $Q_{2n}$ defined by $x_n = 0$ is isomorphic to $Q_{2n-2} \times \aone$ with coordinate $y_n$ on the $\aone$-factor.  The point $x_1 = \cdots = x_{n-1} = y_1 = \cdots = y_{n} = z = 0$ defines a point $0$ on $Z_n$.  The normal bundle to $Z_n$ is a trivial line bundle, with an explicit trivialization defined by the equation $x_n = 0$.  Note that, by construction $E_n$ is a closed subscheme of $Z_n$, and therefore $U_{2n} \subset X_{2n}$ as well.  Likewise, the subscheme $Z_n \setminus E_n$ is isomorphic to $X_{2n-2} \times \aone$.

The closed subscheme of $U_n$ defined by setting $x_1 = \cdots = x_{n-1} = y_1 = \cdots = y_n = z = 0$ is isomorphic to $\gm{}$ with coordinate $x_n$.  The inclusion map $\gm{} \to U_n$ is a monomorphism of presheaves, and splits the projection $U_n \cong {\mathbb A}^{2n-1} \times \gm{} \to \gm{}$.  In particular, the map $U_n \to \gm{}$ is an $\aone$-weak equivalence.  Likewise, the closed subscheme of $Q_{2n}$ defined by $x_1 = \cdots = x_{n-1} = y_1 = \cdots = y_n = z = 0$ is isomorphic to $\aone$ with coordinate $x_n$ and we have a pullback diagram of the form
\[
\xymatrix{
0 \ar[r]\ar[d] & \aone \ar[d] \\
Z_n \ar[r] & Q_{2n};
}
\]
where the top map is given by $x_n = 0$ and thus maps on normal bundles to the horizontal closed immersions have compatible trivializations.  Since $E_n$ is disjoint from this copy of $\aone$, it follows that we have a sequence of inclusions of the form:
\[
\gm{} \longrightarrow \aone \longrightarrow X_{2n}.
\]
We would like to understand the homotopy cofiber of the composite map, which coincides with the homotopy cofiber of the map $U_n \to X_n$ since the projection map $U_n \to \gm{}$ is an $\aone$-weak equivalence.

The cofiber of a composite fits into a cofiber sequence involving the cofibers of the terms by the ``octahedral" axiom in a model category.  Since $\aone$ is $\aone$-contractible, the homotopy cofiber of the map $\aone \to X_{2n}$ is $X_{2n}$ pointed by $0$.  The cofiber of $\gm{} \to \aone$ is ${\pone}$ by the purity isomorphism (see Theorem~\ref{thm:homotopypurity}).  Likewise, the cofiber of the map $\gm{} \to X_{2n}$ coincides with $Th(\nu_{(Z_n \setminus E_n)/X_{2n}}) = {\pone}_+ \wedge (X_{2n-2} \times \aone)$, and there is thus a cofiber sequence of the form
\[
{\pone} \longrightarrow {\pone} \wedge (Z_n \setminus E_n)_+ \longrightarrow X_{2n}.
\]
By construction, and functoriality of the purity isomorphism (again Theorem~\ref{thm:homotopypurity}), the left hand map is the $\pone$-suspension of the map $S^0_k \to (Z_n \setminus E_n)_+$ given by inclusion of $0$ in the first factor.  It is therefore split by the map $(Z_n \setminus E_n)_+ \to S^0_k$ that corresponds to adding a disjoint base-point to the structure map.  Thus, one concludes that there is an induced $\aone$-weak equivalence $\pone \wedge (Z_n \setminus E_n) \to X_{2n}$.  Since $(Z_n \setminus E_n) \cong X_{2n-2} \times \aone$, it is $\aone$-contractible, and the suspension $\pone \wedge (Z_n \setminus E_n)$ is also $\aone$-contractible.
\end{proof}

\begin{question}
Is the total space of a Jouanolou device over $X_{2n}$ isomorphic to an affine space?
\end{question}

\section{Further computations in $\aone$-homotopy theory}
In the preceding section, we revisited some constructions from affine (and quasi-affine) varieties from the standpoint of $\aone$-homotopy theory.  Of these, constructions and results, only Theorem~\ref{thm:highdimquasiaffinecontractible} really required tools of $\aone$-homotopy theory.  To connect with some of the other questions mentioned in the introduction, in particular the generalized van de Ven Question \ref{question:vandeVenquestion}, in this section we will discuss connectivity from the standpoint of $\aone$-homotopy theory, and close with some of the basic computations of the analogs of classical (unstable) homotopy groups of spheres.

\subsection{$\aone$-homotopy sheaves}
Suppose $(\mathscr{X},x)$ is a pointed space.  Earlier, we defined bi-graded motivic spheres $S^{i,j}$.  These bi-graded motivic spheres allow us to define corresponding homotopy groups.

\subsubsection*{Basic definitions}
\begin{defn}
\label{defn:connectedcomponents}
Given a space $\mathscr{X}$, the sheaf of $\aone$-connected components, denoted $\bpi_0^{\aone}({\mathscr X})$, is the Nisnevich sheaf on $\Sm_k$ associated with the presheaf $U \mapsto [U,\mathscr{X}]_{\aone}$.
\end{defn}

\begin{defn}
If $\mathscr{X}$ is a motivic space, then $\mathscr{X}$ is $\aone$-connected if the canonical morphism $\bpi_0^{\aone}({\mathscr X}) \to \Spec k$ is an isomorphism (and $\aone$-disconnected otherwise).
\end{defn}

\begin{defn}
\label{defn:homotopysheaves}
Given a pointed space $(\mathscr{X},x)$, the $i$-th $\aone$-homotopy sheaf, denoted $\bpi_i^{\aone}(\mathscr{X},x)$, is the Nisnevich sheaf on $\Sm_k$ associated with the presheaf $U \mapsto [S^i \wedge U_+,(\mathscr{X},x)]_{\aone}$.
\end{defn}

As in classical topology, one can formally show that $\bpi_1^{\aone}(\mathscr{X},x)$ is a Nisnevich sheaf of groups, and $\bpi_i^{\aone}({\mathcal X},x)$ is a Nisnevich sheaf of abelian groups for $i > 1$.  In fact, results of Morel show that, just like in topology, these sheaves of groups are ``discrete" in an appropriate sense; see \cite{MICM} for an introduction to these ideas and \cite{MField} for details.  The following result, called the $\aone$-Whitehead theorem for its formal similarity to the ordinary Whitehead theorem for CW complexes, is a formal consequence of the definitions.

\begin{prop}[{\cite[\S 3 Proposition 2.14]{MV}}]
A morphism $f: \mathscr{X} \to \mathscr{Y}$ of $\aone$-connected spaces is an $\aone$-weak equivalence if and only if for any choice of base-point $x$ for $\mathcal{X}$, setting $y = f(x)$ the induced morphism
\[
\bpi_i^{\aone}(\mathscr{X},x) \longrightarrow \bpi_i^{\aone}({\mathscr Y},y)
\]
is an isomorphism.
\end{prop}

The following result is called the unstable $0$-connectivity theorem.

\begin{thm}[{\cite[\S 2 Corollary 3.22]{MV}}]
\label{thm:unstable0connectivity}
If $\mathscr{X}$ is a space, then the canonical map $\mathscr{X} \to \bpi_0^{\aone}(\mathscr{X})$ is an epimorphism after Nisnevich sheafification.
\end{thm}

\subsubsection*{$\aone$-rigid varieties embed into $\ho{k}$}
One rather fundamental difference between the $\aone$-homotopy category and the classical homotopy category is that while classical homotopy types are essentially discrete, $\aone$-homotopy types may vary in families.  We begin by recalling the following definition, which begins to analyze the interaction with morphisms from the affine line and $\aone$-connected components.

\begin{defn}
\label{defn:A1rigid}
A smooth scheme of finite type $X\in \Sm_{k}$ is $\aone$-rigid if the map
\begin{equation}
\label{equation:A1rigidbijection}
\Sm_{k}(Y\times\aone,X)
\longrightarrow
\Sm_{k}(Y,X)
\end{equation}
induced by the $0$-section $\Spec k\to\aone$ is a bijection for every $Y\in \Sm_{k}$.
\end{defn}

\begin{rem}
Let $\pi\colon Y\times\aone\to\aone$ denote the projection map.
Then the composite map
\begin{equation}
\label{equation:A1rigidcomposite}
\Sm_{k}(Y,X)
\overset{\pi^{\ast}}{\to}
\Sm_{k}(Y\times\aone,X)
\longrightarrow
\Sm_{k}(Y,X)
\end{equation}
is the identity.  Thus $X$ is $\aone$-rigid if and only if $\pi^{\ast}$ is surjective or equivalently \eqref{equation:A1rigidbijection} is injective for all $Y\in \Sm_{k}$.
\end{rem}

\begin{lem}
A smooth scheme of finite type $X\in \Sm_{k}$ is $\aone$-local fibrant if and only if it is $\aone$-rigid.
\end{lem}

\begin{proof}
Since every object of $\Sm_{k}$ is local projective fibrant, $X\in \Sm_{k}$ is $\aone$-local fibrant if and only if \eqref{equation:A1rigidbijection} is a weak equivalence of simplicial sets.  We note that every discrete simplicial set is cofibrant and fibrant.  Thus \eqref{equation:A1rigidbijection} is a weak equivalence if and only if it is a homotopy equivalence or equivalently a bijection.
\end{proof}

\begin{cor}
\label{cor:aonerigidfullyfaithful}
The full subcategory of $\aone$-rigid schemes in $\Sm_{k}$ embeds fully faithfully into $\ho{k}$.  Moreover, if $X$ is $\aone$-rigid, the canonical morphism $X \to \bpi_0^{\aone}(X)$ is an isomorphism of Nisnevich sheaves.
\end{cor}

\begin{ex}
\label{ex:curvesaonerigid}
We note that $\gm{}$ is $\aone$-rigid.  To show \eqref{equation:A1rigidbijection} is injective for all $Y\in \Sm_{k}$ we may assume $Y=\Spec R$, where $R$ is a finitely generated $k$-algebra.  In fact, if $R$ is a reduced commutative ring, then the pullback map $R^{\times} \to R[x]^{\times}$ is bijective.  More generally, any open subscheme of $\gm{}$ is $\aone$-rigid.  Similarly, one may show that if $X$ is a smooth curve of genus $g \geq 1$ and $U \subset X$ is an open subscheme, then $U$ is also $\aone$-rigid.
\end{ex}

\begin{ex}
The scheme $\pone$ is not $\aone$-rigid because $\pi^{\ast}$ is not surjective when $Y=\Spec k$ (there are of course many non-constant embeddings $\aone\hookrightarrow\pone$).  We will explore this example more in Section~\ref{ss:aoneconnected}.
\end{ex}

\begin{ex}
Recall that a semi-abelian variety is a smooth connected algebraic group $G$ obtained as an extension
\[
1 \longrightarrow T \longrightarrow G \longrightarrow A \longrightarrow 1
\]
of an abelian variety $A$, i.e., a smooth connected proper algebraic group, by a torus $T$.  For example, $A$ can be the Jacobian of an algebraic curve of positive genus and $T$ can be the multiplicative group scheme.  For any map $\phi\colon\aone\to G$ the composite $\aone\to G\to A$ is constant.  To show this we may assume $k$ is algebraically closed.  Indeed every map $\rho\colon\pone\to A$ is constant: by L{\"u}roth's theorem, we are reduced to consider the normalization, i.e., we may assume that $\rho$ is birational onto its image.  In that case, the differential $d\rho: T_{\pone}\to T_{A}$ is injective being non-zero at the generic point.  Now the tangent sheaf $T_{A}$ is trivial while $T_{\pone}\cong \mathscr{O}_{\pone}(2)$.  However,  there is no injective map $\mathscr{O}_{\pone}(2) \to \mathscr{O}_{\pone}^{\oplus n}$, where $n$ is the dimension of $A$.

Returning to $\phi$ we conclude there exists $g\in G(k)$ such that $\phi$ factors through the translate $gT\subset G$, and we may view $\phi$ as a map $\aone\to (\aone\setminus \{0\})^{n}$ for some $n>0$.  It follows that $f$ is constant because every map $\aone\to \aone\setminus \{0\}$ is constant.
In fact, the affine line provides a useful geometric characterization of semi-abelian varieties by \cite[Proposition 5.4.5]{MR3645068}:
If $G$ is a smooth connected algebraic group over a perfect field, then $G$ is a semi-abelian variety if and only if every map $\aone\to G$ is constant.

Finally, since abelian varieties exist in non-constant families, we see $\aone$-homotopy types exist in non-constant families.
\end{ex}

\begin{rem}
\label{rem:algebraichyperbolicity}
Over the complex numbers, an algebraic variety $X$ is called {\em Brody} hyperbolic if every holomorphic map $\cplx \to X$ is constant and {\em algebraically hyperbolic} (also called {\em Mori hyperbolic}) if every algebraic morphism $\aone \to X$ is constant.  It is not hard to show that Mori hyperbolic varieties are $\aone$-rigid in the sense above.
\end{rem}

\subsection{$\aone$-connectedness and geometry}
\label{ss:aoneconnected}
Having explored varieties that were discrete from the standpoint of $\aone$-homotopy theory, we now discuss aspects of connectedness in $\aone$-homotopy theory.
Recall that a motivic space $\mathscr{X}$ is $\aone$-connected if the canonical morphism $\bpi_0^{\aone}({\mathscr X}) \to \Spec k$ is an isomorphism (and $\aone$-disconnected otherwise).

\begin{rem}
Since the $\aone$-homotopy category is constructed by a localization procedure, the sheaf $\bpi_0^{\aone}({\mathscr X})$ is rather abstractly defined and hard to ``compute" in practice.  To give one indication of how $\aone$-connectedness interacts with arithmetic, suppose $X$ is a scheme $k$-scheme with $k$ a field.  Since stalks in the Nisnevich topology on $\Sm_k$ are henselizations of points on smooth schemes, Theorem~\ref{thm:unstable0connectivity} implies that if $X$ is an $\aone$-connected smooth scheme, then $X(S)$ is non-empty for $S$ every henselization of a smooth scheme at a point.  In particular, $\aone$-connected smooth schemes always have $k$-rational points.
\end{rem}

One would like to have a more ``geometric" interpretation of $\aone$-connectedness.  Of course, any $\aone$-contractible space is $\aone$-connected, by the very definition.  For this, we recall how connectedness is studied in topology: a topological space is {\em path} connected if any two points can be connected by a map from the unit interval.  Replacing the unit interval by the affine line, we could define a notion of $\aone$-path connectedness.  For flexibility, we will use a slightly more general definition.

\begin{defn}
If $X$ is a smooth $k$-scheme, say that $X$ is {\em $\aone$-chain connected} if for every separable, finitely generated extension $K/k$, $X(K)$ is non-empty, and for any pair $x, y \in X(K)$, there exist an integer $N$ and a sequence $x = x_0, x_1,\ldots,x_N = y \in X(K)$ together with morphisms $f_1,\ldots,f_N: \aone_K \to X$ with the property that $f_i(0) = x_{i-1}$ and $f_i(1) = x_i$; loosely speaking: any two points can be connected by the images of a chain of maps from the affine line.
\end{defn}

\begin{rem}
Note: $K$ is not necessarily a finite extension, so this definition is nontrivial even when $k = \cplx$.  Indeed, in that case, we ask, e.g., that $\cplx(t)$-points, $\cplx(t_1,t_2)$-points, etc. can all be connected by the images of chains of affine lines.
\end{rem}

From the definitions given, it is not clear that either $\aone$-connectedness implies $\aone$-chain connectedness or vice versa.  In one direction, the problem is that $\aone$-chain connectedness only imposes conditions over fields: while fields are examples of stalks in the Nisnevich topology, they do not exhaust all examples of stalks.

\begin{prop}[{\cite[Lemma 3.3.6]{MIntro} and \cite[Lemma 6.1.3]{MStable}}]
If $X$ is an $\aone$-chain connected smooth variety, then $X$ is $\aone$-connected.
\end{prop}

\begin{proof}[Idea of proof.]
The proof uses the fact that we are working with the Nisnevich topology in a fairly crucial way.  To check triviality of all stalks, it suffices to show that $\bpi_0^{\aone}(X)(S)$ is trivial for $S$ a henselian local scheme.  Chain connectedness implies that the sections over the generic point of $S$ are trivial and also that the sections over the closed point are trivial.  We can then try to use a sandwiching argument to establish that sections over $S$ are also trivial: in practice, this uses the homotopy purity theorem \ref{thm:homotopypurity}!
\end{proof}

Conversely, it is not clear that $\aone$-connectedness implies $\aone$-chain connectedness.  However, one may prove the following result.

\begin{thm}[{\cite[Theorem 6.2.1]{AM}}]
\label{thm:aoneconnectedrtrivial}
If $X$ is a smooth proper $k$-variety, and $K/k$ is any separable finitely generated extension, then $\bpi_0^{\aone}(X)(K) = X(K)/\sim$.  In particular, if $X$ is $\aone$-chain connected, then $X$ is $\aone$-connected.
\end{thm}

\begin{rem}
Another proof of this result has been given in \cite{BalweHogadiSawant}.
\end{rem}

\subsubsection*{$\aone$-connectedness and rationality properties}
The preceding theorem suggests a link between $\aone$-connectedness and rationality properties.  Indeed, Manin defined the notion of $R$-equivalence for rational points on an algebraic variety: if $L/k$ is a finite extension and $X$ is a smooth $k$-scheme, we say two $L$-points $x_0$ and $x_1$ are directly  $R$-equivalent if there exists a rational map $\pone_L \to X_L$ defined at the points $x_0$ and $x_1$.  We say two $L$-points are $R$-equivalent if they are equivalent for the equivalence relation generated by direct $R$-equivalence, and we write $X(L)/R$ for the set of $R$-equivalence classes of $L$-rational points.  One says that a variety $X/k$ is {\em universally $R$-trivial} if $X(L)/R = \ast$ for every finitely generated separable extension $L/k$.  With that definition, a smooth proper variety $X$ that is $\aone$-chain connected is automatically {\em universally $R$-trivial}, and Theorem~\ref{thm:aoneconnectedrtrivial} may be phrased as saying that $\aone$-connected smooth proper varieties are universally $R$-trivial.  In fact, one may make a slightly stronger version of this statement.

\begin{prop}
\label{prop:aoneconnectedunivRtrivial}
If $k$ is a field and $U$ is an $\aone$-connected smooth $k$-variety that admits a smooth proper compactification $X$, then $X$ is $\aone$-connected and $U$ is universally $R$-trivial.
\end{prop}

\begin{rem}
Of course, the hypothesis of admitting a smooth proper compactification is superfluous if either $k$ has charateristic $0$ or $U$ has small dimension.
\end{rem}

\begin{proof}
The easiest proof of this fact uses a tool that we have not yet introduced called the zeroth $\aone$-homology sheaf, which plays a role very similar to singular homology in classical topology, but is rather far from motivic (co)homology mentioned earlier.  Since $U \subset X$ is open, the map
\[
\mathrm{H}_0^{\aone}(U) \longrightarrow \mathrm{H}_0^{\aone}(X).
\]
is an epimorphism by \cite[Proposition 3.8]{ABirational}.  On the other hand, since $U$ is $\aone$-connected, there is an isomorphism $H_0^{\aone}(U) \cong \Z$.  Since $X$ is proper, one concludes $\mathrm{H}_0^{\aone}(X) \cong \Z$ as well.  Then, the fact that $X$ is $\aone$-connected follows from \cite[Theorem 4.15]{ABirational}.  To conclude, simply observe that if $x_0$ and $x_1$ are $L$-points in $U$, then $x_0$ and $x_1$ are connected by a chain of affine lines over $L$ in $X$.  Restricting this chain of affine lines to $U$ gives the required witness to $R$-equivalence.
\end{proof}

In light of the van de Ven question mentioned in the introduction, we observe that Proposition~\ref{prop:aoneconnectedunivRtrivial} has the following consequence on the rationality of $\aone$-contractible varieties.

\begin{cor}
Assume $k$ is a field having characteristic $0$.  If $X$ is an $\aone$-contractible smooth $k$-scheme then $X$ is universally $R$-trivial.
\end{cor}

\begin{rem}
If $k$ is a field and $X$ is a smooth proper variety, then any $\aone$-connected smooth proper variety is separably rationally connected.  However, $\aone$-connectedness has cohomological implications, e.g., the (prime to the characteristic part of the) Brauer group of an $\aone$-connected smooth proper $k$-scheme is automatically trivial (see \cite[\S 4]{ABirational} for a detailed discussion of this point).  For example, it is known that there exist $k$-unirational varieties that are not $\aone$-connected \cite[\S 2.3]{AM} \cite[Example 4.18]{ABirational}.  Thus, $\aone$-connectedness of a smooth scheme has nontrivial implications for the rationality properties of the scheme.
\end{rem}

We know that $\aone$-connectedness implies universal $R$-triviality for smooth schemes admitting a smooth proper compactification.  However, the example of $\gm{}$ shows that $R$-equivalent $k$-points need not be connected by (chains of) rational curves.  Nevertheless, the following question remains open:

\begin{question}
Is it true for an arbitrary smooth $k$-scheme $X$ that $\aone$-connectedness is equivalent to $\aone$-chain connectedness?
\end{question}

\paul{Counterexample in \cite{BalweHogadiSawant}?}

\begin{rem}
The relationship between $R$-equivalence and $\aone$-weak equivalence of points has been studied on certain linear group schemes in \cite{BalweSawant}.
\end{rem}

\subsection{$\aone$-homotopy sheaves spheres and Brouwer degree}
Earlier, we saw that ${\mathbb A}^n \setminus 0$ was a motivic sphere: it was isomorphic in $\ho{k}$ to $S^{n-1,n}$.  It is not hard to show that ${\mathbb A}^n \setminus 0$ is $\aone$-connected for $n \geq 2$, so it is natural to inquire about its connectivity and to compute its first non-vanishing $\aone$-homotopy sheaf.  We quickly summarize some results of F. Morel, though we do not have enough space to motivate the proofs.

\begin{thm}[F. Morel]
If $k$ is a field, then ${\mathbb A}^n \setminus 0$ is at least $(n-2)$-$\aone$-connected, i.e., if $n \geq 2$, it is $\aone$-connected and $\bpi_i^{\aone}({\mathbb A}^n \setminus 0,x)$ vanishes for any choice of $k$-point $x \in {\mathbb A}^n \setminus 0 (k)$ and any integer $1 \leq i \leq n-2$.
\end{thm}

Morel also computed the first non-vanishing $\aone$-homotopy sheaf in terms of Milnor--Witt K-theory introduced earlier.

\begin{thm}[F. Morel]
If $k$ is a field, then for any integer $n \geq 2$, and any finitely generated separable extension $L/k$
\[
\bpi_{n-1}^{\aone}({\mathbb A}^n \setminus 0)(L) = \K^{MW}_n(L).
\]
\end{thm}

From this result, Morel deduced the computation of the homotopy endomorphisms of ${\mathbb A}^n \setminus 0$.  If $k \hookrightarrow \cplx$, then ${\mathbb A}^{n}(\cplx)$ is homotopy equivalent to $S^{2n-1}$.  Therefore, realization defines a homomorphism $[{\mathbb A}^n \setminus 0,{\mathbb A}^n \setminus 0]_{\aone} \to [S^{2n-1},S^{2n-1}] \cong \Z$ where the latter identification is the usual Brouwer degree.

\begin{thm}
\label{thm:degree}
If $k$ is a field, then for any integer $n \geq 2$, there is a canonical ``motivic Brouwer degree" isomorphism
\[
[{\mathbb A}^n \setminus 0,{\mathbb A}^n \setminus 0]_{\aone} \cong GW(k).
\]
Moreover, if $k \hookrightarrow \cplx$, the induced homomorphism $GW(k) \to \Z$ coincides with the homomorphisms induced by sending a stable isomorphicm class of symmetric bilinear forms to the rank of its underlying $\cplx$-vector space.
\end{thm}

Later, we will see that computations of motivic Brouwer degree appear in proofs of $\aone$-contractibility statements.

\subsection{$\aone$-homotopy at infinity}
\label{ss:endspaces}
In this section, we introduce some notions of $\aone$-homotopy theory at infinity.  Unfortunately, we are unable at the moment to make a good definition of ``end space" in order to define a workable notion of $\aone$-fundamental group at infinity.

\subsubsection*{One-point compactifications}
Fix a field $k$, and suppose $X$ is a smooth $k$-scheme.  By a smooth compactification of $X$ we will mean a pair $(\bar{X},\psi)$, where $\psi: X \to \bar{X}$ is an open dense immersion, and $\bar{X}$ is smooth.  We will say that such a smooth compactification is {\em good} if the closed complement of $\psi$ (viewed as a scheme with the reduced induced structure) is a simple normal crossings divisor.

\begin{lem}
If $X$ is a smooth scheme, and $\bar{X}_0$ and $\bar{X}_1$ are smooth compactifications, with boundaries $\partial \bar{X}_0$ and $\partial \bar{X}_1$, then $\bar{X}_0/\partial\bar{X}_0$ and $\bar{X}_1/\partial\bar{X}_1$ are cdh-locally weak equivalent and thus $\Sigma \bar{X}_0/\partial \bar{X}_0$ and $\Sigma \bar{X}_1/\partial \bar{X}_1$ are weakly equivalent in the $\aone$-homotopy category.
\end{lem}
\begin{proof}
By definition of the cdh topology, if we are given an abstract blow-up square of the form
\[
\xymatrix{
E \ar[r]\ar[d] & X' \ar[d] \\
Y \ar[r] & X,
}
\]
i.e., $Y \to X$ is a closed immersion, $X' \to X$ is proper and the induced map $X' \setminus E \to X \setminus Y$ is an isomorphism, then there is a cdh local weak equivalence $X'/E \to X/Y$.  To establish the statement, simply take the closure $\bar{X}$ of the image of the diagonal map $X \to \bar{X}_0 \times \bar{X}_1$ and observe that $\bar{X} \to \bar{X}_0$ and $\bar{X} \to \bar{X}_1$ yield abstract blow-up squares.  The second statement follows from the first because any morphism of presheaves that is a cdh local weak equivalence becomes a Nisnevich local weak equivalence after a single suspension (i.e., Theorem~\ref{thm:voevodskyresolutionofsings}).
\end{proof}

The lemma above shows that one-point compactifications are well-behaved in the cdh-local version of the $\aone$-homotopy category.  Alternatively, the $S^1$-stable $\aone$-homotopy type of a one-point compactification is well-defined.

\begin{defn}
\label{defn:onepointcompactification}
If $X$ is a smooth $k$-scheme, then for any compactification $\bar{X}$ of $X$, we set $\dot{X} = \bar{X}/\partial \bar{X}$; there is a natural map $X \to \dot{X}$.
\end{defn}

\begin{lem}
Suppose $X$ is a smooth scheme.  The following statements hold.
\begin{enumerate}[noitemsep,topsep=1pt]
\item If $X$ is proper, then $\dot{X} = X_+$.
\item For any integer $n \geq 0$, there is an $\aone$-weak equivalence $\dot{(X \times {\mathbb A}^n)} \cong \dot{X} \sma {\pone}^{\sma n}$.
\end{enumerate}
\end{lem}

\begin{proof}
The first statement is immediate from $X/{\emptyset} = X_+$.
For the second statement, one can first observe that ${\mathbb P}^n$ is a compactification of ${\mathbb A}^n$ with boundary ${\mathbb P}^{n-1}$ and use the $\aone$-weak equivalence
${\mathbb P}^n/{\mathbb P}^{n-1} \cong {\pone}^{\sma n}$.
Then, use the fact that if $\bar{X}$ is any compactification of $X$, then $\bar{X} \times {\mathbb P}^n$ is a compactification of $X \times {\mathbb A}^n$.
\end{proof}

\subsubsection*{Stable end spaces}
Our goal is to make some progress toward a definition of end-spaces in algebraic geometry.  There are many possible approaches to such a definition, and we do not know whether they are all equivalent.  One approach is to consider a ``punctured formal neighborhood" of the boundary in a compactification; this approach is not suited to motivic homotopy theory using smooth schemes.  The following definition is motivated by the definition of singular homology at infinity, in the case where the boundary is suitably ``tame".
\vspace{0.1in}

J. Wildeshaus introduced the notion of ``boundary motive" \cite{Wildeshaus} of a variety; it can be thought of as a motivic version of the singular chain complex at infinity.  In fact, with notational modifications, Wildeshaus' definition gives a $\pone$-stable homotopy type.  The only novelty of the definition we give below is that it gives an $S^1$-stable homotopy type.

\begin{defn}
\label{defn:endspace}
Assume $X$ is a smooth $k$-scheme.  The $S^1$-stable end space is defined to fit into the following exact triangle:
\[
X \coprod \{ \infty \} \longrightarrow \dot{X} \longrightarrow e(X)[1]
\]
\end{defn}

\begin{rem}
The main benefit of working $S^1$-stably instead of $\pone$-stably is that one has some hope of uncovering unstable phenomena related to ``$\aone$-connected components at $\infty$".  Indeed, the zeroth $\aone$-homology sheaf of a smooth proper scheme (see the proof of Proposition~\ref{prop:aoneconnectedunivRtrivial}) detects rational points, while the corresponding $\pone$-stable object only sees zero cycles of degree $1$ \cite{ABirational,AsokHaesemeyer}.  Thus, the above definition should be refined enough to detect some algebro-geometric analog of the number of ends of a space.
\end{rem}

\begin{prop}
Assume $k$ is a field having characteristic $0$.
\begin{enumerate}[noitemsep,topsep=1pt]
\item The construction $e(-)$ is a functor on the category of smooth schemes and proper maps.
\item If $X$ is $S^1$-stably $\aone$-contractible, then $e(X) \cong \dot{X}/\infty[-1]$.
\end{enumerate}
\end{prop}

\begin{ex}
The end space of $\real$ is $S^0$.  Analogously, $e(\aone) = \gm{}$.  In particular, end spaces need not even be $\aone$-connected.  Similarly, we find $e({\mathbb A}^n) \cong {\mathbb A}^n \setminus 0$.
\end{ex}

\begin{ex}
Suppose $X$ is an $\aone$-contractible smooth affine scheme of dimension $n$.
If $X$ is isomorphic to ${\mathbb A}^n$ then $e(X) \cong {\mathbb A}^n \setminus 0$.
\end{ex}

End spaces as defined here are also compatible with realization.  Indeed, the triangle defining $e(X)$ makes sense in the usual stable homotopy category, and we define $e_{\cplx}(X)$ to be the cofiber of the natural map from the suspension spectrum of $X(\cplx) \coprod \{ \infty \}$ to the suspension spectrum of $\dot{X(\cplx)}$.  Because realization behaves well with respect to homotopy cofibers, the following result is immediate.

\begin{prop}
If $\iota: k \hookrightarrow \cplx$ is an embedding, then ${\mathfrak R}_{\iota}(e(X)) = e_{\cplx}(X)$.
\end{prop}

\begin{question}
Is there a ``good" unstable definition of end spaces in $\aone$-homotopy theory?
\end{question}



\section{Cancellation questions and $\aone$-contractibility}
\subsection{The biregular cancellation problem}
\label{ss:cancellation}
After the work of Iitaka--Fujita and Danielewski, it became clear that cancellation questions could admit negative solutions for smooth affine varieties of negative logarithmic Kodaira dimension.  Perhaps the main remaining question in this direction is as follows.

\begin{question}[Biregular cancellation]
If $X$ is a smooth scheme of dimension $d$ such that
\begin{enumerate}[noitemsep,topsep=1pt]
\item $X$ is stably isomorphic to ${\mathbb A}^d$, or
\item $X$ is a direct factor of ${\mathbb A}^N$ for some $N > d$, or
\item $X$ is a retract of ${\mathbb A}^N$ for some $N > d$, or
\end{enumerate}
is $X$ necessarily isomorphic to affine space?
\end{question}

\begin{rem}
The question of whether varieties that are stably isomorphic to affine space are necessarily isomorphic to affine space is sometimes called {\em the Zariski cancellation question}, but as mentioned earlier, Zariski never explicitly stated this question.  On the other hand Beilinson apparently asked whether any retract of affine space is isomorphic to affine space \cite[\S 8]{ZaidenbergSurvey}.
\end{rem}

The biregular cancellation question is known to have a positive answer for smooth affine schemes of dimension $1$ by \cite{AHE}, and also for smooth affine schemes of dimension $2$: the result was established by Miyanish--Sugie and Fujita \cite{MiyanishiSugie,FujitaCancellation} over characteristic $0$ fields and extended to perfect fields of arbitrary characteristic in \cite{RussellCancellation}.  In contrast, we now know that the biregular cancellation problem admits a negative answer over algebraically closed fields having positive characteristic.  Indeed, N. Gupta constructed a counter-example in dimension $3$ \cite{Gupta3dimcancellation} and extended this result in a number of directions \cite{GuptaCancellation,GuptaFamilies}.  We refer the reader to \cite{GuptaSurvey} for more discussion of these results.  However, the specific form of these counterexamples does not allow them to be lifted to characteristic $0$.

\begin{lem}
If $X$ is a smooth scheme of dimension $d$ that is a retract of ${\mathbb A}^N$, then $X$ is a smooth affine, $\aone$-contractible scheme.  In particular, any counter-example to the biregular cancellation problem is necessarily a smooth affine $\aone$-contractible scheme.
\end{lem}

\begin{proof}
Any retract of an $\aone$-weak equivalence is an $\aone$-weak equivalence.
\end{proof}

\begin{rem}
Gupta's counterexamples to cancellation above provided the first examples of smooth {\em affine} $\aone$-contractible schemes in positive characteristic.
\end{rem}

\begin{rem}
Over $\cplx$, the biregular cancellation question remains open in dimensions $\geq 3$.
\end{rem}

Granted the Quillen--Suslin theorem on triviality of vector bundles on affine space, it is easy to see that any algebraic vector bundle on a variety that is a retract of affine space is automatically trivial.  The representability theorem for vector bundles guarantees that the same statement holds for $\aone$-contractible smooth affine varieties.

\begin{thm}
If $X$ is a smooth affine $\aone$-contractible variety, then all vector bundles on $X$ are trivial.
\end{thm}

\subsection{$\aone$-contractibility vs topological contractibility}
\label{ss:aonecontractibilityvstopologicalcontractibility}
We now compare $\aone$-contractibility and topological contractibility in more detail.  In particular, we would like to know whether topological contractibility and $\aone$-contractibility are actually different.  The best we can do at the moment is to proceed dimension by dimension.  The only topologically contractible smooth curve is $\aone$.  However, already in dimension $2$ problems appear to arise.

\subsubsection*{Affine lines on topologically contractible surfaces}
With the exception of ${\mathbb A}^2$, most topologically contractible surfaces appear to have very few affine lines, as we now explain.  Based on the classification results (see \cite{ZaidenbergSurvey} and the references therein for more details), it suffices to treat the case of surfaces of logarithmic Kodaira dimensions $1$ and $2$ (there are no contractible surfaces of logarithmic Kodaira dimension $0$).

\begin{rem}
General conjectures in algebraic geometry and arithmetic of Green--Griffiths and Lang suggest that smooth proper varieties of general type should not have ``many" rational curves (see, e.g., \cite{Demailly}).  Analogously, one hopes that affine varieties of log general type should not have ``many" morphisms from the affine line (see, e.g., \cite{LuZhang}).  The topologically contractible surfaces of logarithmic Kodaira dimension $2$ contain no contractible curves by work of Zaidenberg \cite{Zaidenbergloggentype, ZaidenbergloggentypeII} and Miyanishi-Tsunoda \cite{MiyanishiTsunoda}; what can one say about morphisms from the affine line to such a surface?  For example, are such surfaces Mori hyperbolic (see Remark~\ref{rem:algebraichyperbolicity})?
\end{rem}

\begin{rem}
The surfaces of logarithmic Kodaira dimension $1$ are all obtained from some special surfaces (the so-called tom Dieck-Petrie surfaces) by repeated application of a procedure called an affine modification (an affine variant of a blow-up).  How does $\aone$-chain connectedness behave with respect to affine modifications (we understand well how $\aone$-chain connectedness behaves with respect to blow-ups of projective schemes with smooth centers).  One could also try to use the rationality results of Gurjar-Shastri, i.e., that any smooth compactification of a topologically contractible surface is rational \cite{GSI,GSII}.
\end{rem}

Based on these observations, it seems reasonable to expect that topologically contractible surfaces that are not isomorphic to ${\mathbb A}^2$ are disconnected from the standpoint of $\aone$-homotopy theory, which leads to the following conjecture, which suggests an answer to Question~\ref{questionintro:lowdimensions} from the introduction.

\begin{conj}
A smooth topologically contractible surface $X$ is $\aone$-contractible if and only if it is isomorphic to ${\mathbb A}^2$.
\end{conj}

The generalized van de Ven Question \ref{question:vandeVenquestion} asks whether all topologically contractible varieties are rational.  For $\aone$-contractible varieties, by ``soft" methods, one can establish ``near rationality" as we observed above. The upshot of this discussion is that $\aone$-contractibility is a significantly stronger restriction on a space than topological contractibility.  In support of the above conjecture, the following classification result was observed in \cite{DPO}.

\begin{prop}
\label{prop:isotoaffineplane}
An $\AA^{1}$-contractible and $\AA^{1}$-chain connected smooth affine surface over an algebraically closed field of characteristic $0$ is isomorphic to the affine plane $\AA^{2}$.
\end{prop}


\begin{ex}[\cite{DPO}]
\label{example:tomDieck-Petriesurfaces}
For coprime integers $k>l\geq 2$, the smooth tom Dieck-Petrie surface is defined as
\begin{equation}
\label{equation:tomDieck-Petriesurfaces}
\mathcal{V}_{k,l}
:=
\big\{\frac{(xz+1)^k-(yz+1)^l-z}{z}=0\big\}
\subset
\AA^{3}
\end{equation}
We note that $\mathcal{V}_{k,l}$ is topologically contractible and stably $\AA^{1}$-contractible.
However, $\mathcal{V}_{k,l}$ has logarithmic Kodaira dimension $\overline{\kappa}(\mathcal{V}_{k,l})=1$,
and thus it cannot be $\AA^{1}$-chain connected.
This example shows that the affine modification construction does not preserve $\AA^{1}$-chain connectedness.
It is an open question whether or not $\mathcal{V}_{k,l}$ is $\AA^{1}$-contractible.
\end{ex}

Establishing $\aone$-connectedness is a first step towards understanding $\aone$-homotopy type.  A first step toward answering Question~\ref{question:vandeVenquestion} in higher dimensions thus seems to begin with an analysis of the following problem.

\begin{problem}
Which classes of topologically contractible varieties are known to be $\aone$-chain connected?
\end{problem}

\subsubsection*{Chow groups and vector bundles on topologically contractible surfaces}
The only general result about vector bundles smooth topologically contractible varieties of arbitrary dimension pertains to the Picard group.

\begin{thm}[{\cite[Theorem 1]{GurjarPicard}}]
\label{thm:gurjarPic}
If $X$ is a topologically contractible smooth complex variety, then $Pic(X) = 0$.
\end{thm}

\begin{proof}
Gurjar states this result for affine varieties.  To remove the affineness assumption, one may either inspect the proof and see that the assumption is never used, or one may reduce to the affine case by simply observing that $Pic(X)$ is $\aone$-invariant by Proposition~\ref{prop:picrepresentable}, and any topologically contractible smooth complex variety is is $\aone$-weakly equivalent to a topologically contractible smooth affine scheme by appeal to Lemma~\ref{lem:jouanolouthomason}.
\end{proof}

On the other hand, Theorem ~\ref{thm:vbrepresentable} together with standard techniques of obstruction theory allows one to understand vector bundles on topologically contractible smooth affine varieties.  Indeed, the ideas of \cite{AsokFasel,AsokFaselSplitting} show that classification of vector bundles on smooth affine schemes of low dimensions are reduced to analysis of Chow groups.  If $X$ is a topologically contractible smooth affine complex variety of dimension $d$, then one may relate the structure of $CH^d(X)$ to geometry.  Indeed, if $Y$ is any smooth complex affine variety of dimension $d$, then a theorem of Roitman implies that $CH^d(Y)$ is uniquely divisible.  Thus, if $CH^d(Y)$ is furthermore finitely generated, it must be trivial.  The latter condition may be guaranteed by imposing conditions on the geometry of compactifications and is thus related to the generalized van de Ven question \ref{question:vandeVenquestion}.  The Chow groups of topologically contractible surfaces, which were originally computed by Gurjar and Shastri \cite{GSI,GSII}, may be computed in this way.  Indeed, Gurjar and Shastri show that the generalized van de Ven question has a positive answer in dimension $2$, so it follows immediately that if $X$ is a topologically contractible smooth complex affine surface, then $CH^2(X)$ is trivial.

\begin{thm}
\label{thm:vbtopcontractiblesurfaces}
If $X$ is a topologically contractible smooth complex surface, then every algebraic vector bundle on $X$ is trivial.
\end{thm}

\begin{proof}
Suppose $X$ is a smooth affine surface over an algebraically closed field.  By Serre's splitting theorem, it suffices to prove that rank $1$ and rank $2$ bundles are trivial.  However, $Pic(X) \cong CH^1(X)$.  The results of \cite[Theorem 1]{AsokFasel} imply that the canonical map
\[
(c_1,c_2): \mathscr{V}_2(X) \longrightarrow CH^1(X) \times CH^2(X)
\]
is a bijection.  If $k = \cplx$ and $X$ is furthermore topologically contractible, then Theorem~\ref{thm:gurjarPic} implies that $CH^1(X) = Pic(X) = 0$.  The argument that $CH^2(X) = 0$ is given before the statement.
\end{proof}

\begin{rem}
In fact, the results of Gurjar and Shastri give a much more refined result than $CH^2(X) = 0$ for a topologically contractible smooth complex affine variety.  In \cite{AAcyclic}, the Voevodsky motive of a topologically contractible smooth complex affine surface is seen to be isomorphic to that of a point.  This implies that the Chow groups are {\em universally trivial}, i.e., for every finitely generated extension $L/\cplx$, $CH^2(X_L) = 0$ as well.
\end{rem}

Similarly, the generalized Serre Question \ref{questionintro:generalizedserrequestion} in dimension $3$ may be reduced to a question that is purely cohomological.

\begin{thm}
\label{thm:vbonthreefolds}
If $X$ is a topologically contractible smooth complex threefold, then every algebraic vector bundle on $X$ is trivial if and only if $CH^2(X)$ and $CH^3(X)$ are trivial.
\end{thm}

\begin{proof}
As in the proof of Theorem~\ref{thm:vbtopcontractiblesurfaces}, it suffices to treat the case of ranks $1$, $2$ and $3$ by Serre's splitting theorem.  It is known that the Picard group of any topologically contractible threefold is trivial, so the rank $1$ case follows.  The results of \cite[Theorem 1]{AsokFasel} imply that if, furthermore, $CH^2(X)$ is trivial, then every rank $2$ vector bundle is trivial.  Finally, classical results of Mohan Kumar and Murthy imply that there is a unique rank $3$ vector bundle with given $(c_1,c_2,c_3)$ \cite{MKM}.  In particular, if $CH^3(X)$ is also trivial, it follows that every rank $3$ bundle on such an $X$ is trivial.
\end{proof}

\begin{rem}
In \cite{AAcyclic}, it is observed that one may produce threefolds with trivial Chow groups by means of the technique of affine modifications, so one may produce many examples of topologically contractible threefolds satisfying the above hypotheses.  Similar observations are used in \cite{MR3549169} to establish triviality of vector bundles on Koras--Russell threefolds (see Section~\ref{ss:KRthreefolds} and the discussion after Theorem~\ref{thm:MC-Vanish} for more details).
\end{rem}

As sketched above, one may use geometry to analyze vector bundles on topologically contractible smooth complex affine threefolds.  Indeed, suppose $X$ is a topologically contractible smooth complex affine threefold admitting a compactification $\bar{X}$ that is rationally connected.  Of course, this is weaker than assuming the generalized van de Ven Question \ref{question:vandeVenquestion} has a positive solution in dimension $3$.  In that case $CH^3(\bar{X}) = \Z$ and thus $CH^3(X)$ is trivial.  \cite[Theorem 1.3]{TianZong} implies that $CH^2(\bar{X})$ is generated by classes of rational curves.  Using this, and the localization sequence, one can sometimes establish that $CH^2(X)$ is itself torsion.  In that case, \cite[Appendix Theorem]{KumarConner} due to Srinivas implies that $CH^2(X)$ is actually trivial.  For example, it follows from \cite[Corollary 18]{KumarConner} that if $X$ is a topologically contractible smooth complex affine threefold that admits a finite morphism from ${\mathbb A}^3$, then all vector bundles on $X$ are trivial.

\begin{rem}
For some comments on the situation in dimension $\geq 4$, see Conjecture~\ref{conj:stablecontractibility}.
\end{rem}

\subsection{Cancellation problems and the Russell cubic}
\label{ss:KRthreefolds}
We now investigate $\aone$-contractible smooth affine varieties over fields having characteristic $0$.  For concreteness, it's useful to focus on one particular case:  let $\mathcal{KR}$ be the so-called Russell cubic, i.e., the smooth variety in ${\mathbb A}^4$ defined by the equation:
\[
x + x^2y + z^3 + t^2 = 0
\]

There is a natural $\Gm$-action on $\mathcal{KR}$ given by
\begin{equation}
\label{equation:KRGmaction}
\Gm\times\mathcal{KR}
\to
\mathcal{KR};
(\lambda,x,y,z,t)\mapsto (\lambda^{6}x,\lambda^{-6}y,\lambda^{2}z,\lambda^{3}t).
\end{equation}
With respect to this action we have the $\Gm$-invariant variables $u:=xy, v:=yt^{2}\in\mathscr{O}(\mathcal{KR})^{\Gm}$.
In fact,
the GIT-quotient for \eqref{equation:KRGmaction} is given by
\begin{equation}
\label{equation:KRGITquotient}
\pi
:
\mathcal{KR}
\to
\A^{2}_{u,v};
(x,y,z,t)\mapsto (u,v).
\end{equation}
The fiber $\pi^{-1}(\alpha,\beta)\subset\mathcal{KR}$ is described by the equations $\alpha=xy$, $\beta=yt^{2}$, and
\begin{equation}
\label{equation:KRfiber1}
x+x^{2}y+z^{3}+t^{2}=0.
\end{equation}
Multiplying \eqref{equation:KRfiber1} by $y$ yields the equation
\begin{equation}
\label{equation:KRfiber2}
\alpha+\alpha^{2}+yz^{3}+\beta=0.
\end{equation}
Using these equations one checks that there is exactly one closed orbit in each fiber.
Thus the set of closed orbits is parameterized by the normal scheme $\A^{2}_{u,v}$ and we conclude the corresponding GIT-quotient is indeed \eqref{equation:KRGITquotient}.
\vspace{0.1in}

Makar-Limanov succeeded in showing that $\mathcal{KR}$ is non-isomorphic to ${\mathbb A}^3$ \cite{zbMATH00997597} by calculating his eponymous invariant.
We recall that the Makar-Limanov invariant of an affine algebraic variety $X$ is the subring $\ML(X)$ of $\Gamma(X,\mathscr{O}_{X})$ comprised of
regular functions that are invariant under all $\GG_{a}$-actions on $X$.
Using the bijection between $\GG_{a}$-actions on $X$ and locally nilpotent derivations $\partial$ on the $k$-algebra $\Gamma(X,\mathscr{O}_{X})$ one finds that
\[
\ML(X)=\bigcap_{\partial}\ker(\partial).
\]
Clearly we have $\ML(\AA^{3})=k$ and similarly for all affine spaces, while extensive calculations reveal that $\ML(\mathcal{KR})=k[x]$.
That is,
$\mathcal{KR}$ admits in a sense fewer $\GG_{a}$-actions than $\AA^{3}$.
Here we observe the inclusion $\ML(\mathcal{KR})\subset k[x]$:
the locally nilpotent derivations $x^{2}\partial_{z}-2z\partial_{y}$ and $x^{2}\partial_{t}-3t^{2}\partial_{y}$ of $k[x,y,z,t]$ induce locally nilpotent derivations on the coordinate ring $k[\mathcal{KR}]$.
One easy checks that their kernels intersect in $k[x]$.
The interesting part of Makar-Limanov's calculation is to show that $\partial(x)=0$ for every locally nilpotent derivation of $k[\mathcal{KR}]$.
Alternatively one can use Kaliman's result in \cite{MR1895930} saying that if the general fibers of a regular function ${\mathbb A}^3\to{\mathbb A}^1$ are isomorphic to ${\mathbb A}^2$
then all its fibers are isomorphic to ${\mathbb A}^2$.
All the closed fibers of the projection map $\mathcal{KR}\to{\mathbb A}^1_x$ are isomorphic to ${\mathbb A}^2$ except for over the origin,
which yields a copy of the cylinder on the cuspidal curve $\{z^{3}+t^{2}=0\}$.
\vspace{0.1in}

Dubouloz \cite{MR2534798} showed that the Makar-Limanov invariant cannot distinguish between the cylinder $\mathcal{KR}\times\aone$ on the Russell cubic and the affine space ${\mathbb A}^4$.
Furthermore,  M.P. Murthy showed that all vector bundles on $\mathcal{KR}$ are trivial \cite{Murthy} (it is also known that the Chow groups of $\mathcal{KR}$ are trivial).
However, the $\gm{}$-action on $\mathcal{KR}$ has an isolated fixed point.
If $\mathcal{KR}$ is not stably isomorphic to ${\mathbb A}^3$, is there an $\aone$-homotopic obstruction to stable isomorphism?

\begin{question}
Is the Russell cubic $\mathcal{KR}$ $\aone$-contractible?
\end{question}

This question, which has recently been solved in the affirmative, has guided much of the research in the area.  First, one might try to compute the $\aone$-homotopy groups; for this even to be sensible, we should make sure that the first obstruction to $\aone$-contractibility vanishes.
For a generalization of the following observation we refer the reader to \cite{DPO}.

\begin{prop}[B. Antieau (unpublished)]
The Russell cubic $\mathcal{KR}$ is $\aone$-chain connected.
\end{prop}

\begin{entry}[Approach 1]
Can one detect nontriviality of any of the higher $\aone$-homotopy groups of $\mathcal{KR}$?  One approach to this problem is to think ``naively" of, e.g., the $\aone$-fundamental group.  Think of chains of maps from $\aone$ that start and end at a fixed point up to ``naive" homotopy equivalence (this naturally forms a monoid rather than a group).  The resulting object maps to the actual $\aone$-fundamental group, but what can one say about its image?
\end{entry}

\begin{entry}[Approach 2]
Since to disprove $\aone$-contractiblity, we only need one cohomology theory that is $\aone$-representable that detects nontriviality, it is useful to look at invariants that are not as ``universal" as $\aone$-homotopy groups.  For another approach, using group actions, let us mention that J. Bell showed that rational $\gm{}$-equivariant $K_0$ of $\mathcal{KR}$ is actually nontrivial \cite{Bell}.  Unfortunately, his computations together with the Atiyah-Segal completion theorem in equivariant algebraic K-theory also show that the ``Borel style" equivariant $K_0$ is isomorphic to the Borel style equivariant $K_0$ of a point \cite{AtiyahSegal}.  Nevertheless, $\aone$-homotopy theory gives a wealth of new cohomology theories with which to study the Russell cubic.  For example, it would be interesting to know if one of the more ``refined" Borel style equivariant theories is refined enough to detect failure of $\aone$-contractibility.
\vspace{0.1in}

If $\mu_n \subset \gm{}$ is a sufficiently ``large" subgroup then the $\mu_n$-equivariant $K_0$ of $\mathcal{KR}$ is also nontrivial.  Moreover, the fixed-point loci for the $\mu_n$-actions are all affine spaces (indeed, if $n$ is prime, then the only nontrivial subgroup is the trivial subgroup which has the total space as fixed point locus).  Thus, for many purposes, one might simply look at $\mu_n$-equivariant geometry.
\vspace{0.1in}

In equivariant topology, a map is a ``fine" equivariant weak equivalence if it induces a weak equivalence on fixed point loci for all subgroups.  Transplanting this to $\aone$-homotopy theory: if we knew that equivariant algebraic K-theory was representable on an appropriate equivariant $\aone$-homotopy category, such a category has been constructed for finite groups by Voevodsky \cite{Deligne},
and we knew enough about the weak equivalences in the theory, then Bell's result might formally imply that $X$ is not $\aone$-contractible.
\end{entry}

\subsubsection*{The Russell cubic and equivariant K-theory}
\label{subsection:a1cotkrt}
For representability of equivariant algebraic K-theory let us work relative to a regular Noetherian commutative unital ring $k$ of finite Krull dimension.  We assume that $G \to k$ is a finite constant group scheme (or more generally that it satisfies the resolution property: every coherent $G$-module on $X$ in $\Sch^G_k$ is the equivariant quotient of a $G$-vector bundle).  Under these assumptions, Nisnevich descent for equivariant algebraic $K$-theory of smooth schemes over $k$ was established in \cite{MR3431674}.  The fact that equivariant algebraic $K$-theory satisfies equivariant Nisnevich descent for smooth schemes implies that it is representable in the equivariant motivic homotopy category.
\vspace{0.1in}

Let $X$ be a $G$-scheme over $k$.
Write $\mcal{P}^{G}(X)$ for the exact category of $G$-vector bundles.
Then the equivariant algebraic $K$-groups are the homotopy groups $K_{i}^{G}(X) := \pi_{i}\mcal{K}(\mcal{P}^{G}(X))$ of the associated $K$-theory space, defined by Waldhausen's $S_{\bullet}$-construction.
We obtain a presheaf of simplicial sets ${\mathcal K}^G$ on $\Sch^G_k$ such that $\pi_{i}{\mathcal K}^{G}(X) = K_{i}^{G}(X)$ for all $X$ by applying a rectification procedure to the pseudo-functor
$X\mapsto {\mathcal P}^{G}(X)$.
With the same hypothesis as above, there is a natural isomorphism
\[
K^G_i(X)
\iso
[S^i \wedge X_+,{\mathcal K}^{G}]_{\mathscr{H}^{G}_{\bullet}(k)}
\]
for any $X$ in $\Sm_{k}^{G}$ and the pointed $G$-equivariant motivic homotopy category $\mathscr{H}^{G}_{\bullet}(k)$ of $k$.
This is the desired representability result for equivariant algebraic $K$-theory mentioned above.
\vspace{0.1in}

An explicit computation of the $\mu_p$-equivariant Grothendieck groups of Koras--Russell threefolds was carried out in \cite{MR3549169}.
For concreteness we specialize to the case $k=\CC$.
When $X$ is a complex variety with an action of an algebraic group $G$,
we let $R(G)\simeq K^G_0(k)$ denote the representation ring of $G$.
If $H \subseteq G$ is a closed subgroup,
we note there is a restriction map $K^G_0(X) \to K^{H}_0(X)$.
Let $X$ be a smooth affine variety with $\CC^{\times}$-action and let $n>0$ be an integer.
There is a natural ring isomorphism
\begin{equation}
\label{equation:K0CCton}
\phi\colon K^{\CC^{\times}}_0(X)
{\underset{R(\CC^{\times})}\otimes} R(\mu_n)
\xrightarrow{\cong}
K^{\mu_n}_0(X).
\end{equation}

An algebraic $\CC^{\times}$-action on a smooth complex affine variety is called hyperbolic if it has a unique fixed point
and the weights of the induced linear action on the tangent space at this fixed point are all non-zero and their product is negative.
Recall from \cite{MR1487230} that a {\sl Koras--Russell threefold} $X$ is a smooth hypersurface in $\A^4_\CC$ which is
\begin{enumerate}[noitemsep,topsep=1pt]
\item
topologically contractible,
\item
has a hyperbolic $\CC^{\times}$-action, and
\item
the quotient $X//{\CC^{\times}}$ is isomorphic to the quotient of the linear $\CC^{\times}$-action on the tangent space at the fixed point (in the sense of GIT).
\end{enumerate}
It is shown in \cite[Theorem~4.1]{MR1487230} that the coordinate ring of a threefold $X$ satisfying (1)-(3) has the form
\begin{equation}
\label{eqn:KRCR}
\CC[X]
=
\frac{\CC[x,y, z,t]}{t^{\alpha_2} - G(x, y^{\alpha_1}, z^{\alpha_3})}.
\end{equation}
Here $\alpha_1, \alpha_2, \alpha_3$ are pairwise coprime positive integers.
We let $r$ denote the $x$-degree of the polynomial $G(x, y^{\alpha_1},0)$ and set $\epsilon_X = (r-1)(\alpha_2-1)(\alpha_3-1)$.
A Koras--Russell threefold $X$ is said to be {\sl nontrivial} if $\epsilon_X \neq 0$.
\vspace{0.1in}

Bell \cite{Bell} showed that the $\CC^{\times}$-equivariant Grothendieck group of $X$ is of the form
\begin{equation}
\label{eqn:Bell-*}
\begin{array}{lll}
K^{\CC^{\times}}_0(X) & = & {R(\CC^{\times})} \oplus
\left(\frac{R(\CC^{\times})}{(f(t))}\right)^{\rho-1} \\
& = & \Z[t, t^{-1}] \oplus \left(\frac{\Z[t, t^{-1}]}{(f(t))}\right)^{\rho-1} \\
& = & \Z[t, t^{-1}] \oplus \Z^{(\alpha_2 -1)(\alpha_3-1)},
\end{array}
\end{equation}
where
\begin{equation}
\label{eqn:Bell-*-1}
f(t) = \frac{(1-t^{\alpha_2 \alpha_3})(1-t)}{(1-t^{\alpha_2})(1-t^{\alpha_3})}
\end{equation}
is a polynomial of degree $(\alpha_2-1)(\alpha_3-1)$
and $\rho\ge 2$ is the number of irreducible factors of
$G(x,y^{\alpha_1},0)\in \CC[x,y]$.
In particular, $K^{\CC^{\times}}_0(X)$ is nontrivial.
A combination of \eqref{equation:K0CCton} - \eqref{eqn:Bell-*-1} together with explicit calculations reveal that the $\mu_p$-equivariant Grothendieck group of $X$ is trivial for almost all primes $p$.
This implies that the suggested approach to showing non-$\A^1$-contractibility of a Koras--Russell threefold via $\mu_p$-equivariant Grothendieck groups cannot work.

\begin{thm}
\label{thm:KR-trivial}
Let $p$ be a prime and let $n \ge 1$ be an integer.
Let $\mu_{p^n}$ act on a Koras--Russell threefold $X$ via the inclusion $\mu_{p^n} \subset\CC^{\times}$.
Then the following hold.
\begin{enumerate}[noitemsep,topsep=1pt]
\item
The structure map $X \to \Spec(\CC)$ induces an isomorphism $R(\mu_{p^n}) \oplus F_{p^n} \xrightarrow{\cong} K^{\mu_{p^n}}_0(X)$.
\item
$F_{p^n}$ is a finite abelian group which is nontrivial if and only if $X$ is nontrivial and $p|\alpha_2\alpha_3$.
\end{enumerate}
\end{thm}

If the integers $p$ and $\alpha_2\alpha_3$ are coprime it follows that every $\mu_{p^n}$-equivariant vector bundle on $X$ is stably trivial,
i.e.,
for any $\mu_{p^n}$-equivariant vector bundle $E$ on $X$,
there exist $\mu_{p^n}$-representations $F_1$ and $F_2$ such that $E \oplus F_1\simeq F_2$.
\vspace{0.1in}

\subsubsection*{Higher Chow groups and stable $\aone$-contractibility}
Another natural idea is to study the higher Chow groups of Koras--Russell threefolds.
Showing triviality of the said groups goes a long way in concluding $\aone$-contractibility of $\mathcal{KR}$.

\begin{prop}\label{prop:Contr-1}
Let $X$ be a Koras--Russell threefold of the first kind with coordinate ring
\begin{equation*}\label{eqn:Contr-1-1}
\CC[X] = \frac{\CC[x,y,z,t]}{\left(ax + x^my + z^{\alpha_2} + t^{\alpha_3}
\right)},
\end{equation*}
where $m>1$ is an integer, $a \in {\CC}^{*}$, and $\alpha_2, \alpha_3 \ge 2$ are coprime.
For $Y$ any smooth complex affine variety,
the pullback map $CH^{\ast}(Y) \to CH^{\ast}(X \times Y)$ is an isomorphism.
\end{prop}
A related calculation shows the same conclusion holds for Koras--Russell threefolds of the second kind.
As for the proof of Proposition \ref{prop:Contr-1} a key input is the observation that the ring homomorphism
$$
\CC
\to
{\CC[u, v]}/{(u^a + v^b)}
$$
induces an isomorphism on higher Chow groups for coprime integers $a, b\geq 2$.
\vspace{0.1in}

Combined with the isomorphism between higher Chow groups and motivic cohomology,
as shown by Voevodsky \cite[Corollary~2]{MR1883180},
we obtain the following.
\begin{thm}
\label{thm:MC-Vanish}
Let $X$ be a Koras--Russell threefold of the first or second kind,
and let $Y$ be any smooth complex affine variety.
Then the pullback map $H^{*, *}(Y,\Z) \to H^{*, *}(X \times Y,\Z)$ induced by the projection $X \times Y \to Y$
is an isomorphism of (bigraded) integral motivic cohomology rings.
\end{thm}

A consequence of Theorem \ref{thm:MC-Vanish} is that every vector bundle on $X$ is trivial.
This was originally shown by Murthy \cite[Corollary~3.8]{Murthy} by a completely different method.
\vspace{0.1in}

Theorem \ref{thm:MC-Vanish} is a key input in the approach to $\aone$-contractibility of Koras--Russell threefolds in \cite{MR3549169}.
To proceed it is convenient to employ some techniques from stable motivic homotopy theory.
In particular,
the slice filtration on the stable motivic homotopy category $\mathscr{SH}(\CC)$ will be put to good use.
Recall the objects of $\mathscr{SH}(\CC)$ are sequences of pointed motivic spaces related by structure maps with respect to $(\mathbb{P}^{1},\infty)$.
We note that $\mathscr{SH}(\CC)$ is a triangulated category with shift functor $E\mapsto E[1]$ given by smashing with the topological circle.
Denote by $\Sigma^\infty_{{\mathbb P}^{1}} (X,x)\in\mathscr{SH}(\CC)$ the $(\mathbb{P}^{1},\infty)$-suspension spectrum of $X\in\Sm_\CC$ and a rational point $x\in X(\CC)$.
For fixed $F\in\mathscr{SH}(\CC)$,
we say that $E\in\mathscr{SH}(\CC)$ is
\begin{enumerate}[noitemsep,topsep=1pt]
\item \emph{$F$-acyclic} if $E\wedge F\simeq \ast$;
\item \emph{$F$-local} if $\Hom_{\mathscr{SH}(\CC)}(D,E)=0$ for every $F$-acyclic spectrum $D$.
\end{enumerate}
It is clear that the $F$-local spectra form a colocalizing subcategory of $\mathscr{SH}(\CC)$.
Note that if $F$ is a ring spectrum,
then any $F$-module $E$ is $F$-local (every map $D\to E$ factors through $D\wedge F$ and hence it is trivial if $D$ is $F$-acyclic).
\vspace{0.1in}

Let $\MZ\in\mathscr{SH}(\CC)$ denote the motivic ring spectrum that represents motivic cohomology,
i.e.,
for every $X\in\Sm_\CC$ and integers $n,i\in\Z$ there is an isomorphism
\begin{equation}
\label{eqn:motivic-coh-SH}
H^{n,i}(X,\Z)
\simeq
\Hom_{\mathscr{SH}(\CC)}(\Sigma^\infty_{{\mathbb P}^{1}} X_+,\MZ(i)[n]).
\end{equation}
Here,
for $E\in\mathscr{SH}(\CC)$,
the Tate twist $E(1)$ is defined by $E(1)=E\wedge\Sigma^\infty_{{\mathbb P}^{1}}(\mathbb{G}_{m},1)[-1]$.
The Betti realization of $\MZ$ identifies with the classical Eilenberg-Mac Lane spectrum $\HZ$ representing singular (co)homology of topological spaces.

\begin{lem}
\label{lem:HZ-local}
For every $X\in\Sm_\CC$ and closed point $x\in X$ the suspension $\Sigma^\infty_{{\mathbb P}^{1}} (X,x)\in\mathscr{SH}(\CC)$ is $\MZ$-local.
\end{lem}
\begin{proof}
Resolution of singularities allows one to show that
$\Sigma^\infty_{{\mathbb P}^{1}}(X,x)$ is in the smallest thick subcategory of $\mathscr{SH}(\CC)$ containing $\Sigma^\infty_{{\mathbb P}^{1}} Y_+$ for any smooth projective variety $Y$.
It suffices now to show that $\Sigma^\infty_{{\mathbb P}^{1}} Y_+$ is $\MZ$-local for such $Y$.
Voevodsky's slice filtration for any $E\in\mathscr{SH}(\CC)$ is a tower of spectra
\[
\dotsb \to f_{q+1}E \to f_qE \to f_{q-1}E \to \dotsb \to E,
\quad q\in\Z.
\]
Here the $q$th slice $s_qE$ of $E$ is defined by the distinguished triangle
\[
f_{q+1}E \to f_qE \to s_qE \to f_{q+1}E[1].
\]
Levine \cite{Levine} has shown that the slice filtration of $\Sigma^\infty_{{\mathbb P}^{1}}Y_+$  for $Y$ any smooth projective variety is complete in the sense that
\[
\holim_{q\to\infty} f_q(\Sigma^\infty_{{\mathbb P}^{1}} Y_+)
\simeq
\ast.
\]
Equivalently,
if we define $c_qE$ by the distinguished triangle $f_qE \to E \to c_qE \to f_qE[1]$,
then
\[
\Sigma^\infty_{{\mathbb P}^{1}} Y_+\simeq \holim_{q\to \infty} c_q(\Sigma^\infty_{{\mathbb P}^{1}} Y_+).
\]
Since the subcategory of $\MZ$-local spectra is colocalizing,
it now suffices to prove that $c_q(\Sigma^\infty_{{\mathbb P}^{1}} Y_+)$ is $\MZ$-local for every $q\in \Z$.
By definition of the slice filtration, we have $c_q(\Sigma^\infty_{{\mathbb P}^{1}} Y_+)\simeq\ast$ for $q\le 0$.
Using the distinguished triangles
\[
s_qE \to c_qE \to c_{q-1}E \to s_qE[1]
\]
and induction on $q$,
we are reduced to proving the slices $s_q(\Sigma^\infty_{{\mathbb P}^{1}} Y_+)$ are $\MZ$-local.
In fact,
all slices in $\mathscr{SH}(\CC)$ are $\MZ$-local:
any slice $s_qE$ is a module over the zeroth slice $s_0(\sphere)$ of the sphere spectrum,
and hence it is $s_0(\sphere)\simeq\MZ$-local.
\end{proof}

\begin{thm}
\label{thm:Main-KR}
Let $X$ be a Koras--Russell threefold of the first or second kind.
Then there exists an integer $n \ge 0$ such that the suspension $\Sigma^n_{{\mathbb P}^{1}}(X, 0)$ is $\A^1$-contractible.
\end{thm}
\begin{proof}
We first reformulate Theorem \ref{thm:MC-Vanish} as an equivalence in $\mathscr{SH}(\CC)$,
using its structure of a closed symmetric monoidal category, see \S\ref{ss:stablestory}.
The structure map $X\to \Spec(\CC)$ induces a morphism in $\mathscr{SH}(\CC)$
\begin{equation}
\label{equation:MZmap}
\MZ
\simeq
\sHom(\Sigma^\infty_{{\mathbb P}^{1}} \Spec(\CC)_+,\MZ)
\to
\sHom(\Sigma^\infty_{{\mathbb P}^{1}}X_+,\MZ).
\end{equation}
In view of \eqref{eqn:motivic-coh-SH},
Theorem \ref{thm:MC-Vanish} asserts that for every smooth complex affine variety $Y$ and $n,i\in\Z$,
there is an induced isomorphism
\[
\Hom_{\mathscr{SH}(\CC)}(\Sigma^\infty_{{\mathbb P}^{1}}Y_+(i)[n],\MZ)
\to
\Hom_{\mathscr{SH}(\CC)}(\Sigma^\infty_{{\mathbb P}^{1}}Y_+(i)[n], \sHom(\Sigma^\infty_{{\mathbb P}^{1}} X_+,\MZ)).
\]
The objects $\Sigma^\infty_{{\mathbb P}^{1}} Y_+(i)[n]$ form a family of generators of $\mathscr{SH}(\CC)$,
because every smooth variety admits an open covering by smooth affine varieties.
Thus \eqref{equation:MZmap} and its retraction $\sHom(\Sigma^\infty_{{\mathbb P}^{1}} X_+,\MZ)\to\MZ$ induced by the base point $0\in X$ are isomorphisms.
From the distinguished triangle
\[
\MZ[-1]
\to
\sHom(\Sigma^\infty_{{\mathbb P}^{1}} (X,0),\MZ)
\to
\sHom(\Sigma^\infty_{{\mathbb P}^{1}} X_+,\MZ)
\to
\MZ,
\]
we deduce that $\sHom(\Sigma^\infty_{{\mathbb P}^{1}} (X,0),\MZ)\simeq\ast$.
By \cite[Theorems 1.4 and 2.2]{Riou} or \cite[Theorem 52]{RO:mz},
$\Sigma^\infty_{{\mathbb P}^{1}} (X,0)$ is strongly dualizable in $\mathscr{SH}(\CC)$,
so that
\[
\sHom(\Sigma^\infty_{{\mathbb P}^{1}} (X,0),\MZ)
\simeq
\sHom(\Sigma^\infty_{{\mathbb P}^{1}} (X,0),\1)\wedge\MZ.
\]
Thus $\sHom(\Sigma^\infty_{{\mathbb P}^{1}} (X,0),\1)$ is $\MZ$-acyclic,
and for any $E\in\mathscr{SH}(\CC)$ we obtain
\[
\Hom_{\mathscr{SH}(\CC)}(E,\Sigma^\infty_{{\mathbb P}^{1}}(X,0)\wedge\MZ)
\simeq
\Hom_{\mathscr{SH}(\CC)}(E\wedge \sHom(\Sigma^\infty_{{\mathbb P}^{1}} (X,0),\1),\MZ)
\simeq
\ast,
\]
since $E\wedge \sHom(\Sigma^\infty_{{\mathbb P}^{1}}(X,0),\1)$ is $\MZ$-acyclic and $\MZ$ is $\MZ$-local (being an $\MZ$-module).
By the Yoneda lemma, this implies $\Sigma^\infty_{{\mathbb P}^{1}}(X,0)$ is $\MZ$-acyclic,
i.e.,
\[
\Sigma^\infty_{{\mathbb P}^{1}}(X,0)\wedge\MZ\simeq\ast,
\]
On the other hand,
by Lemma \ref{lem:HZ-local},
$\Sigma^\infty_{{\mathbb P}^{1}}(X,0)$ is $\MZ$-local.
It follows that every endomorphism of $\Sigma^\infty_{{\mathbb P}^{1}}(X,0)$ is trivial,
and hence $\Sigma^\infty_{{\mathbb P}^{1}}(X,0)\simeq\ast$.
Owing to Lemma \ref{lem:compactness} this completes the proof.
\end{proof}

\begin{rem}
In general, $\AA^{1}$-weak equivalences do not desuspend.  To illustrate this, for simplicity, take $k = \cplx$, and let $\{p_1,\ldots,p_n\}$ and $\{q_1,\ldots,q_n\}$ be two collections of complex points in ${\mathbb A}^1$, but the argument works much more generally.  If $n \geq 1$, then $\aone \setminus \{ p_1,\ldots,p_n \}$ is an $\aone$-rigid variety in the sense of Definition~\ref{defn:A1rigid}.  In particular, the $\aone$-weak equivalences $\aone \setminus \{ p_1,\ldots,p_n \} \cong \aone \setminus \{ q_1,\ldots,q_n \}$ are simply isomorphisms of varieties by appeal to Corollary~\ref{cor:aonerigidfullyfaithful}.  However, any isomorphism of varieties of this form is induced by an automorphism of the affine line.  The automorphism group of the affine line acts $2$-transitively, but not $n$-transitively for any $n \geq 3$.  In fact, as soon as $n \geq 3$, there is a moduli space of configurations of dimension $n-2$.
\vspace{0.1in}

The variety $\Sigma_{\pone} \aone \setminus \{ p_1,\ldots,p_n \}$ is $\aone$-weakly equivalent to ${\mathbb A}^2 \setminus \{ x_1,\ldots, x_n\}$ for any collection of $n$ $\cplx$-points of ${\mathbb A}^2$.  Indeed, it is not hard to show that the algebraic  automorphism group of ${\mathbb A}^2$ acts $n$-fold transitively on ${\mathbb A}^2(\cplx)$ for any $n \geq 1$, in contrast to the situation for the affine line.  Thus, if we choose coordinates $x,y$ on ${\mathbb A}^2$, we may move the points $x_1,\ldots,x_n$ to lie on the $x$-axis and then cover ${\mathbb A}^2 \setminus \{ x_1,\ldots, x_n\}$ by $\aone \setminus \{p_1,\ldots,p_n\} \times \aone$ and $\aone \times \gm{}$ with intersection $\aone \setminus \{p_1,\ldots,p_n\} \times \gm{}$.  The required weak equivalence then follows by the same homotopy colimit argument used to prove that ${\mathbb A}^2 \setminus 0 \cong \pone \wedge \gm{}$.  By increasing the number of points, we see that there are arbitrary dimensional moduli of smooth varieties that become $\aone$-weakly equivalent after a single $\pone$-suspension.
\end{rem}


The discussion above also has implications for topologically contractible smooth complex varieties.  The motivic conservativity conjecture (see, e.g., \cite[Proposition 3.4]{Huber} or \cite[Conjecture 2.1`]{AyoubConjectures}) implies the rational Voevodsky motive of a topologically contractible variety is that of a point.  In dimension $2$, the {\em integral} Voevodsky motive of topologically contractible surface is trivial by the results of \cite{AAcyclic}; it is also observed there that triviality holds for a number of higher dimensional examples.  It thus is not inconsistent with known examples to suggest that the integral Voevodsky motive of a topologically contractible smooth complex variety is always that of a point.  In conjunction with the proof of Theorem~\ref{thm:Main-KR}, the following seems reasonable.

\begin{conj}
\label{conj:stablecontractibility}
If $X$ is a topologically contractible smooth complex affine variety, then there exists an integer $n \geq 0$ such that $\Sigma^n_{\pone}(X,x)$ is $\aone$-contractible; in fact, $n = 2$ should suffice.
\end{conj}

\begin{rem}
The stronger assertion here is obtained by combining the weaker assertion and a conjecture about conservativity of $\gm{}$-stabilization \cite{BachmannYakerson}.  Conjecture~\ref{conj:stablecontractibility} is reminiscent of an open version of the Cannon--Edwards double suspension theorem \cite{Cannon,Edwards}.  Conjecture~\ref{conj:stablecontractibility} in conjunction with $\aone$-representability of Chow groups (see Theorem ~\ref{thm:chowrepresentable}) implies that if $X$ is any topologically contractible smooth complex variety, then $CH^i(X) = 0$ for every $i > 0$.  In conjunction with Theorem~\ref{thm:vbonthreefolds}, Conjecture~\ref{conj:stablecontractibility} thus implies that the generalized Serre question has a positive answer in dimension $3$.
\end{rem}

\subsection{$\aone$-contractibility of the Koras--Russell threefold}
\label{ss:krthreefoldI}
In what follows we will outline how Theorem \ref{thm:Main-KR} can be used to show the Russell threefold $\mathcal{KR}$ is $\A^1$-contractible over any base field of characteristic zero.
This was carried out by Dubouloz and Fasel in \cite{MRDF}.
\vspace{0.1in}

Observe that $\mathcal{KR}$ contains both the affine line $\aone_{y}$ and the affine plane $\AA^{2}_{z,t}$
intersecting transversally in the origin.
The idea is now to show that the inclusion $\AA^{2}_{z,t}\to\mathcal{KR}$ is an $\A^1$-equivalence.
There is a naturally induced commutative diagram of homotopy cofiber sequences
\begin{equation}
\label{equation:KRdiagram}
\xymatrix{
\AA^{2}_{z,t}\setminus\{(0,0)\} \ar[r]\ar[d]_{i} & \AA^{2}_{z,t} \ar[r]\ar[d] & \PP^{1}_{z}\wedge\PP^{1}_{t} \ar[d] \\
\mathcal{KR}\setminus \AA^{1}_{y} \ar[r] & \mathcal{KR} \ar[r] & \mathcal{KR}/(\mathcal{KR}\setminus \AA^{1}_{y}).
}
\end{equation}

Here the rightmost vertical map is an $\A^1$-equivalence induced by the inclusion $\{0\}\subset\AA^{1}_{y}$:
This follows since the normal bundle of $\AA^{1}_{y}$ in $\mathcal{KR}$ is trivial,
so that by homotopy purity \ref{thm:homotopypurity} we obtain $\AA^{1}$-equivalences
\[
\mathcal{KR}/(\mathcal{KR}\setminus \AA^{1}_{y})
\sim_{\AA^{1}}
(\AA^{1}_{y})_{+}\wedge(\mathbb{P}^{1})^{\wedge2}
\sim_{\AA^{1}}
(\mathbb{P}^{1})^{\wedge2}.
\]

By a general result we are reduced to showing that $i$ is an $\AA^{1}$-equivalence.  The perhaps most technical argument in the proof consists of showing that $\mathcal{KR}\setminus \AA^{1}_{y}$ is $\AA^{1}$-weak equivalent to the punctured affine space $\AA^{2}_{z,t}\setminus\{(0,0)\}$.  We discuss this part later in this section.  As a consequence the leftmost vertical map in \eqref{equation:KRdiagram} is a self-map of $\AA^{2}_{z,t}\setminus\{(0,0)\}$ up to $\AA^{1}$-equivalence.  As a consequence, its $\aone$-homotopy class is determined by its motivic Brouwer degree (see Theorem~\ref{thm:degree}).  The computation of this degree may be turned into a understanding a certain map in sheaf cohomology with coefficients in Milnor-Witt $K$-theory.

\begin{lem}
\label{lem:Brouwer}
Let $f:\A^n\setminus \{0\}\to \A^n\setminus \{0\}$ be a morphism in $\ho{k}$.
Then $f$ is an isomorphism if and only if
\[
f^*:H^{n-1}(\A^n\setminus \{0\},\KMW_n)\to H^{n-1}(\A^n\setminus \{0\},\KMW_n)
\]
is an isomorphism.
\end{lem}

According to Lemma \ref{lem:Brouwer} we are reduced to showing that
\[
i^*:H^{1}(\mathcal{KR}\setminus \AA^{1}_{y},\KMW_2)\to H^1(\A^2_{z,t}\setminus \{0\},\KMW_2)
\]
is an isomorphism.
In effect, we consider the commutative diagram
\[
\xymatrix{H^1(\mathcal{KR}\setminus \AA^{1}_{y},\KMW_2)\ar[r]^-\partial\ar[d]_-{i^*} & H^2((\pone)^{\wedge 2},\KMW_2)\ar[r]\ar@{=}[d]
& H^2(\mathcal{KR},\KMW_2)\ar[d]\ar[r] & H^2(\mathcal{KR}\setminus \AA^{1}_{y},\KMW_2)\ar[d] \\
H^1(\A^2_{z,t}\setminus \{0\},\KMW_2)\ar[r]^-{\partial^\prime} & H^2((\pone)^{\wedge 2},\KMW_2)\ar[r] & H^2(\A^2_{z,t},\KMW_2)\ar[r]
& H^2(\A^2_{z,t}\setminus \{0\},\KMW_2)}
\]
obtained from \eqref{equation:KRdiagram}.
One checks readily that $\partial^\prime$ is an isomorphism,
so that $i^*$ is an isomorphism if and only if $\partial$ is an isomorphism.
Since $\partial$ is a $\KMW_0(k)$-linear map between free $\KMW_0(k)$-modules of rank one,
it suffices to show the following assertion.

\begin{prop}
\label{prop:connectivesurjective}
The connecting homomorphism $\partial:H^{1}(\mathcal{KR}\setminus \AA^{1}_{y},\KMW_2)\to H^2((\pone)^{\wedge 2},\KMW_2)$ is surjective.
\end{prop}
\begin{proof}
As the first row in the above diagram is exact, it is sufficient to prove that $H^2(\mathcal{KR},\KMW_2)=0$.
Fasel's projective bundle theorem \cite{zbMATH06170873} implies
\[
H^i(\mathcal{KR},\KMW_j)
\cong
H^{i+n}(\mathcal{KR}_+\wedge (\pone)^{\wedge n},\KMW_{j+n})
\]
for all $i,n\in \N$ and $j\in \Z$.
Theorem \ref{thm:Main-KR} holds more generally over  fields of characteristic zero and shows that $\mathcal{KR}\wedge (\pone)^{\wedge n}=*$ for $n\gg 0$.
Combined with the homotopy cofiber sequence
\[
(\pone)^{\wedge n}\to \mathcal{KR}_+\wedge (\pone)^{\wedge n} \to \mathcal{KR}\wedge (\pone)^{\wedge n}
\]
we find $H^i(\mathcal{KR},\KMW_j)=H^{i+n}((\pone)^{\wedge n},\KMW_{j+n})$ for $i\geq 1$,
and the latter group is trivial.
\end{proof}

Dubouloz and Fasel also give an alternate proof of Proposition \ref{prop:connectivesurjective} by means of explicit symbol calculations \cite{MRDF}.  Moreover, the proof of $\aone$-contractibility for $\mathcal{KR}$ works more generally for Koras--Russell threefolds of the first kind.  It is unclear whether a similar proof works for Koras--Russell threefolds of the second kind.  The main issue at stake for such a threefold $X$ is whether there exists an $\aone$-weak equivalence between the complement $X\setminus\A^1_{y}$ and some punctured affine plane.  On the other hand, in light of Theorem~\ref{thmintro:uncountableexistence} from the introduction the proof is robust enough to provide many new examples of affine $\aone$-contractible varieties of dimension $3$.

\begin{thm}[{\cite[Corollary 1.3]{MRDF}}]
\label{thm:3dimlfamilies}
Assume $k$ is a field.  For every integer $m \geq 2$, there exists a smooth affine morphism $\pi: X \to {\mathbb A}^{m-2}$ of relative dimension $3$ whose fibers are all $\aone$-contractible.  Furthermore, fibers of $\pi$ over $k$-points are pairwise non-isomorphic and stably isomorphic.
\end{thm}

\begin{question}
Looking further forward: can one characterize affine spaces among affine $\aone$-contractible varieties in a motivic version of the Poincar\'e conjecture
(e.g., by defining some notion of $\aone$-fundamental group at infinity)?
\end{question}

\subsection{Koras--Russell fiber bundles}
\label{ss:KRfiberbundles}
The geometry becomes much more pronounced in the proof showing that $\mathcal{KR}\setminus \AA^{1}_{y}$ is $\AA^{1}$-weak equivalent to the punctured affine space $\AA^{2}_{z,t}\setminus\{(0,0)\}$;
the salient geometric features of Koras--Russell threefolds of the first kind have been further developed into the context of Koras--Russell fiber bundles introduced in \cite{DPO}.
In the following we assume $k$ is an algebraically closed field of characteristic zero.

\begin{defn}
\label{def:KRBundle}
Suppose $s(x)\in k[x]$ has positive degree and let $R(x,y,t)\in k[x,y,t]$.
Define the closed subscheme $\mathcal{X}(s,R)$ of $\AA^{1}_{x}\times\AA^{3}=\Spec(k[x][y,z,t])$ by the equation
\[
\{s(x)z=R(x,y,t)\}.
\]
We say that the projection map
\begin{equation}
\label{equation:KRfiberbundleprojection}
\rho
:=
\pr_{x}:
\mathcal{X}(s,R)
\to
\AA^{1}_{x}
\end{equation}
defines a {\it Koras--Russell fiber bundle }if
\begin{itemize}[noitemsep,topsep=1pt]
\item[(a)] $\mathcal{X}(s,R)$ is a smooth scheme, and
\item[(b)] For every zero $x_{0}$ of $s(x)$, the zero locus in $\AA^2=\Spec(k[y,t])$ of the polynomial $R(x_0,y,t)$ is an integral rational plane curve with a unique place at infinity
and at most unibranch singularities.
\end{itemize}
\end{defn}
\begin{rem}
\label{remark:KRaffinespace}
One can show that a Koras--Russell fiber bundle is isomorphic to $\AA^{3}$ if and only if for every zero $x_0$ of $s(x)$
the curve $\{R(x_0,y,t)=0\}$ is isomorphic to $\AA^{1}$.
\end{rem}

For concreteness we discuss two classes of examples of Koras--Russell fiber bundles.

\begin{ex}
\label{ex:deformedKR}
Deformed Koras--Russell threefolds of the first kind are defined as
\begin{equation}
\label{equation:firstkindpolynomial}
\mathcal{X}(n,\alpha_{i},p)
:=
\{
x^{n}z=
y^{\alpha_{1}}+t^{\alpha_{2}}+xp(x,y,t)
\},
\end{equation}
where $n,\alpha_{i}\geq 2$ are integers, $\alpha_{1}$ and $\alpha_{2}$ are coprime, and $p(x,y,t)\in k[x,y,t]$ satisfies $p(0,0,0)\in k^{\ast}$.
By \cite{MRDF} it is known that $\mathcal{X}(n,\alpha_{i},p)$ is $\AA^{1}$-contractible when $p(x,y,t)=q(x)\in k[x]$ and $q(0)\in k^{\ast}$.
Note that $\mathcal{X}(n,\alpha_{i},p)$ is smooth according to the Jacobian criterion since $p(0,0,0)\in k^{\ast}$.
Moreover,
the unique singular fiber of the projection map
\begin{equation}
\label{equation:firstkindpolynomialprojection}
\pr_{x}:
\mathcal{X}(n,\alpha_{i},p)
\to
\AA^{1}_{x}
\end{equation}
is a cylinder on the cuspidal curve $\Gamma_{\alpha_{1},\alpha_{2}}:=\{y^{\alpha_{1}}+t^{\alpha_{2}}=0\}\subset\AA^{2}$, which is $\AA^{1}$-contractible
\begin{equation}
\label{equation:firstkindpolynomialinverse}
\pr_{x}^{-1}(0)
=
\Gamma_{\alpha_{1},\alpha_{2}}\times \AA^{1}_{z}
\sim_{\AA^{1}}
\ast.
\end{equation}
Here \eqref{equation:firstkindpolynomialprojection} is a flat $\AA^{2}$-fibration restricting to a trivial $\AA^{2}$-bundle over $\AA^{1}_{x}\setminus\{0\}$ and $\mathcal{X}(n,\alpha_{i},p)$ is factorial.
The $\AA^{1}$-homotopy theory of deformed Koras--Russell threefolds of the first kind \eqref{equation:firstkindpolynomial} is essentially governed by \eqref{equation:firstkindpolynomialprojection}
and \eqref{equation:firstkindpolynomialinverse}.
\end{ex}

\begin{ex}
For $1\leq i\leq m$,
choose distinct linear forms $l_{i}(x)=(x-x_{i})\in k[x]$,
$n_{i},\alpha_{i},\beta_{i}\geq 2$,
where $\alpha_{i}$ and $\beta_{i}$ are coprime,
and $a\in k^{\times}$.
We define $\mathcal{X}_{m}(n_{i},\alpha_{i},\beta_{i},a)$ or simply $\mathcal{X}_{m}$ by the equation
\begin{equation}
\label{equation:mdegeneratefibers}
\mathcal{X}_{m}
=
\mathcal{X}_{m}(n_{i},\alpha_{j},\beta_{j},a)
:=
\big\{\left(\prod_{i=1}^{m} l_{i}(x)^{n_i}\right)z
=
\sum_{i=1}^{m}\left(\left(\prod_{j\neq i}l_{j}(x)\right)(y^{\alpha_{i}}+t^{\beta_{i}})\right)+a\prod_{i=1}^{m}l_{i}(x)\big\}.
\end{equation}
In the case of two degenerate fibers,
\eqref{equation:mdegeneratefibers} takes the form
\begin{equation}
\label{equation:twodegeneratefibers}
\mathcal{X}_{2}
=
\{(x-x_{1})^{n_{1}}(x-x_{2})^{n_{2}}z
=
(x-x_{1})(y^{\alpha_{2}}+t^{\beta_{2}})
+
(x-x_{2})(y^{\alpha_{1}}+t^{\beta_{1}})
+
a(x-x_{1})(x-x_{2})\}.
\end{equation}
For all $m$ the Makar-Limanov invariant of $\mathcal{X}_{m}(n_{i},\alpha_{j},\beta_{j},a)$ equals $k[x]$;
hence it is non-isomorphic to $\AA^{3}$.
Moreover,
the projection map
\begin{equation}
\label{equation:degeneratefirstkindpolynomialprojection}
\pr_{x}:
\mathcal{X}_{m}(n_{i},\alpha_{j},\beta_{j},a)
\to
\AA^{1}_{x}
\end{equation}
defines a trivial $\AA^{2}$-bundle over the punctured affine line $\AA^{1}_{x}\setminus\{x_{1},\dots,x_{m}\}$.
Its fiber over the closed point $x_{i}\in\AA^{1}_{x}$ is isomorphic to the cylinder on the cuspidal curve $\Gamma_{\alpha_{i},\beta_{i}}:\{y^{\alpha_{i}}+t^{\beta_{i}}=0\}$.
By counting closed fibers non-isomorphic to $\AA^{2}$ one concludes that $\mathcal{X}_{m}$ and $\mathcal{X}_{m'}$ are non-isomorphic when $m\neq m'$.
\end{ex}

Now we turn to the geometric properties of deformed Koras--Russell threefolds as in Example \ref{ex:deformedKR}.
In particular,
this will explain the $\AA^{1}$-equivalence between $\mathcal{KR}\setminus \AA^{1}_{y}$ and the punctured affine plane.

There is an induced $\mathbb{G}_a$-action on $\mathcal{X}(n,\alpha_{i},p)$ determined by the locally nilpotent derivation
$$
\partial
=
x^{n}\frac{\partial}{\partial y}+(\alpha_{1}y^{\alpha_{1}-1}+x\frac{\partial}{\partial y}p(x,y,t))\frac{\partial}{\partial z},
$$
on the coordinate ring of $\mathcal{X}(n,\alpha_i,p)$,
with fixed point locus the affine line $\{x=y=t=0\}\cong\mathbb{A}^1_z$.
The geometric quotient $\mathcal{X}(n,\alpha_i,p)\rightarrow \mathcal{X}(n,\alpha_i,p)/\mathbb{G}_a$ yields an $\mathbb{A}^1$-bundle
$\mathcal{X}(n,\alpha_i,p)\setminus \mathbb{A}^1_z\rightarrow\mathfrak{S}(\alpha_1,\alpha_2)$
in the category of algebraic spaces \cite{MR0302647}.
In fact there exists a factorization
\begin{equation}
\label{equation:asfactorization}
\xymatrix{
\mathcal{X}(n,\alpha_{i},p)\setminus \mathbb{A}^{1}_{z}  \ar[dr]_-{\rho} \ar[rr]^-{\pi_{\vert}} & & \mathbb{A}^{2}_{x,t}\setminus\{(0,0)\} \\
& \mathfrak{S}(\alpha_{1},\alpha_{2}), \ar[ur]_-{\delta} &
}
\end{equation}
where $\pi_{\vert}$ is the restriction of $\pi=\text{pr}_{x,t}:\mathcal{X}(n,\alpha_i,p)\rightarrow \mathbb{A}^2_{x,t}$ to $\mathcal{X}(n,\alpha_i,p)\setminus \mathbb{A}^1_z$.
To construct \eqref{equation:asfactorization} we form a cyclic Galois cover of $\mathbb{A}^{2}_{x,t}$ of order $\alpha_{1}$ and hence of $\mathcal{X}(n,\alpha_{i},p)$ by pullback via $\pi$.
The maps arise as geometric quotients for $\mu_{\alpha_{1}}$-equivariant maps by gluing copies of $\mathbb{A}^{2}\setminus\{(0,0)\}$ via a family of cuspidal curves.
Here $\rho$ is an \'etale locally trivial $\mathbb{A}^{1}$-bundle,
we have $\mathfrak{S}(\alpha_{1},\alpha_{2})\cong\mathfrak{S}(\alpha_{1},1)$,
and both of the projection maps for the smooth quasi-affine $4$-fold
\begin{equation}
\label{equation:asfiberproduct}
(\mathcal{X}(n,\alpha_{i},p)\setminus \mathbb{A}^{1}_{z})\times_{\mathfrak{S}(\alpha_{1},\alpha_{2})} (\mathcal{X}(n,\alpha_{1},1,p)\setminus \mathbb{A}^{1}_{z})
\end{equation}
are Zariski  locally trivial $\mathbb{A}^{1}$-bundles,
and hence $\mathbb{A}^{1}$-weak equivalences.
The fiber product in \eqref{equation:asfiberproduct} is formed in algebraic spaces over the punctured affine plane $\mathbb{A}^{2}_{x,t}\setminus\{(0,0)\}$.
Furthermore,
the projection map $\text{pr}_{x}:\mathcal{X}(n,\alpha_{1},1,p)\to\mathbb{A}^{1}_{x}$ is a trivial $\mathbb{A}^{2}$-bundle.
It follows that $\mathcal{X}(n,\alpha_{1},1,p)\cong\mathbb{A}^{3}_{x,y,z}$ and we can finally conclude that there exist $\aone$-weak equivalences
\[
\mathcal{X}(n,\alpha_{i},p)\setminus \mathbb{A}^{1}_{z}
\sim_{\mathbb{A}^{1}}
\mathcal{X}(n,\alpha_{1},1,p)\setminus \mathbb{A}^{1}_{z}
\sim_{\mathbb{A}^{1}}
\mathbb{A}^{2}_{x,t}\setminus\{(0,0)\}.
\]

The currently most general result concerning $\AA^{1}$-contractibility of Koras--Russell fiber bundles was shown in \cite{DPO}.
\begin{thm}
\label{thm:KRBundle-Stable-A1Cont}
Suppose $\rho:\mathcal{X}(s,R)\to\AA^{1}_{x}$ is a Koras--Russell fiber bundle with basepoint the origin.
The $S^{1}$-suspension and hence the $\PP^{1}$-suspension of $\mathcal{X}(s,R)$ are $\AA^{1}$-contractible:
\[
\mathcal{X}(s,R)\wedge S^{1}\sim_{\AA^{1}}
\mathcal{X}(s,R)\wedge\PP^{1}\sim_{\AA^{1}}
\ast.
\]
\end{thm}

\begin{question}
\label{question:KRbundles}
Can one generalize the notion of a Koras--Russell fiber bundle and show Theorem \ref{thm:KRBundle-Stable-A1Cont} over arbitrary fields of characteristic zero?
\end{question}

\begin{rem}
Work on Question \ref{question:KRbundles} is likely to involve base change arguments to an algebraic closure.
\end{rem}

\subsection*{Acknowledgements}
Asok was partially supported by National Science Foundation Awards DMS-1254892 and DMS-1802060.
{\O}stv{\ae}r was partially supported by RCN Frontier Research Group Project no.~250399,
Friedrich Wilhelm Bessel Research Award from the Humboldt Foundation,
and Nelder Visiting Fellowship from Imperial College London.
The authors are grateful to Ben Antieau, Brent Doran, Adrien Dubouloz, Jean Fasel, Marc Hoyois, Amalendu Krishna, Fabien Morel, Sabrina Pauli, and Ben Williams for collaborative efforts and many interesting discussions on the topics surveyed in this paper.

\newpage
{\begin{footnotesize}
\raggedright
\bibliographystyle{alpha}
\bibliography{A1contractibilitySurvey}
\end{footnotesize}}

\Addresses

\end{document}